\documentclass[a4paper,10pt]{article}
\usepackage{amssymb,amsmath,amsthm}
\usepackage{graphicx}

\usepackage{amsfonts}
\usepackage{latexsym}
\usepackage[english]{babel}
\usepackage{hyperref} 

\numberwithin{equation}{section}

\newtheorem{theorem}{Theorem}[section]

\newtheorem{lemma}{Lemma}[section]
\newtheorem{corollary}{Corollary}[section]

\newtheorem{remark}{Remark}[section]
\newtheorem{remarks}{Remark}[section]
\newtheorem{definition}{Definition}[section]
\newcommand{\be}{\begin{equation}}
\newcommand{\ee}{\end{equation}}

\newcommand{\om}{\omega}

\newcommand{\e}{\varepsilon}

\newcommand{\R}{\mathbb R}
\newcommand{\C}{\mathbb C}
\newcommand{\mG}{\mathcal{G}}

\newcommand{\Z}{\mathbb Z}

\newcommand{\N}{\mathbb N}
\newcommand{\T}{\mathbb T}

\newcommand{\s }{\sigma }
\newcommand{\ii }{{\rm i} }
\renewcommand{\d }{\delta }

\newcommand{\g }{\gamma}

\newcommand{\vphi}{\varphi }

\renewcommand{\t }{\tau }
\renewcommand{\o }{\omega }

\newcommand{\inv}{^{-1}}
\newcommand{\p}{\pi}

\newcommand{\Lipg}{{\rm{Lip}(\g)}}
\newcommand{\lip}{{\rm lip}}

\newcommand{\mR}{\mathcal{R}}

\begin{document}

\title{\textbf{Quasi-periodic solutions of forced Kirchoff equation}} 

\date{}
\author{ Riccardo Montalto \footnote{Supported in part by the Swiss National Science Foundation.}}

\maketitle

\noindent
{\bf Abstract.}
In this paper we prove the existence and the stability of small-amplitude quasi-periodic solutions with Sobolev regularity, for the 1-dimensional forced Kirchoff equation with periodic boundary conditions. This is the first KAM result for a quasi-linear wave-type equation. The main difficulties are: $(i)$ the presence of the highest order derivative in the nonlinearity which does not allow to apply the classical KAM scheme, $(ii)$ the presence of double resonances, due to the double multiplicity of the eigenvalues of $- \partial_{xx}$. The proof is based on a Nash-Moser scheme in Sobolev class. The main point concerns the invertibility of the linearized operator at any approximate solution and the proof of tame estimates for its inverse in high Sobolev norm. To this aim, 
we conjugate the linearized operator to a $2 \times 2$, time independent, block-diagonal operator. This is achieved by using {\it changes of variables} induced by diffeomorphisms of the torus, {\it pseudo-differential} operators and a KAM {\it reducibility} scheme in Sobolev class.
\\



\smallskip

\noindent

\smallskip

\noindent

\tableofcontents

\section{Introduction and main results}
We consider the Kirchoff equation in 1-dimension with periodic boundary conditions 
\begin{equation}\label{main equation}
\partial_{tt} v - \Big( 1 + \int_{\T} |\partial_x v|^2\,dx  \Big) \partial_{xx} v = \delta f(\omega t, x) \,, \quad x \in \T\,,
\end{equation}
where $\T := \R/(2 \pi \Z)$ is the 1-dimensional torus,  $\delta> 0 $ is a small parameter, $f \in {\mathcal C}^q(\T^\nu \times \T, \R)$ and $\omega \in \Omega \subseteq \R^\nu$, with $\Omega$ bounded.
Our aim is to prove the existence and the linear stability of small-amplitude quasi-periodic solutions with Sobolev regularity, for $\delta$ small enough and for $\omega$ in a suitable {\it Cantor like set} of parameters with asymptotically full Lebesgue measure.

\noindent
The Kirchoff equation has been introduced for the first time in 1876 by Kirchoff, in dimension 1, without forcing term and with Dirichlet boundary conditions, namely 
\begin{equation}\label{Kirchoff-Dirichlet}
\partial_{tt} v - \Big(1 + \int_0^\pi |\partial_x v|^2 \, d x \Big) \partial_{xx} v = 0\,, \qquad v(t, 0) = v(t , \pi) = 0\,,
\end{equation}
to describe the transversal free vibrations of a clamped string in which the dependence of the tension on the deformation cannot be neglected. It is a quasi-linear PDE, namely the nonlinear part of the equation contains as many derivatives as the linear differential operator. The Cauchy problem for the Kirchoff equation (also in higher dimension) has been extensively studied, starting from the pioneering paper of Bernstein \cite{Bernstein}. Both local and global existence results have been established for initial data in Sobolev and analytic class, see \cite{arosio-spagnolo}, \cite{arosio-panizzi}, \cite{dancona-spagnolo}, \cite{dickey}, \cite{lions}, \cite{manfrin}, \cite{pozohaev}. 

\noindent
Concernig the existence of periodic solutions, Kirchoff himself observed that the equation \eqref{Kirchoff-Dirichlet} admits a sequence of {\it normal modes}, namely solutions of the form $v(t, x) = v_j(t) \sin(j x)$ where $v_j(t)$ is $2 \pi$-periodic. Under the presence of the forcing term $f(t, x)$ the {\it normal modes} do not persist, since, expanding $v(t, x) = \sum_{j} v_j(t) \sin(j x)$, $f(t, x) = \sum_{j} f_j(t) \sin(j x)$, all the components $v_j(t)$ are coupled in the integral term $\int_\T |\partial_x v|^2\, d x$ and the equation \eqref{Kirchoff-Dirichlet} is equivalent to the infinitely many coupled ODEs
$$
v_j''(t) + j^2 v_j(t) \Big(1 + \sum_{k} k^2 |v_k(t)|^2 \Big) = f_j(t)\,,\qquad j = 1,2, \ldots\,.
$$
The existence of periodic solutions for the Kirchoff equation, also in higher dimension, have been proved by Baldi in \cite{Baldi Kirchhoff}, both for Dirichlet boundary conditions ($v = 0$ on $\partial \Omega$) and for periodic boundary conditions ($\Omega = \T^d$). This result is proven via Nash-Moser method and thanks to the special structure of the nonlinearity (it is diagonal in space), the linearized operator at any approximate solution can be inverted by Neumann series. This approach does not imply the linear stability of the solutions and it does not work in the quasi-periodic case, since the {\it small divisors} problem is more difficult. 

\noindent
In general, the presence of derivatives in the nonlinearity makes uncertain the existence of
global (even not periodic or quasi-periodic) solutions, see for example the non-existence results
in \cite{KM}, \cite{Lax} for the equation $v_{tt} - a(v_x) v_{xx} = 0$, $a > 0$, $a(v) = v^p$, $p \geq 1$, near zero.

\noindent
 Concerning the existence of periodic solutions, 
the first bifurcation result is due to Rabinowitz \cite{Rabinowitz-tesi-1967},  for fully nonlinear forced wave equations
with a small dissipation term 
$$
v_{tt} - v_{xx} + \alpha v_t + \e f(t,x, v, v_t, v_x, v_{tt}, v_{tx}, v_{xx}) = 0\,, \quad x \in \T\,, \qquad \alpha \neq0\,
$$
with frequency $\omega = 1$ ($2 \pi$-periodic solutions). Then Craig \cite{C} proved the existence of small-amplitude periodic solutions, for a large set of frequencies $\omega$, for the autonomous {\it pseudo differential} equation 
$$
\partial_{tt} v - \partial_{xx} v = a(x) v + b(x, |D|^\beta v)\,, \qquad \beta <1
$$
and Bourgain \cite{B2} obtained the same result for the equation $\partial_{tt} v - \partial_{xx} v + m v + (\partial_t v)^2 = 0$. The above results are based on a Newton-Nash-Moser scheme and a Lyapunov-Schmidt decomposition. 

\noindent
For the water waves equations, which are fully nonlinear PDEs, we mention 
the pioneering work of Iooss-Plotnikov-Toland \cite{Ioo-Plo-Tol}   
about  existence of time periodic standing waves, 
and of Iooss-Plotnikov 
\cite{IP09}, \cite{IP11} for 3-dimensional travelling water waves. 
The key idea is to 
use diffeomorphisms of the torus $\T^2$ and pseudo-differential operators, in order to conjugate
the linearized operator 
to one with constant coefficients plus a sufficiently smoothing remainder. 
This is enough to invert the whole linearized operator (at any approximate solution) by Neumann series.  
Very recently Baldi \cite{Baldi-Benj-Ono} has further developed the techniques of \cite{Ioo-Plo-Tol},
proving the existence of periodic solutions for fully nonlinear autonomous, reversible Benjamin-Ono equations. We mention also the recent paper of Alazard and Baldi \cite{Alazard-Baldi} concerning the existence of periodic standing-wave solutions of the water waves equations with surface tension. 

\noindent
These methods do not work for proving the existence of quasi-periodic solutions and they do not imply the linear stability.

\noindent
Existence of quasi-periodic solutions of PDEs (that we shall call in a broad sense KAM theory) with 
unbounded  perturbations (the nonlinearity contains derivatives) has been developed 
by Kuksin \cite{K2} for KdV and then Kappeler-P\"oschel \cite{KaP}. The key idea is to work with a variable coefficients normal form along the KAM scheme. The homological equations, arising at each step of the iterative scheme, are solved thanks to the so called Kuksin lemma, see Chapter 5 in \cite{KaP}. This approach has been improved by Liu-Yuan \cite{LY0}, \cite{LY} who proved a stronger version of the Kuksin Lemma and applied it to derivative NLS and Benjamin-Ono equations. These methods apply to dispersive PDEs  like KdV, derivative NLS but not to derivative wave equation (DNLW) which contains first order derivatives in the nonlinearity. KAM theory for DNLW equation has been recently developed by Berti-Biasco-Procesi in \cite{BBiP1} for Hamiltonian and in \cite{BBiP2} for reversible equations. The key ingredient is to provide a sufficiently accurate asymptotic expansion of the perturbed eigenvalues which allows to impose the {\it Second order Melnikov} conditions. This is achieved by introducing the notion of quasi-T\"oplitz vector field which has been developed by Procesi-Xu \cite{PX} and it is inspired to the T\"oplitz-Lipschitz property developed by Eliasson-Kuksin in \cite{EK1}, \cite{EK}. Existence of quasi-periodic solutions can be also proved by imposing only {\it first order Melnikov conditions}. This method has been developed, for PDEs in higher space dimension, by Bourgain in \cite{Bo1}, \cite{B3}, \cite{B5} for analytic NLS and NLW, extending the result of Craig-Wayne \cite{CW} for semilinear 1-dimensional wave equation. This approach is based on the so-called {\it multiscale analysis} of the linearized equations and it has been recently improved by Berti-Bolle \cite{BBo10}, \cite{BB12} for NLW, NLS with differentiable nonlinearity and by Berti-Corsi-Procesi \cite{BCP} on compact Lie-groups. It is especially convenient in the case of high multiplicity of the eigenvalues, since the second order Melnikov conditions are violated. As a consequence of having imposed only the first order Melnikov conditions, this method does not provide any information about the linear stability of the quasi-periodic solutions, since the linearized equations have variable coefficients. Indeed there are very few results concerning the existence and linear stability of quasi-periodic solutions in the case of multiple eigenvalues. We mention Chierchia-You \cite{ChierchiaYou}, for analytic 1-dimensional NLW equation with periodic boundary conditions (double eigenvalues) and in higher space dimension Eliasson-Kuksin \cite{EK} for analytic NLS. 

\noindent
 All the aforementioned KAM results concern {\it semi-linear} PDEs, namely PDEs in which the order of nonlinear part of the vector field is strictly smaller than the order of the linear part. For quasi-linear (either fully nonlinear) PDEs, 
the first KAM results have been recently proved by Baldi-Berti-Montalto  in \cite{BBM-Airy}, 
\cite{BBM-auto}, \cite{BBM-mKdV} for perturbations of 
Airy, KdV and mKdV equations, by Feola-Procesi \cite{Feola-Procesi} for fully nonlinear reversible Schr\"odinger equation and by Feola \cite{Feola} for quasi-linear Hamiltonian Schr\"odinger equation. For the water waves equations with surface tension, the existence of quasi-periodic standing wave solutions has been recently proved by Berti-Montalto in \cite{Berti-Montalto}.

\medskip

\noindent
The key analysis of the present paper concerns the linearized operator obtained at any step of the Nash-Moser scheme. The main purpose is to reduce the linearized operator to a $2 \times 2$ time independent block diagonal operator. This cannot be achieved by implementing directly a KAM {\it reducibility} scheme, since the constant coefficients part of the linearized operator has the same order as the non-constant part, implying that the {\it loss of derivatives} accumulates quadratically along the iterative scheme. In order to overcome this problem, we perform some transformations which reduce the order of the derivatives in the perturbation but not its size. We use {\it quasi-periodic reparametrization of time} and {\it pseudo differential operators} to reduce the linearized operator to a diagonal operator plus a one smoothing remainder, see \eqref{pappalardo}. At this point we perform a  KAM reducibility scheme that reduces quadratically the size of the remainder at each step of the iteration. Note that, because of the double multiplicity of the eigenvalues $|j|^2$, $j \in \Z$ of the operator $- \partial_{xx}$, the linearized operator cannot be completely diagonalized. This problem is overcome by working with a $2 \times 2$ block diagonal {\it normal form} along the iteration, which is obtained by pairing the space Fourier modes $j$ and $- j$. This strategy has been also developed by Feola \cite{Feola} for quasi-linear Hamiltonian NLS. We will explain more precisely our procedure in Section \ref{idea dell dim}.

\medskip

\noindent
We now state precisely the main results of this paper. Rescaling the variable $v \mapsto \delta^{\frac13} v$ and writing the equation \eqref{main equation} as a first order system, we get the PDE
\begin{equation}\label{Kirchoff first order}
\begin{cases}
\partial_t v = p \\
\partial_t p = \Big( 1 + \e \int_{\T} |\partial_x v|^2\,dx  \Big) \partial_{xx} v + \e  f(\omega t, x)\,,
\end{cases} \qquad \e:= \delta^{\frac23}
\end{equation}
which is a Hamiltonian equation of the form 
$$
\begin{cases}
\partial_t v = \nabla_p H(\omega t, v, p) \\
\partial_t p = - \nabla_v H(\omega t, v, p)
\end{cases}
$$
whose Hamiltonian is 
\begin{align}
& H(\omega t, v, p) \nonumber\\
&  := \frac12 \int_{\T} \big( p^2 + |\partial_x v|^2 \big)\, dx  + \e \Big( \frac12  \int_{\T} |\partial_x v|^2 \, dx \Big)^2 - \e \int_{\T} f(\omega t, x) v\, dx\,. \label{Hamiltoniana Kirchoff}
\end{align}
 We look for quasi-periodic solutions $(v(\omega t, x), p(\omega t, x))$, $v , p : \T^\nu \times \T \to \R$ of the equation \eqref{Kirchoff first order}. This is equivalent to find zeros $(v(\vphi, x), p(\vphi, x))$ of the nonlinear operator 
\begin{equation}\label{operatore non lineare}
{F}( \e, \omega, v,p):= \begin{pmatrix}
 \omega \cdot \partial_\vphi v - p \\
 \omega \cdot \partial_\vphi p - \Big( 1 + \e \int_{\T} |\partial_x v|^2\,dx  \Big) \partial_{xx} v - \e  f 
\end{pmatrix} 
\end{equation}

\noindent
in the Sobolev space $H^s(\T^{\nu + 1}, \R^2) = H^s(\T^{\nu + 1}, \R) \times H^s(\T^{\nu + 1}, \R)$
where 
\begin{align}
 & H^s(\T^{\nu + 1}, \R)   \label{H s} \\
 &  := \Big\{ v(\vphi, x) = \sum_{\begin{subarray}{c}\ell \in \Z^\nu \\
  j \in \Z
 \end{subarray}} \widehat v_{ j}(\ell) e^{\ii (\ell \cdot \vphi + j  x)}:  \| v\|_s^2 :=  \sum_{\begin{subarray}{c}\ell \in \Z^\nu \\
  j \in \Z
 \end{subarray}} \langle \ell, j \rangle^{2 s} |\widehat v_{j}(\ell)|^2 < + \infty \Big\}\,, \nonumber
\end{align}
$\langle \ell, j \rangle := {\rm max}\{1, |\ell|, |j| \}$, $|\ell| := {\rm max}_{i = 1, \ldots, \nu} |\ell_i|$. From now on we fix $s_0 := [(\nu + 1)/2] + 1$, where for any real number $x \in \R$, we denote by $[x]$ its integer part, so that for any $s \geq s_0$ the Sobolev space $H^s(\T^{\nu + 1})$ is compactly embedded in the continuous functions ${\mathcal C}^0(\T^{\nu + 1})$. 

\noindent
We assume that the forcing term $f \in {\mathcal C}^q(\T^\nu \times \T, \R)$ has zero average, namely 
\begin{equation}\label{condizione media f}
\int_{\T^{\nu + 1}} f(\vphi, x)\, d \vphi\, d x = 0\,.
\end{equation}

\noindent
 Now, we are ready to state the main Theorems of this paper. 
\begin{theorem}\label{main theorem kirchoff}
There exist $q := q(\nu) > 0$, $s := s(\nu) > 0$ such that: for any $f \in {\mathcal C}^q(\T^\nu \times \T, \R)$ satisfying the condition \eqref{condizione media f}, there exists $\e_0 = \e_0(f, \nu) > 0$ small enough such that for all $\e \in (0, \e_0)$,  there exists a Cantor set ${\mathcal C}_\e \subseteq \Omega$ of asimptotically full Lebesgue measure i.e. 
$$
|{\mathcal C}_\e| \to |\Omega | \qquad \text{as} \qquad  \e \to 0\,,
$$
such that for any $\omega \in {\mathcal C}_\e$ there exist $v(\e, \omega), p(\e, \omega) \in H^{ s}(\T^{\nu + 1}, \R)$, satisfying 
\begin{equation}\label{condizione media nulla soluzione vphi x}
\int_{\T^{\nu + 1}} v(\vphi, x)\, d \vphi\, d x = \int_{\T^{\nu + 1}} p(\vphi, x)\, d \vphi \, d x = 0\,,
\end{equation}
  such that ${ F}(\e, \omega, v(\e, \omega), p(\e, \omega)) = 0\,,$ where the nonlinear operator $F$ is defined in \eqref{operatore non lineare} and 
\begin{equation}\label{pappoletta}
\|  v(\e, \omega)\|_{\bar s}\,, \| p(\e, \omega) \|_{s} \to 0 \qquad \text{as} \qquad \e \to 0\,.
\end{equation}
\end{theorem}
\begin{remark}
The condition \eqref{condizione media f} on the forcing term $f$ is an essential requirement to get the above existence result. Indeed, if \eqref{condizione media f} does not hold and if $(v, p)$ solves $F(\e, \omega, v, p) = 0$, integrating with respect to $(\vphi, x)$, we get immediately a contraddiction.
\end{remark}
 
 \noindent
We now discuss the precise meaning of linear stability. The linearized PDE on a quasi-periodic function $(v(\omega t, x), p(\omega t, x))$, associated to the equation \eqref{Kirchoff first order}, has the form 
\begin{equation}\label{equazione linearizzata}
\begin{cases}
\partial_t \widehat v = \widehat p \\
\partial_t \widehat p = a(\omega t) \partial_{xx}\widehat v + {\mathcal R}(\omega t)[\widehat v]
\end{cases} 
\end{equation}
where 
\begin{align}
\qquad a(\omega t)  & := 1 + \e \int_{\T} | v_x(\omega t, x)|^2\, d x\,, \label{coefficienti equazione linearizzata1}\\ 
{\mathcal R}(\omega t)[\widehat v] & := - 2 \e v_{xx}(\omega t, x) \int_\T  v_{xx}(\omega t, x) \widehat v\, d x\,. \label{coefficienti equazione linearizzata2}
\end{align}
In order to state precisely the next Theorem, let us introduce, for any $s \geq 0$, the Sobolev spaces
$$
H^s(\T_x, \R) := \big\{ u(x) = \sum_{j \in \Z} u_j e^{\ii j x} : \| u \|_{H^s_x}^2 := \sum_{j \in \Z} \langle j \rangle^{2 s} |u_j|^2 < + \infty \big\}\,,\quad 
$$
$$
H_0^s(\T_x, \R) := \big\{ u \in H^s(\T_x, \R) : \int_\T u(x)\, d x = 0  \big\}\,,
$$
where $\langle j \rangle := {\rm max}\{1, |j|\}$.
\begin{theorem}\label{theorem linear stability} {{\bf (Linear stability)}} 
There exist $ \bar \mu >  0 $, depending on $ \nu $, such that for all $S > s_0 + \overline \mu$, there exists $\e_0 = \e_0(S, \nu) > 0$ such that: for all $\e \in (0, \e_0)$, for all ${\bf v} = (v, p) \in H^{S}(\T^{\nu + 1}, \R^2) $ with $\| {\bf v} \|_{ s_0  + \bar \mu} \leq 1$, 
there exists a Cantor like set $  \Omega_\infty ({\bf v})  \subset \Omega $ 
such that,  for all $ \omega \in \Omega_\infty ({\bf v}) $, for all $s_0 \leq s \leq S - \overline \mu$ the following holds: 
for any initial datum $(\widehat v^{(0)}, \widehat p^{(0)}) \in H^s(\T_x, \R) \times H^{s - 1}_0(\T_x, \R)$ the solution $ t \in \R \mapsto (\widehat v(t, \cdot), \widehat p(t, \cdot))\in H^s(\T_x, \R) \times H^{s - 1}_0(\T_x, \R)$ of the equation \eqref{equazione linearizzata}, with initial datum $\widehat v(0, \cdot) = \widehat v^{(0)}$, $\widehat p(0, \cdot) = \widehat p^{(0)}$ is stable, namely 
$$
\sup_{t \in \R} \Big(\| v(t, \cdot) \|_{H^s_x} + \| p(t, \cdot)\|_{H^{s - 1}_x} \Big) \leq C(s)\big( \| \widehat v^{(0)}\|_{H^s_x} + \| \widehat p^{(0)}\|_{H^{s - 1}_x} \big)\,.
$$
\end{theorem}
\begin{remark}
Note that the linear stability can be proved only for initial data $\widehat p^{(0)}$ with zero-average in $x$. Indeed, the equation \eqref{equazione linearizzata} projected on the zero Fourier mode is the ODE 
$$
\begin{cases}
\dot v_0 (t)= p_0(t) \\
\dot p_0(t) = 0 
\end{cases}
$$ 
whose solutions are 
$$
p_0(t) = p^{(0)} \,, \qquad v_0(t) = v^{(0)} + p^{(0)} t\,,\qquad v^{(0)}, p^{(0)} \in \R\,, \qquad \forall t \in \R\,.
$$
Hence, if $p^{(0)} \neq 0$, $|v_0(t)| \to + \infty$ as $t \to \pm \infty$ and we do not have the stability. 
\end{remark}

\subsection{Ideas of the proof}\label{idea dell dim}
In this section we explain in detail the main ideas of the proof. Because of the special structure of the nonlinear operator $F$ defined in \eqref{operatore non lineare}, it is convenient to perform the decomposition \eqref{proiettore media}, \eqref{proiettore media vettoriale}, in order to split the equation $F(\e, \omega, v, p) = 0$ into the equations \eqref{sistema su media nulla}, \eqref{sistema sul modo 0}. The equation \eqref{sistema sul modo 0} arises by projecting the nonlinear operator $F$ on the zero Fourier mode in $x$ and it is a constant coefficients PDE which can be easily solved by imposing a diophantine condition on the frequency vector $\omega$ (see Lemma \ref{lemma equazione sulle medie nulle}). Hence, we are reduced to find zeros of the nonlinear operator ${\mathcal F}$ defined in \eqref{operatore su media nulla} which is obtained by restricting $F$ to the space of the functions with zero average in $x$. Theorems \ref{main theorem kirchoff}, \ref{theorem linear stability} then follow by Theorem \ref{main theorem kirchoff 1}, which is based on  a Nash-Moser iteration on the nonlinear map ${\mathcal F}$ on the scale of Sobolev spaces $H^s_0(\T^{\nu + 1}, \R^2)$, see \eqref{Sobolev media nulla}. The main issue concerns the invertibility of the linearized operator ${\mathcal L} = \partial_{(u, \psi)} {\mathcal F}(u, \psi)$ in \eqref{operatore linearizzato} at any approximate solution and the proof of tame estimates for its inverse (see Theorem \ref{partial invertibility of the linearized operator}). This information is obtained by conjugating ${\mathcal L}$ to a $2 \times 2$ time-independent block diagonal operator. Such a conjugacy procedure is the content of Sections \ref{riduzione linearizzato}, \ref{sec:redu}. 

\noindent
{\bf Regularization of the linearized operator.} The goal of Section \ref{riduzione linearizzato} is to reduce the linearized operator ${\mathcal L}$ in \eqref{operatore linearizzato} to the operator ${\mathcal L}_4$ in \eqref{cal L4} which has the form 
\begin{equation}\label{pappalardo}
 {\bf h} = (h, \overline h) \mapsto \omega \cdot \partial_\vphi {\bf h} + \ii m T |D| {\bf h} + {\mathcal R}_4 {\bf h}\,,
\end{equation}
where $m \in \R$ is close to $1$, $T := \begin{pmatrix}
1 & 0 \\
0 & - 1
\end{pmatrix}$, $|D| = \sqrt{- \partial_{xx}}$ and ${\mathcal R}_4$ is a Hamiltonian (see Section \ref{sezione formalismo hamiltoniano}) and 1-smoothing operator. More precisely the operator ${\mathcal R}_4$ satisfies $|{\mathcal R}_4 |D||_s < +\infty$, see Lemma \ref{stime prima riducibilita}, where the {\it block-decay norm} $|\cdot |_s$ is defined in \eqref{decadimento Kirchoff}. This regularization procedure is splitted in three parts. 

\medskip

\noindent
{\it 1. Symmetrization and complex variables.} In Section \ref{step 1 riduzione}, we symmetrize the highest order non-constant coefficients term $a(\vphi) \partial_{xx}$ in \eqref{operatore linearizzato}, by conjugating ${\mathcal L}$ with the transformation ${\mathcal S}$, defined in \eqref{cal S1}. The conjugated operator ${\mathcal L}_1$, defined in \eqref{cal L1}, has the form 
$$
\begin{pmatrix}
\widehat u \\
\widehat \psi
\end{pmatrix} \mapsto \begin{pmatrix}
\omega \cdot \partial_\vphi + a_0(\vphi) & - a_1(\vphi) |D| \\
a_1(\vphi) |D| + {\mathcal R}^{(1)} & \omega \cdot \partial_\vphi - a_0(\vphi)
\end{pmatrix}\begin{pmatrix}
\widehat u \\
\widehat \psi
\end{pmatrix}
$$
where $a_1, a_0$ are real valued Sobolev functions in $H^s(\T^\nu, \R)$, with $ a_1 - 1\,,\, a_0 = O(\e)$ and ${\mathcal R}^{(2)}$ is an arbitrarily regularizing operator of the form \eqref{operatore forma buona resto}. In section \ref{variabili complesse}, we introduce the complex variables $h = \frac{1}{\sqrt{2}} (\widehat u + \ii \widehat \psi)$ and the operator ${\mathcal L}_1$ transforms into ${\mathcal L}_2$ defined in \eqref{cal L2 complex coordinates} which has the form 
$$
\begin{pmatrix}
h \\
\overline h
\end{pmatrix}
\mapsto 
 \begin{pmatrix}
\big(\omega \cdot \partial_\vphi  + \ii a_1(\vphi) |D|  + \ii {\mathcal R}^{(2)} \big) h +  \big(a_0(\vphi)  + \ii {\mathcal R}^{(2)} \big) \overline h \\
\qquad \qquad \text{complex\,\,conjugate}
\end{pmatrix},
$$
with ${\mathcal R}^{(2)}  := \frac{1}{2}{\mathcal R}^{(1)}$.
\medskip

\noindent
{\it 2. Change of variables.} In Section \ref{sezione diffeo del toro}, we reduce to constant coefficients the highest order term $\ii a_1(\vphi) |D|$ in the operator ${\mathcal L}_2$. Note that it depends only on time. This is due to the special structure of the equation, since the nonlinear term is {\it diagonal} in space. To reduce to constant coefficients $\ii a_1(\vphi) |D|$, we conjugate ${\mathcal L}_2$ by means of the reparametrization of time ${\mathcal A} h (\vphi, x) := h(\vphi + \omega \alpha(\vphi, x))$ induced by the diffeomorphism of the torus $\T^\nu$, $\vphi \mapsto \vphi + \omega \alpha(\vphi)$. Since $\omega$ is diophantine, choosing $\alpha(\vphi)$ as in \eqref{definition alpha}, the transformed operator ${\mathcal L}_3$ defined in \eqref{cal L3} is 
$$
 \begin{pmatrix}
 h \\
 \overline h
 \end{pmatrix} \mapsto \begin{pmatrix}
\big( \omega \cdot \partial_\vartheta + \ii m |D| + \ii {\mathcal R}^{(3)} \big) h + \big( b_0 + \ii {\mathcal R}^{(3)} \big) \overline h \\
\qquad \qquad \text{complex\,\,conjugate}
\end{pmatrix}
$$
where $m \in \R$ is a constant $m \simeq 1$, $b_0 = O(\e)$ is a real valued Sobolev function in $H^s(\T^\nu, \R)$ and ${\mathcal R}^{(4)}$ is a one-smoothing operator still satisfying the estimates \eqref{stime cal R3}. Actually ${\mathcal R}^{(4)}$ is arbitrarily smoothing, since it has the form \eqref{operatore forma buona resto}, but we only need that it is one-smoothing. 

\medskip

\noindent
{\it 3. Descent method.} In Section \ref{descent method}, we perfom one step of descent method, in order to remove the zero-th order term from the operator ${\mathcal L}_3$. Since the operator ${\mathcal R}^{(3)}$ is already one-smoothing, we need just to remove the multiplication operator $\overline h \mapsto b_0(\vphi) \overline h$. For this purpose we transform ${\mathcal L}_3$ by means of the symplectic transformations ${\mathcal V} = {\rm exp}(\ii V(\vphi) |D|^{- 1})$, $V(\vphi) = \begin{pmatrix}
0 & v(\vphi) \\
- v(\vphi) & 0
\end{pmatrix}$ where $v : \T^\nu \to \R$ is a real valued Sobolev function. Choosing $v$ as in \eqref{definizione w0(vphi)}, the transformed operator ${\mathcal L}_4$ in \eqref{cal L4} is the sum of a diagonal operator and a 1-smoothing operator ${\mathcal R}_4$, such that ${\mathcal R}_4 |D|$ has finite block-decay norm.

\medskip

\noindent
{\bf $2 \times 2$-block diagonal reducibility scheme.} Once \eqref{pappalardo} has been obtained, we perform a quadratic KAM reducibility scheme which conjugates the operator ${\mathcal L}_4$ to the $2 \times 2$ block diagonal operator ${\mathcal L}_\infty$ (see Theorems \ref{thm:abstract linear reducibility}, \ref{teoremadiriducibilita}). The reason for which we cannot completely diagonalize the operator ${\mathcal L}_4$ is the following: since we deal with periodic boundary conditions, the eigenvalues of the operator $m |D|$ are double, therefore the second order Melnikov conditions for the differences $m |j| - m |\pm j|$ are violated. This implies that after the first step of the KAM iteration, the correction to the diagonal part $\ii m |D|$ is an operator of the form $\ii \widehat D$, 
$
\widehat D= {\rm diag}_{j \in \N} \widehat{\bf D}_j\,, \quad 
$
where $\widehat{\bf D}_j$ is a linear self-adjoint operator ${\rm span}\{ e^{\ii j x}, e^{- \ii j x}\} \to {\rm span}\{ e^{\ii j x}, e^{- \ii j x}\}$ which we identify with the $2 \times 2$ self-adjoint matrix of its Fourier coefficients with respect to the basis $\{e^{\ii j x}, e^{- \ii j x} \}$. 
The self-adjointness of the $2 \times 2$ blocks is provided by the Hamiltonian structure. In order to deal with these $2 \times 2$ block diagonal operators, it is convenient to introduce a $2 \times 2$ block representation for linear operators. We develop this formalism in Section \ref{sezione rappresentazione 2 per 2}. We remark that the problem of the double multiplicity of the eigenvalues has been overcome for the first time by Chierchia-You \cite{ChierchiaYou}, for analytic semilinear Klein-Gordon equation with periodic boundary condition. We also mention that the $2 \times 2$-block diagonal reducibility scheme, adopted in Section \ref{sec:redu}, has been recently developed by Feola \cite{Feola} for quasi-linear Hamiltonian NLS equation.

\noindent
 One of the main task in the KAM reducibility scheme is to provide, along the iterative scheme, an asymptotic expansion of the perturbed $2 \times 2$ blocks of form 
\begin{equation}\label{asintotica intro}
  \begin{pmatrix} m |j|  & 0 \\
 0 &  m |j|
 \end{pmatrix}+ O(\e |j|^{- 1})\,.
\end{equation} 
This expansion allows to show that the required second order Melnikov non-resonance conditions are fullfilled for a large set of frequencies $\omega$. The asymptotic \eqref{asintotica intro} is achieved since the initial remainder ${\mathcal R}_0$ is 1-smoothing and this property is preserved along the reducibility scheme (see \eqref{Rsb} in Theorem \ref{thm:abstract linear reducibility}). This is the reason why we performed the regularization procedure of Section \ref{riduzione linearizzato} up to order $O(|D|^{- 1})$. We use the block-decay norm $| \cdot |_s$ (see \eqref{decadimento Kirchoff}) to estimate the size of the remainders along the iteration. This is convenient since the class of operators having finite block-decay norm is closed under composition (Lemma \ref{interpolazione decadimento Kirchoff}), solution of the homological equation (Lemma \ref{homologica equation}) and projections (Lemma \ref{lemma smoothing decay}). 

\medskip

\noindent
{\bf Linear stability.} A final comment concerns Theorem \ref{theorem linear stability} which is proved in Section \ref{proof linear stability}. Using the splitting \eqref{proiettore media}, \eqref{proiettore media vettoriale}, the linearized equation \eqref{equazione linearizzata} is decoupled into the two systems \eqref{papapapa}, \eqref{blablabla}. The system \eqref{papapapa} is a constant coefficients ODE which can be solved explicitly, hence it is enough to study the stability for the PDE \eqref{blablabla}, which is obtained by \eqref{equazione linearizzata}, restricting the vector field to the zero average functions in $x$. All the transformations we perform along the reduction procedure of Sections \ref{riduzione linearizzato}, \ref{sec:redu} are T\"oplitz in time operators (see Section \ref{astratto operatori Toplitz}), hence they can be regarded as time dependent quasi-periodic maps acting on the phase space (functions of $x$ only). Hence, by the procedure of Sections \ref{riduzione linearizzato}, \ref{sec:redu}, the linear equation \eqref{blablabla}, transforms into the PDE \eqref{equazione lineare ridotta}, whose vector field is a time independent $2 \times 2$ block-diagonal operator. Thanks to the Hamiltonian structure, such a vector field is skew self-adjoint, implying that all the Sobolev norms of the solutions remain constant for all time. This is enough to deduce the linear stability.  

\section{Functional setting}\label{sec:2} 

We may regard a function  $ u \in L^2 (\T^\nu \times \T, \C )$ of space-time also as a 
 $ \vphi $-dependent family of  functions $ u(\vphi, \cdot ) \in L^2 (\T_x, \C) $ that we expand in Fourier series as
\be\label{raggruppamento modi Fourier}
u(\vphi, x ) =   \sum_{j \in \Z  } u_{j} (\vphi) e^{\ii j x } =
\sum_{\begin{subarray}{c}
\ell \in \Z^\nu \\
 j \in \Z 
 \end{subarray}}  \widehat u_j(\ell)  e^{\ii (\ell \cdot \vphi + j  x)}   \, ,
\ee
where 
$$
u_j(\vphi) := \frac{1}{2 \pi} \int_{\T} u(\vphi, x) e^{- \ii j  x}\, d x\,, 
$$
$$
\widehat u_j(\ell) :=  \frac{1}{(2 \pi)^{\nu + 1}} \int_{\T^{\nu + 1}} u(\vphi, x) e^{- \ii (\ell \cdot \vphi + j  x)}\,d \vphi\, d x\,.
$$
We also consider the space of the $L^2$ real valued functions that we denote by $L^2(\T^{\nu + 1}, \R)$, $L^2(\T_x, \R)$. We define for any $s \geq 0$ the Sobolev spaces $H^s(\T^{\nu + 1}, \C)$, $H^s(\T_x, \C)$ as 
$$
H^s(\T^{\nu + 1}, \C) := \big\{ u \in L^2 (\T^\nu \times \T, \C ) : \| u \|_s^2 := \sum_{(\ell, j) \in \Z^\nu \times \Z} \langle \ell, j \rangle^{2 s} |\widehat u_j(\ell)|^2 < +\infty \big\}\,, 
$$
$$
H^s(\T_x, \C) := \big\{ u \in L^2 (\T_x, \C ) : \| u \|_{H^s_x}^2 := \sum_{j \in \Z} \langle  j \rangle^{2 s} |\widehat u_j|^2 < +\infty \big\} 
$$
where $ \langle \ell, j \rangle := {\rm max}\{1 , |\ell|, |j| \}$, $\langle j \rangle := {\rm max}\{1 , |j| \}$ and for any $\ell \in \Z^\nu$, $|\ell| := {\rm max}_{i = 1, \ldots, \nu} |\ell_i|$.
In a similar way, we define the Sobolev spaces of real values functions $H^s(\T^{\nu + 1}, \R)$, $H^s(\T_x, \R)$. When no confusion appears, we simply write $L^2(\T^{\nu + 1})$, $L^2(\T_x)$, $H^s(\T^{\nu + 1})$, $H^s(\T_x)$. For any $s \geq 0$ we also define 
\begin{equation}\label{Sobolev media nulla}
H^s_0(\T^{\nu + 1}) := \big\{ u \in H^s(\T^{\nu + 1}) : \int_\T u(\vphi, x)\, dx = 0 \big\}\,, 
\end{equation}
\begin{equation}\label{Sobolev media nulla2}
H^s_0(\T_x) := \big\{ u \in H^s(\T_x) : \int_\T u( x)\, dx = 0 \big\}\,.
\end{equation}
and $L^2_0(\T^{\nu + 1}) = H^0_0(\T^{\nu + 1})$, $L^2_0(\T_x) = H^0_0(\T_x)$. We define the spaces $H^s_0(\T^{\nu + 1}, \C^2) := H^s_0(\T^{\nu + 1}, \C) \times H^s_0(\T^{\nu + 1}, \C)$ and $H^s_0(\T_x, \C^2) := H^s_0(\T_x, \C) \times H^s_0(\T_x, \C) $ equipped, respectively, by the norms $\|(h, v)\|_s := {\rm max}\{\| h \|_s \,, \| v\|_s \}$ and $\|(h, v)\|_{H^s_x} := {\rm max}\{\| h \|_{H^s_x} \,, \| v\|_{H^s_x} \}$. Similarly we define $H^s_0(\T^{\nu + 1}, \R^2) := H^s_0(\T^{\nu + 1}, \R) \times H^s_0(\T^{\nu + 1}, \R)$ and $H^s_0(\T_x, \R^2) := H^s_0(\T_x, \R) \times H^s_0(\T_x, \R)$ and the norms are defined as in the complex case.  

\noindent
 For a function $f : \Omega_o \to E$, $\omega \mapsto f(\omega)$, where $(E, \| \cdot \|_E)$ is a Banach space and 
$ \Omega_o $ is a subset of $\R^\nu$, we define the sup-norm and the lipschitz semi-norm as 
\be \label{def norma sup lip}
\| f \|^{\sup}_{E, \Omega_o} 
:= \sup_{ \omega \in \Omega_o } \| f(\omega) \|_E \,,\quad \| f \|_{E, \Omega_o}^{\lip} := \sup_{\begin{subarray}{c}
\omega_1, \omega_2 \in \Omega_o \\
\omega_1 \neq \omega_2
\end{subarray}} \frac{\| f(\omega_1) - f(\omega_2)\|_E}{|\omega_1 - \omega_2|}
\ee
and, for $ \g > 0 $, we define the weighted Lipschitz-norm
\be \label{def norma Lipg}
\| f \|^{\Lipg}_{E, \Omega_o}  
:= \| f \|^{\sup}_{E, \Omega_o} + \g \|  f \|^{\lip}_{E, \Omega_o}  \, . 
\ee
To shorten the above notations we simply omit to write $\Omega_o$, namely $\| f \|^{\sup}_{E} = \| f \|^{\sup}_{E, \Omega_o}$, $\| f \|_{E}^{\lip} = \| f \|_{E, \Omega_o}^{\lip}$, $\| f \|^{\Lipg}_{E} = \| f \|^{\Lipg}_{E, \Omega_o}$
If $f : \Omega_o \to \C$, we simply denote $\| f \|_{\C}^{\Lipg}$ by $|f|^\Lipg$
and if $ E = H^s(\T^{\nu + 1}) $ we simply denote $ \| f \|^{\Lipg}_{H^s} := \| f \|^{\Lipg}_s $. Given two Banach spaces $E, F$, we denote by ${\mathcal L}(E, F)$ the space of the bounded linear operators $E \to F$. If $E = F$, we simply write ${\mathcal L}(E)$. 

\medskip

\noindent
{\it Notation:} The notation $a \leq_s b$ means that there exists a constant $C(s) > 0$ depending on $s$ such that $a \leq C(s) b$. The constant $C(s)$ may depend also on the data of the problem, namely the number of frequencies $\nu$, the diophantine exponent $\tau > 0$ appearing in the non-resonance conditions, the forcing term $f$. If the constant $C$ does not depend on $s$ or if $s = s_0 = [(\nu + 1)/2] + 1$, we simply write $a \lessdot b$.

\medskip

\noindent
We recall the classical estimates for the operator $(\omega \cdot \partial_\vphi)^{- 1}$ defined as 
\begin{equation}\label{om d vphi inverso}
(\omega \cdot \partial_\vphi)^{- 1}[1] = 0\,, \quad (\omega \cdot \partial_\vphi)^{- 1}[e^{\ii \ell \cdot \vphi}] = \frac{1}{\ii (\omega \cdot \ell)} e^{\ii \ell \cdot \vphi}\,, \qquad \forall \ell \neq 0\,,
\end{equation}
for $\omega \in \Omega_{\gamma, \tau}$, where for $\gamma, \tau > 0$, 
\begin{equation}\label{diofantei Kn}
\Omega_{\gamma, \tau} := \Big\{ \omega \in \Omega : |\omega \cdot \ell| \geq \frac{\gamma}{| \ell |^\tau}\,, \quad \forall \ell \in \Z^\nu \setminus \{ 0 \}\,, \Big\}\,. 
\end{equation}
 If $h(\cdot ; \omega) \in H^{s + 2 \tau + 1}(\T^{\nu + 1})$, with $\omega \in \Omega_{\gamma, \tau}$, we have 
\begin{equation}\label{stima om d vphi inverso}
\| (\omega \cdot \partial_\vphi)^{- 1} h \|_s \leq \gamma^{- 1} \| h \|_{s + \tau}\,, \qquad \| (\omega \cdot \partial_\vphi)^{- 1} h \|_s^\Lipg \leq \gamma^{- 1} \| h \|_{s + 2 \tau + 1}^\Lipg\,.
\end{equation}
Denote by $\N$, the set of the strictly positive integer numbers $\N := \{1,2,3, \ldots \}$ and we set $\N_0 = \{ 0 \} \cup \N$.
 Given a function $h \in L^2_0(\T^{\nu + 1})$, we can write 
\begin{equation}\label{raggruppamento modi Fourier}
h(\vphi, x) = \sum_{\begin{subarray}{c}
\ell \in \Z^\nu \\
 j \in \Z \setminus \{ 0 \}
 \end{subarray}} \widehat h_{ j}(\ell) e^{\ii (\ell \cdot \vphi + j  x)} = \sum_{\begin{subarray}{c}\ell \in \Z^\nu \\
  j \in \N
  \end{subarray}} \widehat{\bf h}_{ j}(\ell, x) e^{\ii \ell \cdot \vphi} \,,
\end{equation}
where 
\begin{equation}\label{bf h ell alpha}
\widehat{\bf h}_{ j}(\ell, x) :=  \widehat h_{ j}(\ell) e^{\ii  j  x} + \widehat h_{- j}(\ell) e^{- \ii j x}\,, \qquad \forall j \in \N\,.
\end{equation}
It is straightforward to see that if $h \in H^s_0(\T^{\nu + 1})$, one has 
\begin{align}
\| h \|_s^2 = &  \sum_{\begin{subarray}{c}
\ell \in \Z^\nu \\
  j \in \N
  \end{subarray}} \langle \ell, j \rangle^{2 s} \| \widehat{\bf h}_{ j}(\ell) \|_{L^2}^2\,. \label{altro modo norma s}
\end{align}

\noindent
 We now recall the following classical interpolation result.
\begin{lemma}\label{interpolazione C1 gamma}
Let $u, v \in H^s(\T^{\nu + 1})$ with $s \geq s_0$. Then, there exists an increasing function $s \mapsto C(s)$ such that 
$$
\| u v \|_s \leq C(s) \| u \|_s \| v \|_{s_0}+ C(s_0)\| u \|_{s_0} \| v \|_s\,.
$$
If $u(\cdot; \omega)$, $v(\cdot; \omega)$, $\omega \in \Omega_o \subseteq \R^\nu$ are $\omega$-dependent families of functions in $H^{s}(\T^{\nu + 1})$, with $s \geq s_0$ then the same estimate holds replacing $\| \cdot \|_s$ by 

$\| \cdot \|_s^\Lipg$.
\end{lemma}
Iterating the above inequality one gets that, for some constant $K(s)$, for any $n \geq 0$, 
\begin{equation}\label{interpolazione iterata}
\| u ^k\|_s \leq K(s)^k \| u\|_{s_0}^{k - 1} \| u \|_s\,
\end{equation}
and if $u(\cdot; \omega) \in H^s$, $s \geq s_0$ is a family of Sobolev functions, the same inequality holds repacing $\| \cdot \|_s$ by $\| \cdot \|_s^\Lipg$.

\noindent
We also recall the classical Lemmas on the composition operators. Since the variables $ (\vphi, x)$ have the same role, we present it for a  generic Sobolev space  $ H^s (\T^n ) $. For any $s \geq 0$ integer, for any domain $A \subseteq \R^n$ we denote by ${\mathcal C}^s(A)$ the space of the s-times continuously differentiable functions equipped by the usual $\| \cdot \|_{{\mathcal C}^s}$ norm. We consider the composition operator
$$
u(y) \mapsto {\mathtt f}(u)(y) := f(y, u(y))\,. 
$$ 
The following Lemma holds:

\begin{lemma}\label{Moser norme pesate} {\bf (Composition operator)}
Let $ f \in {\mathcal C}^{s + 1}(\T^n \times \R, \R )$, with  $   s \geq s_0 :=   [n/2] + 1 $. If $u \in H^s(\T^n)$, with $\| u \|_{s_0} \leq 1$, then $\| \mathtt f(u)\|_s \leq C(s, \| f \|_{{\mathcal C}^s})(1 + \| u \|_s)$.  
If   $u(\cdot, \omega) \in H^s(\T^n)$, $\omega \in \Omega_o \subseteq \R^\nu$  is a family of Sobolev functions
satisfying $\| u \|_{s_0}^{\Lipg} \leq 1$, then,  
$  \| {\mathtt f}(u) \|_s^{\Lipg} \leq C(s,  \| f\|_{{\mathcal C}^{s + 1}} ) ( 1 + \| u \|_{s}^{\Lipg}) $. 
\end{lemma}
Now we state the tame properties of the composition operator $ u(y) \mapsto  u(y+p(y)) $
induced by a diffeomorphism of the torus $ \T^n $. The Lemma below, can be proved as Lemma 2.20 in \cite{Berti-Montalto}.
\begin{lemma} {\bf (Change of variable)}  \label{lemma:utile} 
Let $p:= p( \cdot; \omega ):\R^n \to \R^n$, $\omega \in \Omega_o \subset \R^\nu $ be a family of $2\p$-periodic functions satisfying   
\begin{equation}\label{mille condizioni p}
  \| p \|_{{\mathcal C}^{s_0 + 1}} \leq 1/2\,,\quad  \| p \|_{s_0}^{\Lipg} \leq 1
\end{equation}
where $s_0 := [n/2] + 1$. Let $ g(y) := y + p(y) $,
$ y \in \T^n $. 
Then the composition operator 
$$
A : u(y) \mapsto (u\circ g)(y) = u(y+p(y))
$$ 
satisfies for all $s \geq s_0$, the tame estimates
\begin{equation}\label{stima cambio di variabile dentro la dim}
\| A u\|_{s_0} \leq_{s_0} \| u \|_{s_0}\,, \qquad \| A u\|_s \leq_s  C(s)\| u \|_s + C(s_0)\| p \|_s \| u \|_{s_0 + 1}\,.
\end{equation}
Moreover, for any family of Sobolev functions $u(\cdot; \omega)$
\begin{align}
& \| A u \|_{s_0}^{\Lipg} \leq_{s_0}  \| u \|_{s_0 + 1}^{\Lipg}\,, \label{stima tame cambio di variabile pietro s0}\\
\label{stima tame cambio di variabile pietro}
    & \| Au \|_s^{\Lipg} \leq_{s} \| u \|_{s + 1}^{\Lipg} + \| p \|_{s}^{\Lipg} \| u \|_{s_0 + 2}^{\Lipg}\,, \quad \forall  s > s_0  \,.
\end{align}
The map $ g $ is invertible with inverse  $ g^{- 1}(z) = z + q(z) $ and 
there exists a constant $\delta := \delta(s_0) \in (0,1) $ such that, if  $ \| p \|_{2 s_0 + 2}^{\Lipg} \leq \d$, then  
\begin{equation}\label{stime-lipschitz-q}
\| q \|_s \leq_s \| p \|_s\,,\qquad \| q \|_{s}^{\Lipg}  \leq_{s}  \| p \|_{s + 1}^{\Lipg} \,. 
\end{equation}
Furthermore, the composition operator $A^{- 1} u(z) := u(z + q(z))$ satisfies the estimate  
\begin{equation}\label{tame-cambio-di-variabile-inverso}
\| A^{- 1} u\|_s \leq_s \| u \|_{s} + \| p \|_{s} \| u \|_{s_0 + 1}\,, \quad \forall s \geq s_0\,
\end{equation}
and for any family of Sobolev functions $u(\cdot; \omega)$
\begin{equation}\label{tame-lipschitz-cambio-di-variabile}
\| A^{- 1} u\|_s^{\Lipg} \leq_s \| u \|_{s + 1}^{\Lipg} + \| p \|_{s + 1}^{\Lipg} \| u \|_{s_0 + 2}^{\Lipg}\,, \quad \forall s \geq s_0\,.
\end{equation}
\end{lemma}

\subsection{T\"oplitz in time linear operators}\label{astratto operatori Toplitz}
Let $ {\mathcal R} : \T^\nu \mapsto {\mathcal L}( L^2_0(\T_x))  $, 
$ \vphi \mapsto {\mathcal R}(\vphi) $,   be a $ \vphi $-dependent family of linear 
operators acting on $ L^2_0 (\T_x) $. We regard $ {\mathcal R} $  also 
 as an operator (that for simplicity we denote by ${\mathcal R} $ as well)
 which acts on functions $ u \in L^2_0(\T^\nu \times \T) $ of space-time, i.e. 
 we consider the operator 
$ {\mathcal R} \in {\mathcal L}(L^2_0(\T^\nu \times \T ) )$  
defined by
$$
( {\mathcal R} u) (\varphi , x) := ({\mathcal R}(\varphi) u(\varphi, \cdot ))(x) \, .  
$$
The action of this operator on a function $ u \in L^2_0(\T^{\nu + 1}) $ 
is given by  
\begin{align}
{\mathcal R} u (\vphi, x) & = \sum_{j , j' \in \Z \setminus \{ 0 \}} {\mathcal R}_j^{j'}(\vphi)  u_{j'}(\vphi) e^{\ii j   x} \nonumber\\
& =  \sum_{\begin{subarray}{c}
 \ell , \ell' \in \Z^\nu \\
  j , j' \in \Z \setminus \{ 0 \}
  \end{subarray}}  {\mathcal R}_j^{j'}(\ell - \ell') \widehat u_{ j'}(\ell') e^{\ii (\ell \cdot \vphi + j  x)} \, \label{action toplitz operator}
\end{align}
where the space Fourier coefficients ${\mathcal R}_j^{j'}(\vphi)$ and the space-time Fourier coefficients ${\mathcal R}_j^{j'}(\ell)$ of the operator ${\mathcal R}$ are defined as 
\begin{equation}\label{coefficienti spazio operatore}
{\mathcal R}_j^{j'}(\vphi) := \frac{1}{2 \pi} \int_{\T} {\mathcal R}(\vphi)[e^{\ii j'  x}] e^{- \ii j  x}\, d x\,, \quad  \vphi \in \T^\nu\,,\,j, j' \in \Z \setminus \{ 0 \}\,,
\end{equation}
\begin{equation}\label{coefficienti spazio tempo operatore}
{\mathcal R}_j^{j'}(\ell) := \frac{1}{(2 \pi)^{\nu }}  \int_{\T^\nu} {\mathcal R}_j^{j'}(\vphi) e^{- \ii \ell \cdot \vphi }\, d \vphi \,,\quad  \ell \in \Z^\nu\,,\, j, j' \in \Z \setminus \{ 0 \}\,.
\end{equation}
We shall identify the operator $ {\mathcal R} = {\mathcal R}(\vphi) $ with the infinite-dimensional matrices of its Fourier coefficiens
\begin{equation}\label{matrice operatore Toplitz}
\Big( {\mathcal R}_j^{j'}(\vphi)\Big)_{j, j' \in \Z \setminus \{ 0 \}} \,, \quad \Big( {\mathcal R}_j^{j'}(\ell - \ell')\Big)_{\begin{subarray}{c}
\ell , \ell' \in \Z^\nu  \\
 j, j' \in \Z \setminus \{ 0 \} 
\end{subarray}}\,
\end{equation}
 and we refer to such operators as  T\"oplitz in time operators. 
 
 \noindent
 If the map $\vphi \in \T^\nu \mapsto {\mathcal R}(\vphi) \in {\mathcal L}(L^2_0(\T_x))$ is differentiable, given $\omega \in \R^\nu$, we can define the operator $\omega \cdot \partial_\vphi {\mathcal R}$ as 
 \begin{equation}\label{definizione D omega cal R}
 \omega \cdot \partial_\vphi {\mathcal R} = \big( \omega \cdot \partial_\vphi {\mathcal R}_j^{j'}(\vphi)\big)_{j, j' \in \Z \setminus \{ 0 \}} = \Big( \ii \,\omega \cdot (\ell - \ell'){\mathcal R}_j^{j'}(\ell - \ell')\Big)_{\begin{subarray}{c}
\ell , \ell' \in \Z^\nu  \\
 j, j' \in \Z \setminus \{ 0 \} \,
\end{subarray}}\,.
 \end{equation}
 We also define the {\it commutator} between two T\"oplitz in time operators ${\mathcal R} = {\mathcal R}(\vphi)$ and ${\mathcal B}= {\mathcal B}(\vphi)$ as $[{\mathcal R}(\vphi), {\mathcal B}(\vphi)] := {\mathcal R}(\vphi) {\mathcal B}(\vphi) - {\mathcal B}(\vphi){\mathcal R}(\vphi)$, $\vphi \in \T^\nu$. 
 
 \noindent
 Given a T\"oplitz in time operator ${\mathcal R}$, we define the conjugated operator $\overline{\mathcal R}$ by 
 \begin{equation}\label{definizione operatore coniugato}
 \overline{\mathcal R} u := \overline{{\mathcal R} \bar u}\,.
 \end{equation}
 One gets easily that the operator $\overline{\mathcal R}$ has the matrix representation 
 \begin{equation}\label{operatore coniugato matrice}
\Big(\overline{{\mathcal R}_{- j}^{- j'}(\vphi)} \Big)_{j, j' \in \Z \setminus \{ 0 \}}\,, \qquad \vphi \in \T^\nu\,.
 \end{equation}
 An operator ${\mathcal R}$ is said to be real if it maps real-valued functions on real valued functions and it is easy to see that ${\mathcal R}$ is real if and only if ${\mathcal R} = \overline{\mathcal R}$.
 
 \noindent
 We define also the transpose operator ${\mathcal R}^T = {\mathcal R}(\vphi)^T$ by the relation 
 \begin{equation}\label{definizione operatore trasposto}
 \langle {\mathcal R}(\vphi)[u]\,,\, v \rangle_{L^2_x} = \langle u\,,\, {\mathcal R}(\vphi)^T[v] \rangle_{L^2_x}\,, \qquad \forall u, v \in L_0^2(\T_x)\,, \quad \forall \vphi \in \T^\nu
 \end{equation}
 where 
 \begin{equation}\label{prodotto scalare reale L2}
 \langle u, v \rangle_{L^2_x} := \int_{\T} u(x) v(x)\,,d x \,, \qquad \forall u, v \in L^2(\T_x)\,.
 \end{equation}
 Note that the operator ${\mathcal R}^T$ has the matrix representation 
 \begin{equation}\label{matrice RT}
 ({\mathcal R}^T)_j^{j'}(\vphi) = {\mathcal R}_{- j'}^{ - j}(\vphi)\,, \quad \forall j, j' \in \Z \setminus \{ 0 \}\,, \quad \forall \vphi \in \T^\nu\,.
 \end{equation}
 An operator ${\mathcal R}$ is said to be symmetric in ${\mathcal R} = {\mathcal R}^T$.
 
 \noindent
 We define also the adjoint operator ${\mathcal R}^* = {\mathcal R}(\vphi)^*$ by 
 \begin{equation}\label{operatore aggiunto cal R}
 \big( {\mathcal R}(\vphi)[u]\,,\, v \big)_{L^2_x} = \big(u\,,\, {\mathcal R}(\vphi)^* [v] \big)_{L^2_x}\,, \quad \forall u, v \in L^2_0(\T_x)\,, \quad \forall \vphi \in \T^\nu\,,
 \end{equation}
 where $ \big( \cdot\,,\, \cdot \big)_{L^2_x}$ is the scalar product on $L^2(\T)$, namely 
 \begin{equation}\label{prodotto scalare complesso}
 \big(u\,,\, v \big)_{L^2_x} := \langle u\,,\, \overline v \rangle_{L^2_x} = \int_{\T} u(x) \overline v(x)\,, d x\,, \quad \forall u, v \in L^2(\T_x)\,.
 \end{equation}
 An operator ${\mathcal R}$ is said to be self-adjoint if ${\mathcal R} = {\mathcal R}^*$.
 It is easy to see that ${\mathcal R}^* = \overline{\mathcal R}^T$ and its matrix representation is given by 
 $$
 ({\mathcal R}^*)_j^{j'}(\vphi) = \overline{{\mathcal R}_{j'}^{j}(\vphi)}\,, \quad \forall j, j' \in \Z \setminus \{ 0 \} \,, \quad \forall \vphi \in \T^\nu\,.
 $$
 In the following we also deal with smooth families of real operators $\vphi \mapsto G(\vphi) \in {\mathcal L}({ L}^2_0(\T_x, \R^2))$, of the form 
 \begin{equation}\label{operatore matriciale reale}
G(\vphi) := \begin{pmatrix}
 A(\vphi) & B(\vphi) \\
 C(\vphi) & D(\vphi)
 \end{pmatrix}\,, \qquad \vphi \in \T^\nu
 \end{equation}
 where $A(\vphi), B(\vphi) , C(\vphi), D(\vphi) \in {\mathcal L}(L^2_0(\T_x, \R))$, for all $\vphi \in \T^\nu$. Actually $G$ may be regarded as an operator in ${\mathcal L}(L^2_0(\T^{\nu + 1}, \R^2))$, according to the fact that $A, B, C, D$ are T\"oplitz in time operators. By \eqref{definizione operatore trasposto}, the transpose operator $G^T$ with respect to the bilinear form 
 \begin{equation}\label{prodotto scalare prodotto L2}
\langle (u_1, \psi_1)\,,\, (u_2, \psi_2) \rangle_{{ L}^2_x} := \langle u_1, u_2 \rangle_{L^2_x} + \langle \psi_1\,,\, \psi_2  \rangle_{L^2_x}\,, 
\end{equation} 
$ \forall (u_1, \psi_1), (u_2, \psi_2) \in {L}^2_0(\T_x, \R^2)$, is given by 
\begin{equation}\label{operatore trasposto matriciale}
G^T  = \begin{pmatrix}
A^T & C^T \\
B^T & D^T
\end{pmatrix}\,.
\end{equation}
Then it is easy to verify that $G$ is symmetric, i.e. $G = G^T$ if and only if $A = A^T$, $B = C^T$, $D = D^T$. It is also convenient to regard the real operator $G$ in the complex variables
\begin{equation}\label{prima volta variabili complesse 0}
z := \frac{u + \ii \psi}{\sqrt{2}}\,, \qquad \overline z = \frac{u - \ii \psi}{\sqrt{2}}\,,
\end{equation}
\begin{equation}\label{prima volta variabili complesse 1}
 u = \frac{z + \overline z}{\sqrt{2}}\,, \qquad \psi = \frac{z - \overline z}{\ii \sqrt{2}}\,.
\end{equation}
The transformed operator ${\mathcal R}$ has the form 
\begin{align}\label{operatore trasformato in variabili complesse}
& {\mathcal R} = \begin{pmatrix}
{ R}_1 & {R}_2 \\
\overline{ R}_2 & \overline{R}_1
\end{pmatrix}\,, \\
&  \quad R_1 := \frac{A+ D  - \ii (B - C)}{2}\,, \quad R_2 := \frac{A - D + \ii (B + C)}{2}\,. \nonumber
\end{align}
Note that the operator ${\mathcal R}$ satisfies
\begin{equation}\label{invarianza sottospazio reale operatori}
{\mathcal R} : {\bf L}^2_0(\T^{\nu + 1}) \to {\bf L}^2_0(\T^{\nu + 1})\,, \quad {\mathcal R}(\vphi) :{\bf L}^2_0(\T_x) \to {\bf L}^2_0(\T_x)\,, \quad \forall \vphi \in \T^\nu
\end{equation}
where ${\bf L}^2_0(\T^{\nu + 1})$, resp. ${\bf L}^2_0(\T_x)$ are the real subspaces of $L^2_0(\T^{\nu + 1}, \C^2)$, resp. $L^2_0(\T_x, \C^2)$ defined as
\begin{align}
& {\bf L}^2_0(\T^{\nu + 1}) :=\big\{  (z, \overline z) : z \in L^2_0(\T^{\nu + 1}, \C) \big\}\,, \label{sottospazio reale z bar z} \\
&   {\bf L}^2_0(\T_x) := \big\{ (z, \overline z) : z \in L^2_0(\T_x, \C) \big\}\,. \label{sottospazio reale z bar z1}
\end{align}
For the sequel, we also introduce for any $s \geq 0$, the real subspaces of $H^s_0(\T^{\nu + 1}, \C^2)$ and $H^s_0(\T_x, \C^2)$
\begin{align}
& {\bf H}^s_0(\T^{\nu + 1}) := H^s_0(\T^{\nu + 1}, \C^2) \cap {\bf L}^2_0(\T^{\nu + 1})\,, \label{definizione bf H s0} \\ 
& {\bf H}^s_0(\T_x) := H^s_0(\T_x, \C^2) \cap {\bf L}^2_0(\T_x)\,. \label{definizione bf H s01}
\end{align}
 \subsection{Hamiltonian formalism}\label{formalismo Hamiltoniano} 
We define the symplectic form ${\mathcal W}$ as
\begin{equation}\label{forma simplettica reale}
{\mathcal W}[ {\bf u}_1, {\bf u}_2 ] := \langle {\bf u}_1, J {\bf u}_2 \rangle_{{ L}^2_x}\,, \quad J = \begin{pmatrix}
0 & 1 \\
- 1 & 0
\end{pmatrix}\,, 
\end{equation}
for all ${\bf u_1}, {\bf u}_2 \in {L}^2_0(\T_x, \R^2 )\,.$
\begin{definition}\label{campo Hamiltoniano reale}
A $\vphi$-dependent linear vector field $X(\vphi) : {L}^2_0(\T_x, \R^2) \to { L}^2_0(\T_x, \R^2)$ is Hamiltonian, if $X(\vphi) = J G(\vphi)$, where $J$ is given in \eqref{forma simplettica reale}
and the operator $G$ is symmetric. 
 The operator 
$$
{\mathcal L} = \omega \cdot \partial_\vphi {\mathbb I}_2 - J G(\vphi) : { H}_0^1(\T^{\nu + 1}, \R^2) \to { L}^2_0(\T^{\nu + 1}, \R^2)\,, \quad {\mathbb I}_2 := \begin{pmatrix}
{\rm Id}_0 & 0 \\
0 & {\rm Id}_0
\end{pmatrix}
$$
where ${\rm Id}_0 : L^2_0(\T^{\nu + 1}) \to L^2_0(\T^{\nu + 1})$ is the identity, is called Hamiltonian operator. 
\end{definition}
\begin{definition}\label{trasformazione simplettica reale}
A $\vphi$-dependent map $\Phi(\vphi) : { L}^2_0(\T_x, \R^2) \to { L}^2_0(\T_x, \R^2) $ is symplectic if for any $\vphi \in \T^\nu$, for any ${\bf u}_1, {\bf u}_2 \in {L}^2_0(\T_x, \R^2) $, 
$$
{\mathcal W}[\Phi(\vphi){\bf u}_1\,,\, \Phi(\vphi){\bf u}_2] = {\mathcal W}[{\bf u}_1, {\bf u}_2]\,,
$$
or equivalently $\Phi(\vphi)^T J \Phi(\vphi) = J$, for all $\vphi \in \T^\nu$. 
\end{definition}
Under a symplectic transformation $\Phi = \Phi(\vphi)$, assuming that the map $\vphi \in \T^\nu \mapsto \Phi(\vphi ) \in {\mathcal L}({ L}^2_0(\T_x, \R^2))$ is differentiable, a linear Hamiltonian operator ${\mathcal L} = \omega \cdot \partial_\vphi {\mathbb I}_2 - J G(\vphi)$ transforms into the operator ${\mathcal L}_+ = \Phi^{- 1} {\mathcal L} \Phi = \omega \cdot \partial_\vphi {\mathbb I}_2 - J G_+(\vphi)$ with
\begin{equation}\label{coniugazione campo hamiltoniano reale}
 \quad G_+(\vphi) := \Phi(\vphi)^T G(\vphi) \Phi(\vphi) + \Phi(\vphi)^T J \omega \cdot \partial_\vphi \Phi(\vphi)\,.
\end{equation}
Note that for all $\vphi \in \T^\nu$, $G_+(\vphi)$ is symmetric, because $G(\vphi)$ is symmetric and $\omega \cdot \partial_\vphi [\Phi(\vphi)^T] J \Phi(\vphi) + \Phi(\vphi)^T J \omega \cdot \partial_\vphi \Phi(\vphi) = 0$ for all $\vphi \in \T^\nu$ and then ${\mathcal L}_+$ is still a Hamiltonian operator. Actually the conjugation \eqref{coniugazione campo hamiltoniano reale} can be interpreted also from a dynamical point of view. Indeed, consider the quasi-periodically forced linear Hamiltonian PDE 
\begin{equation}\label{interpretazione dinamica 1}
\partial_t {\bf h} = J G(\omega t) {\bf h}\,, \quad t \in \R\,, \quad \omega \in \R^\nu\,.
\end{equation}
Under the change of coordinates ${\bf h} = \Phi(\omega t) {\bf v}$, the above PDE is transformed into the equation
 \begin{equation}\label{interpretazione dinamica 2}
 \partial_t {\bf v} = J G_+(\omega t) {\bf v}\,
 \end{equation}
 which is still a linear Hamiltonian PDE. 
  
\subsubsection{Hamiltonian formalism in complex coordinates}\label{sezione formalismo hamiltoniano}
In this section we explain how the real Hamiltonian structure described above, reads in the complex coordinates introduced in \eqref{prima volta variabili complesse 0}, \eqref{prima volta variabili complesse 1}.
According to \eqref{operatore trasformato in variabili complesse}, under the change of coordinates \eqref{prima volta variabili complesse 0}, \eqref{prima volta variabili complesse 1}, a linear Hamiltonian vector field $J G(\vphi)$, transforms into 
\begin{equation}\label{operatore Hamiltoniano coordinate complesse}
{\mathcal R}(\vphi) = - \ii \begin{pmatrix}
{ R}_1(\vphi) &  R_2(\vphi) \\
- \overline{R_2(\vphi)} & - \overline{R_1(\vphi)}
\end{pmatrix}\,,
\end{equation}
where the operators $R_i = R_i(\vphi)$, $i = 1, 2$ are defined as
\begin{equation}\label{R1 R2 operatore Hamiltoniano complesso}
R_1:= \frac{A + D - \ii B+ \ii B^T}{2}\,, \quad R_2 := \frac{ A - D + \ii B + \ii B^T}{2}\,
\end{equation}
(recall that the operator $\overline R$ is defined in \eqref{definizione operatore coniugato}). Note that the operators $R_1(\vphi)$, $R_2(\vphi)$ are linear operators acting on complex valued $L^2$ functions $L^2_0(\T_x, \C)$, moreover since $G(\vphi)$ is symmetric, $A(\vphi) = A(\vphi)^T$, $B(\vphi) = C(\vphi)^T$, $D(\vphi) = D(\vphi)^T$, then it turns out that  
\begin{equation}\label{condizione R1 R2 campo hamiltoniano complesso}
R_1(\vphi) = R_1(\vphi)^*\,, \qquad R_2(\vphi) = R_2(\vphi)^T\,, \qquad \forall \vphi \in \T^\nu\,.
\end{equation}
Since the operator ${\mathcal R}$ in \eqref{operatore Hamiltoniano coordinate complesse} has the form \eqref{operatore trasformato in variabili complesse}, it satisfies the property \eqref{invarianza sottospazio reale operatori}. Furthermore, one has that ${\mathcal R}(\vphi)$ is the linear Hamiltonian vector field associated to the real Hamiltonian
\begin{equation}\label{generica hamiltoniana quadratica reale nel complesso esplicita}
{\mathcal H}({\bf z}) := \langle {\bf G}(\vphi)[{\bf z}]\,,\, {\bf z} \rangle_{L^2_x} \,, \quad {\bf G}(\vphi) := \begin{pmatrix}
\overline{R_2(\vphi)} & \overline{R_1(\vphi)} \\
R_1(\vphi) & R_2(\vphi)
\end{pmatrix} \,, 
\end{equation}
namely
\begin{equation}\label{generica hamiltoniana quadratica reale nel complesso}
{\mathcal H}(z,  \overline z) = \int_{\T} {R}_1(\vphi)[z]\overline z\, d x+ \frac12 \int_{\T} { R}_2(\vphi)[z]  z \, dx + \frac12 \int_\T \overline{{ R}_2(\vphi)}[ \overline z]\, \overline z\,d x\,.
\end{equation}
Indeed, ${\bf G}(\vphi)$ is symmetric, since by \eqref{condizione R1 R2 campo hamiltoniano complesso}, $\overline R_1^T = R_1^* = R_1$ and $R_1^T = \overline R_1$, then 
\begin{equation}\label{campo hamiltoniano complesso}
{\mathcal R}(\vphi)[{\bf z}] = - \ii J \nabla_{\bf z} {\mathcal H}({\bf z})= - \ii J {\bf G}(\vphi)[{\bf z}]\,, \quad {\bf z}  \in {\bf L}^2_0(\T_x)\,, 
\end{equation}
where $\nabla_{\bf z} {\mathcal H} := (\nabla_z {\mathcal H}, \nabla_{\bar z} {\mathcal H})$ with
$$
\nabla_z {\mathcal H} = \frac{1}{\sqrt{2}} (\nabla_\eta {\mathcal H} - \ii \nabla_\psi {\mathcal H})\,, \quad \nabla_{\bar z} {\mathcal H} := \overline{\nabla_z {\mathcal H}} = \frac{1}{\sqrt{2}} (\nabla_\eta {\mathcal H} + \ii \nabla_\psi {\mathcal H})
$$
(recall \eqref{prima volta variabili complesse 0}, \eqref{sottospazio reale z bar z1}). The symplectic form ${\mathcal W}$ in \eqref{forma simplettica reale}, reads in the complex coordinates \eqref{prima volta variabili complesse 0} as  
\begin{equation}\label{forma simplettica coordinate complesse}
{\bf \Gamma}[{\bf z}_1, {\bf z}_2] =  \ii \int_{\T} (z_1  \overline z_2 -   \overline z_1 z_2)\, dx  = \ii \langle {\bf z}_1 \,,\, J {\bf z}_2\rangle_{L^2_x}\,, \quad \forall {\bf z}_1, {\bf z}_2 \in { \bf L}^2_0(\T_x)\,.
\end{equation}
\begin{definition}\label{definizione mappa simplettica complessa}
Let $\Phi_i = \Phi_i(\vphi)$, $\vphi \in \T^\nu$, $i = 1, 2$ be $\vphi$-dependent families of linear operators $L^2_0(\T_x, \C) \to L^2_0(\T_x, \C)$. We say that the map 
$$
\Phi(\vphi) = \begin{pmatrix}
\Phi_1(\vphi) & \Phi_2(\vphi) \\
\overline{\Phi_2(\vphi)} & \overline{\Phi_1(\vphi)}
\end{pmatrix}\,, \qquad \vphi \in \T^\nu
$$ 
is symplectic if 
$$
{\bf \Gamma}[\Phi(\vphi)[{\bf z}_1], \Phi(\vphi)[{\bf z}_2]] = {\bf \Gamma}[{\bf z}_1, {\bf z}_2]\,, \quad \forall {\bf z}_1, {\bf z}_2 \in { \bf L}^2_0(\T_x)\,, \quad \forall \vphi \in \T^\nu\,
$$
or equivalently $\Phi(\vphi)^T J \Phi(\vphi) = J$, for all $\vphi \in \T^\nu$. 
\end{definition}
It is well known that if ${\mathcal R}(\vphi)$ is an operator of the form \eqref{operatore Hamiltoniano coordinate complesse}, \eqref{condizione R1 R2 campo hamiltoniano complesso}, namely by \eqref{campo hamiltoniano complesso} ${\mathcal R}(\vphi)$ is a linear Hamiltonian vector field associated to the quadratic Hamiltonian ${\mathcal H}$ in \eqref{generica hamiltoniana quadratica reale nel complesso}, the operators $ {\rm exp}(\pm{\mathcal R}(\vphi))$ are symplectic maps. 

\noindent
\begin{definition}\label{operatore Hamiltoniano coordinate complesse}
If ${\mathcal R}(\vphi)$ is a Hamiltonian vector field like in \eqref{operatore Hamiltoniano coordinate complesse}, \eqref{condizione R1 R2 campo hamiltoniano complesso}, we define the Hamiltonian operator in complex coordinates as 
$$
{\mathcal L} = \omega \cdot \partial_\vphi {\mathbb I}_2 - {\mathcal R}(\vphi) = \omega \cdot \partial_\vphi {\mathbb I}_2 + \ii J {\bf G}(\vphi)  : {\bf H}^1_0(\T^{\nu + 1}) \to {\bf L}^2_0(\T^{\nu + 1}) \,.
$$
\end{definition}

\noindent
Under the action of a smooth family of symplectic map $\Phi(\vphi)$, $\vphi \in \T^\nu$, a Hamiltonian operator ${\mathcal L}$ transforms into the Hamiltonian operator ${\mathcal L}_+ = \Phi^{- 1}{\mathcal L} \Phi = \omega \cdot \partial_\vphi {\mathbb I}_2 + \ii J {\bf G}_+(\vphi)$ where
$$
 {\bf G}_+(\vphi) := \Phi(\vphi)^T {\bf G}(\vphi) \Phi(\vphi) + \Phi(\vphi)^T J \omega \cdot \partial_\vphi \Phi(\vphi)\,, \quad \forall \vphi \in \T^\nu\,.
$$
Note that the operator ${\bf G}_+(\vphi)$ is symmetric and it has the same form as ${\bf G}(\vphi)$ in \eqref{generica hamiltoniana quadratica reale nel complesso esplicita}.
Arguing as in \eqref{interpretazione dinamica 1}, \eqref{interpretazione dinamica 2}, under the transformation ${\bf v} = \Phi(\omega t){\bf h}$, the PDE 
\begin{equation}\label{interpretazione dinamica 5}
\partial_t {\bf h} = - \ii J {\bf G}(\omega t) {\bf h}\,, \qquad \omega \in \R^\nu\,, \quad t \in \R\,,
\end{equation}
transforms into the PDE 
\begin{equation}\label{interpretazione dinamica 6}
\partial_t {\bf h} = - \ii J {\bf G}_+(\omega t) {\bf h}\,. 
\end{equation}
In the following, we will consider also quasi-periodic reparametrizations of time, namely operators of the form 
 $$
 {\mathcal A}{\bf h}(\vphi, x) = {\bf h}(\vphi + \omega \alpha(\vphi), x)\,,
 $$
 where $\alpha : \T^\nu \to \R$ is a sufficiently smooth function and such that $\| \alpha\|_{{\mathcal C}^1}$ is sufficiently small. The transformation ${\mathcal A}$ is invertible and its inverse ${\mathcal A}^{- 1}$ has the form 
 $$
 {\mathcal A}^{- 1}{\bf h}(\vartheta, x) = {\bf h}(\vartheta + \omega \tilde \alpha (\vartheta), x)\,
 $$ 
 where $\vartheta \mapsto \vartheta + \omega \tilde \alpha(\vartheta)$ is the inverse diffeomorphism of $\vphi \mapsto \vphi + \omega \alpha(\vphi)$. 
The conjugated operator is ${\mathcal A}^{- 1}{\mathcal L} {\mathcal A} = \rho {\mathcal L}_+$, where ${\mathcal L}_+ = \omega \cdot \partial_\vphi + \ii J {\bf G}_+(\vartheta)$ with
\begin{equation}\label{interpretazione dinamica 3}
 \rho(\vartheta) := {\mathcal A}^{- 1}[1 + \omega \cdot \partial_\vphi \alpha](\vartheta)\,, \quad {\bf G}_+(\vartheta) := \frac{1}{\rho(\vartheta)} {\bf G}(\vartheta + \omega \tilde \alpha(\vartheta))\,.
 \end{equation}
Note that ${\mathcal L}_+$ is still a Hamiltonian operator. From a dynamical point of view, under the reparametrization of time 
 $$
 \tau = t + \alpha(\omega t)\,, \qquad t = \tau + \tilde \alpha(\omega \tau)\,,
 $$
  setting ${\bf v} (t) :=  A(\omega t) {\bf h} := {\bf h}(t + \alpha(\omega t), x)$, the PDE \eqref{interpretazione dinamica 5} is transformed into 
  \begin{equation}\label{interpretazione dinamica 4}
  \partial_\tau {\bf v} = - \ii J {\bf G}_+(\omega \tau) {\bf v}\,. 
  \end{equation}
\subsection{$2 \times 2$ block representation of linear operators}\label{sezione rappresentazione 2 per 2}
We may regard a T\"oplitz in time operator given by \eqref{action toplitz operator} as a $2 \times 2$ block matrix 
\begin{equation}\label{notazione a blocchi}
\Big( {\bf R}_j^{j'}(\ell - \ell')\Big)_{\begin{subarray}{c}
\ell, \ell' \in \Z^\nu \\
j, j' \in \N
\end{subarray}}\,, 
\end{equation}
where for all $\ell \in \Z^\nu$, $j, j' \in \N$ the $2 \times 2$ matrix ${\bf R}_j^{j'}(\ell)$ is defined by  
\begin{equation}\label{definizione blocco operatore}
{\bf R}_j^{j'}(\ell) := \begin{pmatrix}
{\mathcal R}_j^{j'}(\ell) & {\mathcal R}_j^{- j'}(\ell) \\
{\mathcal R}_{- j}^{j'}(\ell) & {\mathcal R}_{- j}^{- j'}(\ell)
\end{pmatrix}\,.
\end{equation}
The $2 \times 2$ matrix ${\bf R}_j^{j'}(\ell)$ can be regarded as a linear operator in ${\mathcal L}({\bf E}_{j'}, {\bf E}_j)$, where for all $j \in \N$, the two dimensional space ${\bf E}_j$ is defined as 
\begin{equation}\label{bf E alpha}
{\bf E}_j := {\rm span} \{ e^{\ii j x}, e^{- \ii j x}\}\,.
\end{equation}
Note that for any $j \in \N$, the finite dimensional space ${\bf E}_j$ is the eigenspace of the operator $- \partial_{xx}$ corresponding to the eigenvalue $j^2$. 
We identify the space ${\mathcal L}({\bf E}_{j'}, {\bf E}_j)$ of the linear operators from ${\bf E}_{j'}$ onto ${\bf E}_j$ with the space of the $2 \times 2$ matrices of their Fourier coefficients, namely 
\begin{equation}\label{spazio matrici blocchi}
 {\mathcal L}({\bf E}_{j'}, {\bf E}_j) \simeq \Big\{ M = \Big(M_k^{k'} \Big)_{\begin{subarray}{c}
k = \pm j \\
k' = \pm j'
\end{subarray}} \Big\} \simeq {\rm Mat}(2 \times 2)\,.
\end{equation}
Indeed if $M \in {\mathcal L}({\bf E}_{j'}, {\bf E}_j)$, its action is given by 
\begin{equation}\label{azione blocco finito dimensionale su funzioni}
M u (x)= \sum_{\begin{subarray}{c}
k = \pm j \\
k' = \pm j'
\end{subarray}} M_k^{k'} u_{k'} e^{\ii k  x}\,, \quad \forall u \in {\bf E}_{j'}\,, \quad u(x) = u_{j'} e^{\ii j' x} + u_{- j'} e^{- \ii j' x}\,. 
\end{equation}
If $j = j'$, we use the notation ${\mathcal L}({\bf E}_j) = {\mathcal L}({\bf E}_{j'}, {\bf E}_{j})$ and we denote by ${\bf I}_j$ the identity operator on the space ${\bf E}_j$, namely 
\begin{equation}\label{operatore identita su E alpha}
{\bf I}_j : {\bf E}_j \to {\bf E}_j\,, \qquad u \mapsto u\,.
\end{equation}

\noindent
According to \eqref{raggruppamento modi Fourier}, \eqref{notazione a blocchi}, \eqref{azione blocco finito dimensionale su funzioni}, we may write the action of a T\"oplitz in time operator on a function $h(\vphi, x)$ as 
\begin{equation}\label{azione operator Toplitz a blocchi}
{\mathcal R} h (\vphi, x) = \sum_{\begin{subarray}{c}
\ell, \ell' \in \Z^\nu \\
j, j' \in \N
\end{subarray}} {\bf R}_j^{j'}(\ell - \ell') [\widehat{\bf h}_{j'}(\ell')] e^{\ii \ell \cdot \vphi}\,.
\end{equation}
 We denote by $[{\mathcal R}]$ the $2 \times 2$ block-diagonal part of the operator ${\mathcal R}$, namely
\begin{equation}\label{notazione operatore diagonale a blocchi}
[{\mathcal R}] : = {\rm diag}_{j \in \N} {\bf R}_j^{j}(0)
\end{equation}
and its action on a function $h(\vphi, x)$ is given by 
$$
[{\mathcal R}] h(\vphi, x) = \sum_{\ell \in \Z^\nu\,,\, j \in \N} {\bf R}_j^j(0)[\widehat{\bf h}_j(\ell)] e^{\ii \ell \cdot \vphi}\,.
$$
If ${\bf R}_j^{j'}(\ell) = 0$, for any $(\ell, j, j') \neq (0, j, j)$, we have ${\mathcal R} = [{\mathcal R}]$ and we refer to such operators as $2 \times 2$ block-diagonal operators. 

\noindent
For any $M \in {\mathcal L}({\bf E}_{j'}, {\bf E}_{j})$, we define
the transpose operator $M^T \in {\mathcal L}({\bf E}_j, {\bf E}_{j'})$ by  
\begin{equation}\label{definizione matrice trasposta}
(M^T)_k^{k'} := M_{- k'}^{- k} \,, \quad k = \pm j'\,,\quad k' = \pm j\,,
\end{equation}
the conjugate operator $\overline M \in {\mathcal L}({\bf E}_{j'}, {\bf E}_j)$ by 
\begin{equation}\label{definizione matrice coniugata}
({\overline M})_k^{k'} := \overline M_{- k}^{ - k'}\,, \quad k = \pm j\,, \quad k' = \pm j'\,,
\end{equation}
the adjoint operator $M^* \in {\mathcal L}({\bf E}_j, {\bf E}_{j'})$ as 
\begin{equation}\label{definizione matrice aggiunta}
M^* := \overline M^T\,.
\end{equation}

Given an operator $A \in {\mathcal L}({\bf E}_j)$, we define its trace as 
\begin{equation}\label{definizione traccia}
{\rm Tr}(A) :=  A_j^j  + A_{- j}^{- j}\,.
\end{equation}
It is easy to check that if $A, B \in {\mathcal L}({\bf E}_j)$, then 
\begin{equation}\label{proprieta traccia}
{\rm Tr}(A B) = {\rm Tr}(B A)\,.
\end{equation}
For all $j, j' \in \N$, the space ${\mathcal L}({\bf E}_{j'}, {\bf E}_j)$ defined in \eqref{spazio matrici blocchi}, is a Hilbert space equipped by the inner product given for any $X, Y \in {\mathcal L}({\bf E}_{j'}, {\bf E}_j)$ by
\begin{equation}\label{prodotto scalare traccia matrici}
\langle X, Y \rangle := {\rm Tr}(X Y^*)\,.
\end{equation}
This scalar product induces the $L^2$-norm 
\begin{equation}\label{norma L2 blocco}
\| X \| := \sqrt{{\rm Tr}(X X^*)} = \Big( \sum_{\begin{subarray}{c}
|k| = j \\
|k'| = j'
\end{subarray}} |X_k^{k'}|^2 \Big)^{\frac12}\,.
\end{equation}
Actually all the norms on the finite dimensional space ${\mathcal L}({\bf E}_{j'}, {\bf E}_j)$ are equivalent. 

\noindent
Given a linear operator ${\bf L} : {\mathcal L}({\bf E}_{j'}, {\bf E}_j) \to {\mathcal L}({\bf E}_{j'}, {\bf E}_j)$, we denote by $\| {\bf L}\|_{{\rm Op}(j, j')}$ its operatorial norm, when the space ${\mathcal L}({\bf E}_{j'}, {\bf E}_j)$ is equipped by the $L^2$-norm \eqref{norma L2 blocco}, namely
\begin{equation}\label{norma operatoriale su matrici alpha beta}
\| {\bf L}\|_{{\rm Op}(j, j')} := \sup\Big\{  \| {\bf L}(M) \| : M \in {\mathcal L}({\bf E}_{j'}, {\bf E}_j)\,, \quad \| M\| \leq 1\Big\}\,.
\end{equation} 
We denote by ${\bf I}_{j,  j'}$ the identity operator on ${\mathcal L}({\bf E}_{j'}, {\bf E}_j)$, namely 
\begin{equation}\label{operatore identita matrici alpha beta}
{\bf I}_{j,  j'} : {\mathcal L}({\bf E}_{j'}, {\bf E}_j) \to {\mathcal L}({\bf E}_{j'}, {\bf E}_j)\,, \qquad X \mapsto X\,.
\end{equation}
For any operator $A \in {\mathcal L}({\bf E}_j)$ we denote by $M_L(A) : {\mathcal L}({\bf E}_{j'}, {\bf E}_j) \to {\mathcal L}({\bf E}_{j'}, {\bf E}_j)$ the linear operator defined for any $X \in {\mathcal L}({\bf E}_{j'}, {\bf E}_j)$ as 
\begin{equation}\label{definizione moltiplicazione sinistra matrici}
M_L(A) X := A X\,.
\end{equation} 
Similarly, given an operator $B \in {\mathcal L}({\bf E}_{j'})$, we denote by $M_R(B) : {\mathcal L}({\bf E}_{j'}, {\bf E}_j) \to {\mathcal L}({\bf E}_{j'}, {\bf E}_j)$ the linear operator defined for any $X \in {\mathcal L}({\bf E}_{j'}, {\bf E}_j)$ as 
\begin{equation}\label{definizione moltiplicazione destra matrici}
M_R(B) X := X B\,.
\end{equation}
The following elementary estimates hold:
\begin{equation}\label{norma operatoriale ML MR}
\| M_L(A)\|_{{\rm Op}(j, j')} \leq \| A\|\,, \quad \| M_R(B)\|_{{\rm Op}(j, j')} \leq \| B\|\,.
\end{equation}
For any $j \in \N$, we denote by ${\mathcal S}({\bf E}_j)$, the set of the self-adjoint operators form ${\bf E}_j$ onto itself, namely
\begin{equation}\label{cal S E alpha}
{\mathcal S}({\bf E}_j) := \Big\{ A \in {\mathcal L}({\bf E}_j) : A = A^*\Big\}\,,
\end{equation}
which we identify with the set of the $2 \times 2$ self-adjoint matrices. Furthermore, for any $A \in {\mathcal L}({\bf E}_j)$ denote by ${\rm spec}(A)$ the spectrum of $A$. The following Lemma can be proved by using elementary arguments from linear algebra, hence the proof is omitted.
\begin{lemma}\label{properties operators matrices}
Let $A \in {\mathcal S}({\bf E}_j)$, $B \in {\mathcal S}({\bf E}_{j'})$, then the following holds: 

\noindent
$(i)$ The operators $M_L(A)$, $M_R(B)$ defined in \eqref{definizione moltiplicazione sinistra matrici}, \eqref{definizione moltiplicazione destra matrici} are self-adjoint operators with respect to the scalar product defined in \eqref{prodotto scalare traccia matrici}.

\noindent
$(ii)$ The spectrum of the operator $M_L(A) \pm M_R(B)$ satisfies 
$$
{\rm spec}\Big( M_L(A) \pm M_R(B) \Big) = \Big\{ \lambda \pm \mu : \lambda \in {\rm spec}(A)\,,\quad \mu \in {\rm spec}(B) \Big\}\,.
$$
\end{lemma}
We finish this Section by recalling some well known facts concerning linear self-adjoint operators on finite dimensional Hilbert spaces. Let ${\mathcal H}$ be a finite dimensional Hilbert space of dimension $n$ equipped by the inner product $( \cdot\,,\,\cdot )_{\mathcal H}$. For any self-adjoint operator $A : {\mathcal H} \to {\mathcal H}$, we order its eigenvalues as
\begin{equation}\label{spettro hilbert astratto}
{\rm spec}(A) := \big\{\lambda_1(A) \leq \lambda_2(A) \leq \ldots \leq \lambda_n(A)\big\}\,.
\end{equation}
\begin{lemma}\label{risultato astratto operatori autoaggiunti}
 Let ${\mathcal H}$ be a Hilbert space of dimension $n$. Then the following holds:

\noindent
$(i)$ Let $A_1, A_2 : {\mathcal H} \to {\mathcal H}$ be self-adjoint operators. Then their eigenvalues, ranked as in \eqref{spettro hilbert astratto}, satisfy the Lipschitz property 
$$
|\lambda_k(A_1) - \lambda_k(A_2)| \leq \| A_1 - A_2 \|_{{\mathcal L}({\mathcal H})}\,, \qquad \forall k = 1, \ldots, n\,.
$$

\noindent
$(ii)$ Let $A = \eta {\rm Id}_{\mathcal H} + B$, where $\eta \in \R$, ${\rm Id}_{\mathcal H} : {\mathcal H} \to {\mathcal H}$ is the identity and $B :{\mathcal H} \to {\mathcal H}$ is selfadjoint. Then 
$$
\lambda_k(A) = \eta + \lambda_k(B) \,, \qquad \forall k = 1, \ldots , n\,. 
$$ 

\noindent
$(iii)$ Let $A : {\mathcal H} \to {\mathcal H}$ be self-adjoint and assume that ${\rm spec}(A) \subset \R \setminus \{ 0 \}$. Then $A$ is invertible and its inverse satisfies
$$
\| A^{- 1}\|_{{\mathcal L}({\mathcal H})} = \dfrac{1}{\min_{k = 1, \ldots, n}|\lambda_k(A)|}\,.
$$
\end{lemma}

\subsection{Block-decay norm for linear operators}
In this Section, we introduce the block-decay norm for linear operators.  
Given a T\"oplitz in time operator ${\mathcal R}$ as in \eqref{action toplitz operator}, recalling its $2 \times 2$ block representation \eqref{notazione a blocchi}, \eqref{definizione blocco operatore}, we define its block-decay norm as  
\begin{equation}\label{decadimento Kirchoff}
| {\mathcal R} |_s := {\rm sup}_{j' \in \N} \Big( \sum_{\ell \in \Z^\nu\,,\,j \in \N} \langle \ell, j - j' \rangle^{2 s} \| {\bf R}_j^{j'}(\ell)\|^2\Big)^{\frac12}\,,
\end{equation}
where $\| \cdot \|$ is defined in \eqref{norma L2 blocco}. 
For a family of T\"oplitz in time operators ${\mathcal R} = {\mathcal R}(\omega) \in {\mathcal L}\big(H^s_0(\T^{\nu + 1})\big)$, $\omega \in \Omega_o$, given $\gamma > 0$, we define the norm 
\begin{equation}\label{norma decadimento lipschitz}
|{\mathcal R}|_s^{\Lipg} := |{\mathcal R}|_s^{\sup} + \gamma |{\mathcal R}|_s^{\lip}\,,
\end{equation}
$$
|{\mathcal R}|_s^{\sup} := \sup_{\omega \in \Omega_o} |{\mathcal R}(\omega)|_s\,, \quad |{\mathcal R}|_s^{\lip} := \sup_{\begin{subarray}{c}
\omega_1, \omega_2 \in \Omega_o \\
\omega_1 \neq \omega_2
\end{subarray}} \frac{|{\mathcal R}(\omega_1) - {\mathcal R}(\omega_2)|_s}{|\omega_1 - \omega_2|}\,.
$$
For families of linear operators ${\mathcal R}(\omega)$, $\omega \in \Omega_o$ of the form 
\begin{equation}\label{operatore matriciale decadimento}
{\mathcal R} = \begin{pmatrix}
{\mathcal R}_1 & {\mathcal R}_2 \\
\overline{\mathcal R}_2 & \overline{\mathcal R}_1
\end{pmatrix}\,, 
\end{equation}
where ${\mathcal R}_1 , {\mathcal R}_2 \in {\mathcal L}\big(H^s_0(\T^{\nu + 1}) \big)$ are T\"oplitz in time operators of the form \eqref{action toplitz operator}, 
we define 
\begin{equation}\label{norma decadimento operatore matriciale}
|{\mathcal R}|_s := {\rm max}\{|{\mathcal R}_1|_s, |{\mathcal R}_2|_s \}\,,\,\, |{\mathcal R}|_s^\Lipg := {\rm max}\{ |{\mathcal R}_1|_s^\Lipg , |{\mathcal R}_2|_s^\Lipg \}\,.
\end{equation}
 
\noindent
In the following, we state some properties of this norm. We prove such properties for operators ${\mathcal R} \in {\mathcal L}\big(H^s_0(\T^{\nu + 1}) \big)$. If ${\mathcal R}$ is an operator of the form \eqref{operatore matriciale decadimento} then the same statements hold with the obvious modifications. To state the following lemma we need the following definition. For all $m \in \R$ we define the operator $|D|^m$ as
\begin{equation}\label{definizione |D| m}
|D|^m(1) = 0\,, \qquad |D|^m (e^{\ii j  x}) = |j|^m e^{\ii j x}\quad \forall j \neq 0\,. 
\end{equation}
\begin{lemma}\label{elementarissimo decay}
$(i)$ The norm $| \cdot |_s$ is increasing, namely $|{\mathcal R}|_s \leq |{\mathcal R}|_{s'}$, for $s \leq s'$.

\noindent
$(ii)$ $|{\mathcal R}|_s \leq |{\mathcal R} |D||_s$ and the operator $\omega \cdot \partial_\vphi {\mathcal R}$ (see \eqref{definizione D omega cal R}) satisfies $|\omega \cdot \partial_\vphi {\mathcal R}|_s \leq |{\mathcal R}|_{s + 1}$.

\noindent
$(iii)$ For any $j \in \N$, the $2 \times 2$ block ${\bf R}_j^j(0)$ (see \eqref{definizione blocco operatore}) satisfies $\sup_{j \in \N} \| {\bf R}_j^j(0)\| \lessdot |{\mathcal R}|_{0}$, where $\| \cdot \|$ is defined in \eqref{norma L2 blocco}. Moreover the operator $[{\mathcal R}]$ defined by \eqref{notazione operatore diagonale a blocchi}, satisfies $|[{\mathcal R}]|_s \leq |{\mathcal R}|_s$. 

\noindent
$(iv)$ Items $(i)$-$(iii)$hold, replacing $|\cdot |_s$ by $|\cdot|_s^\Lipg$ and $\| \cdot \|$ by $\| \cdot \|^\Lipg$. 
\end{lemma}
\begin{proof}
The proof is elementary. It follows directly by the definitions \eqref{decadimento Kirchoff}, \eqref{norma decadimento lipschitz}, hence we omit it. 
\end{proof}
\begin{lemma}\label{interpolazione decadimento Kirchoff}
Let ${\mathcal R}$, ${\mathcal B}$ be operators of the form \eqref{action toplitz operator}. Then 
\begin{equation}\label{kartoffen}
|{\mathcal R} {\mathcal B} |_s \leq_s  |{\mathcal R}|_s |{\mathcal B}|_{s_0} + |{\mathcal R}|_{s_0} |{\mathcal B}|_s \,.
\end{equation}
If ${\mathcal R} = {\mathcal R}(\omega)$, ${\mathcal B}= {\mathcal B}(\omega)$ are Lipschitz with respect to the parameter $\omega \in \Omega_o \subseteq \Omega$, then the same estimate holds replacing $|\cdot |_s$ by $| \cdot |_s^\Lipg$. 
\end{lemma}
\begin{proof}
According to the matrix representations \eqref{notazione a blocchi}, \eqref{definizione blocco operatore}, the operator ${\mathcal R}{\mathcal B}$ has the $2 \times 2$ block representation 
$$
 {\mathcal R} {\mathcal B} = \Big(  [{\bf  R}{\bf B}]_j^{j'}(\ell - \ell')\Big)_{\begin{subarray}{c}
  \ell, \ell' \in \Z^\nu \\
 j, j' \in \N
 \end{subarray}}\,, \quad [{\bf R}{\bf B}]_j^{j'}(\ell) = \sum_{\begin{subarray}{c}
 \ell' \in \Z^\nu \\
 k \in \N
 \end{subarray}} {\bf R}_j^{k}(\ell - \ell') {\bf B}_{k}^{j'}(\ell')\,.
$$
By the Cauchy Schwartz inequality, $\| {\bf R}_j^{k}(\ell - \ell') {\bf B}_{k}^{j'}(\ell') \| \leq \| {\bf R}_j^{k}(\ell - \ell') \| \| {\bf B}_{k}^{j'}(\ell')\|$, then for all $j' \in \N$, we get  
\begin{align}
\sum_{\begin{subarray}{c}
\ell \in \Z^\nu \\
j\in \N 
\end{subarray}} \langle \ell, j - j' \rangle^{2 s} \|[{\bf R}{\bf B}]_j^{j'}(\ell) \|^2 & \leq \sum_{\begin{subarray}{c}
\ell \in \Z^\nu \\
j \in \N
\end{subarray}} \Big( \sum_{\begin{subarray}{c}
\ell' \in \Z^\nu \\
k \in \N
\end{subarray}} \langle \ell, j - j' \rangle^{s} \|{\bf R}_j^{k}(\ell - \ell') \| \| {\bf B}_{k}^{j'}(\ell') \| \Big)^2\,. \label{paris 0} 
\end{align}
Using that $\langle \ell, j - j' \rangle^s \leq_s \langle \ell - \ell', j - k \rangle^s + \langle \ell', k - j' \rangle^s$, we get $\eqref{paris 0} \leq_s (A) + (B)$ where 
\begin{equation}\label{(A)}
(A) := \sum_{\begin{subarray}{c}
\ell \in \Z^\nu \\
j \in \N
\end{subarray}} \Big( \sum_{\begin{subarray}{c}
\ell' \in \Z^\nu \\
k \in \N
\end{subarray}} \langle \ell - \ell', j - k \rangle^{s} \|{\bf R}_j^{k}(\ell - \ell') \| \| {\bf B}_{k}^{j'}(\ell') \| \Big)^2\,, 
\end{equation}
\begin{equation}\label{(B)}
(B) := \sum_{\begin{subarray}{c}
\ell \in \Z^\nu \\
j \in \N
\end{subarray}} \Big( \sum_{\begin{subarray}{c}
\ell' \in \Z^\nu \\
k  \in \N
\end{subarray}} \langle \ell',  k - j'  \rangle^{s} \|{\bf R}_j^{k}(\ell - \ell') \| \| {\bf B}_{k}^{j'}(\ell') \| \Big)^2\,.
\end{equation}
By the Cauchy-Schwartz inequality, using that since $s_0 = [(\nu + 1)/2] + 1 > (\nu + 1)/2$, the series $\sum_{\ell' \in \Z^\nu, k \in \N} \langle \ell', k - j'\rangle^{- 2 s_0}  = C(s_0)$,
one has
\begin{align}
(A) & \lessdot \sum_{\begin{subarray}{c}
\ell \in \Z^\nu \\
j \in \N
\end{subarray}}  \sum_{\begin{subarray}{c}
\ell' \in \Z^\nu \\
k \in \N
\end{subarray}} \langle \ell - \ell', j - k \rangle^{ 2s} \|{\bf R}_j^{k}(\ell - \ell') \|^2  \langle \ell', k - j' \rangle^{2 s_0}\| {\bf B}_{k}^{j'}(\ell') \|^2 \nonumber\\
& \lessdot \sum_{\begin{subarray}{c}
\ell' \in \Z^\nu \\
k \in \N
\end{subarray}} \langle \ell', k - j' \rangle^{2 s_0}\| {\bf B}_{k}^{j'}(\ell') \|^2 \sum_{\begin{subarray}{c}
\ell \in \Z^\nu \\
j \in \N
\end{subarray}} \langle \ell - \ell', j - k \rangle^{ 2s} \|{\bf R}_j^{k}(\ell - \ell') \|^2  \nonumber\\
& \lessdot    \sup_{j' \in \N} \sum_{\begin{subarray}{c}
\ell' \in \Z^\nu \\
k \in \N
\end{subarray}} \langle \ell', k - j' \rangle^{2 s_0}\| {\bf B}_{k}^{j'}(\ell') \|^2 \sup_{k \in \N} \sum_{\begin{subarray}{c}
\ell \in \Z^\nu \\
j \in \N
\end{subarray}} \langle \ell - \ell', j - k \rangle^{ 2s} \|{\bf R}_j^{k}(\ell - \ell') \|^2 \nonumber\\
& \stackrel{\eqref{decadimento Kirchoff}}{\lessdot} |{\mathcal R}|_s^2 |{\mathcal B}|_{s_0}^2\,.  \nonumber
\end{align}
By similar arguments, one gets $(B) \lessdot |{\mathcal R}|_{s_0}^2 |{\mathcal B}|_s^2$ and hence the claimed estimate follows by taking the supremum over $j' \in \N$ in \eqref{paris 0}. The estimate in $| \cdot |_s^\Lipg$, follows by applying the estimate \eqref{kartoffen} to 
$$
\frac{{\mathcal R}(\omega_1) {\mathcal B}(\omega_1) - {\mathcal R}(\omega_2) {\mathcal B}(\omega_2)}{\omega_1 - \omega_2} = \frac{({\mathcal R}(\omega_1) - {\mathcal R}(\omega_2)) {\mathcal B}(\omega_1)}{\omega_1 - \omega_2} + \frac{{\mathcal R}(\omega_2) ({\mathcal B}(\omega_1) - {\mathcal B}(\omega_2))}{\omega_1 - \omega_2}
$$
and passing to the sup for $\omega_1, \omega_2 \in \Omega_o$ with $\omega_1 \neq \omega_2$. 
\end{proof}
For all $n \geq 1$, iterating the estimate of Lemma \ref{interpolazione decadimento Kirchoff} we get  
\be\label{Mnab}
| {\mathcal R}^n |_{s_0} \leq [C(s_0)]^{n-1} | {\mathcal R} |_{s_0}^n\,,\quad
| {\mathcal R}^n |_{s} \leq n C(s)^ n |{\mathcal R}|_{s_0}^{n-1} | {\mathcal R} |_{s} \, , \ \forall s \geq s_0 \, ,
\ee
for some constant $C(s) > 0$, and the same bounds also hold for the norm $| \cdot |_s^\Lipg$ if ${\mathcal R} = {\mathcal R}(\omega)$ is Lipschitz with respect to the parameter $\omega$. 
\begin{lemma}\label{stime tame operatori decadimento Kirchoff}
Let ${\mathcal R}$ satisfy $|{\mathcal R}|_s < + \infty$, with $s \geq s_0$. Then for all ${ u} \in {H}^s_0(\T^{\nu + 1})$, the following estimate holds
$$
\| {\mathcal R} {u}\|_s \leq_s |{\mathcal R}|_s \| { u} \|_{s_0} + |{\mathcal R}|_{s_0} \| {u} \|_s\,.
$$
If ${\mathcal R} = {\mathcal R}(\omega)$, $u = u(\cdot, \omega)$ are Lipschitz with respect to the parameter $\omega \in \Omega_o \subseteq \R^\nu$, then the same estimate holds replacing $|\cdot  |_s$ by $| \cdot  |_s^\Lipg$ and $\| \cdot \|_s$ by $\| \cdot \|_s^\Lipg$. 
\end{lemma}
\begin{proof}
The proof is similar to the one of Lemma \ref{interpolazione decadimento Kirchoff}, hence it is omitted. 
\end{proof}
\begin{lemma}\label{decadimento moltiplicazione}
Let $a \in H^s(\T^{\nu})$. Then the multiplication operator ${\mathcal R} : h(\vphi, x) \mapsto a(\vphi) h(\vphi, x)$ satisfies 
$
|{\mathcal R}|_s \lessdot \| a \|_s\,.
$
If $a = a(\cdot; \omega)$ is a Lipschitz family in $H^s(\T^{\nu})$, then the same estimate holds, replacing $\| \cdot \|_s$ by $\| \cdot \|_s^\Lipg$ and $| \cdot |_s$ by $| \cdot |_s^\Lipg$. 
\end{lemma}
\begin{proof}
The operator ${\mathcal R}$ admits the $2 \times 2$-block representation 
$$
{\mathcal R} = \big( {\bf R}_j^j(\ell - \ell')\big)_{\begin{subarray}{c}
\ell, \ell' \in \Z^\nu \\
j \in \N
\end{subarray}}\,, \quad {\bf R}_j^{j}(\ell) := \widehat a(\ell) {\bf I}_j\,, \quad \forall \ell \in \Z^\nu\,, \quad \forall j \in \N\,
$$
(recall \eqref{operatore identita su E alpha}). Since $\| {\bf I}_j\| = \sqrt{2}$, by \eqref{decadimento Kirchoff}, one has $|{\mathcal R}|_s \stackrel{}{\lessdot}  \| a \|_s$. The estimate for $|{\mathcal R}|_s^\Lipg$ follows similarly.
\end{proof}
\begin{lemma}\label{lem:inverti}
Let $ \Phi ={\rm exp}(\Psi) $ with $\Psi := \Psi(\omega)$, depending in a Lipschitz way on the parameter $\omega \in \Omega_o \subseteq \R^\nu $, 
such that  $ | \Psi |D| |_{s_0}^\Lipg \leq 1 $, $|\Psi |D||_s^\Lipg < + \infty$, with $ s \geq s_0$. Then 
\be\label{PhINV}
\,| (\Phi^{\pm 1} - {\rm Id}) |D| |_s \leq_s | \Psi |D| |_s \, , \quad 
| (\Phi^{\pm 1} - {\rm Id}) |D| |_{s}^\Lipg \leq_s  | \Psi |D||_{s}^\Lipg \, . 
\ee
The differential $\partial_\Psi \Phi$ of the map $\Psi \mapsto \Phi^{\pm 1} = {\rm exp}(\pm \Psi)$ satisfies for any $|\Psi|_{s_0} \leq 1$ the estimate
\begin{equation}\label{derivata-inversa-Phi}
\vert \partial_\Psi \Phi^{\pm 1}[\widehat \Psi] \vert_{s} 
\leq_s 
\big(  \vert \widehat \Psi \vert_{s}  
+  |\Psi|_s  
\vert \widehat \Psi \vert_{s_0} \big) \, . 
\end{equation}
Moreover the map $\Phi_{\geq 2} = \Phi - {\rm Id} - \Psi$, satisfies
\begin{align}
& | \Phi_{\geq 2}|D| |_s \leq_s | \Psi |D| |_s |\Psi |D||_{s_0} \, , \label{PhINV2} \\
& | \Phi_{\geq 2}|D| |_{s}^\Lipg \leq_s  | \Psi |D||_{s}^\Lipg  | \Psi |D||_{s_0}^\Lipg\,, \label{PhINV22} \\
& \vert \partial_\Psi \Phi_{\geq 2}[\widehat \Psi] \vert_{s} 
\leq_s 
\big(  |\Psi|_{s_0}\vert \widehat \Psi \vert_{s}  
+  |\Psi|_s  
\vert \widehat \Psi \vert_{s_0} \big) \, . \label{derivata-inversa-Phi2}
\end{align}
\end{lemma}
\begin{proof} 
Let us prove the estimate \eqref{PhINV} for $\Phi$. 
We write
$$
\Phi - {\rm Id} =  \Psi + \sum_{k \geq 2} \frac{\Psi^k}{k !}\,.
$$
For any $k \geq 2$ one has
\begin{align}
|\Psi^k |D||_s & \stackrel{Lemma\,\ref{interpolazione decadimento Kirchoff}}{\leq} C(s)\Big(|\Psi^{k - 1}|_s |\Psi |D||_{s_0} + |\Psi^{k - 1}|_{s_0} |\Psi |D||_s \Big) \nonumber\\
& \stackrel{\eqref{Mnab}}{\leq_s} (k - 1) C(s)^k \Big( |\Psi|_s |\Psi|_{s_0}^{k - 2} |\Psi |D||_{s_0} + |\Psi|_{s_0}^{k - 1} |\Psi |D||_s\Big) \nonumber\\
& \stackrel{Lemma \,\ref{elementarissimo decay}-(ii)}{\leq} 2 (k - 1) C(s)^k |\Psi |D||_s |\Psi |D||_{s_0}^{k - 1} \nonumber\\
&  \stackrel{| \Psi |D| |_{s_0} \leq 1}{\leq_s}  2 (k - 1) C(s)^k |\Psi |D||_s\,. \label{cacca primaria}
\end{align} 
Hence 
\begin{align}
|(\Phi - {\rm Id})|D||_{s}&  \stackrel{\eqref{cacca primaria}}{\leq} |\Psi |D||_s\Big(1 + 2 \sum_{k \geq 2} \frac{(k - 1) C(s)^k}{k !} \Big) \leq_s |\Psi |D||_s\,. \nonumber
\end{align}
The same inequatity holds for the inverse $\Phi^{- 1} = {\rm exp}(- \Psi)$.

\noindent
Now let us prove the estimate \eqref{derivata-inversa-Phi}.  For any $k \geq 1$, one has that 
$$
\partial_\Psi(\Psi^k)[\widehat \Psi] = \sum_{i + j = k - 1} \Psi^i \widehat \Psi \Psi^j\,.
$$
For any $i + j = k - 1$
\begin{align}
| \Psi^i \widehat \Psi \Psi^j |_{ s} & \stackrel{Lemma \,\ref{interpolazione decadimento Kirchoff}}{\leq} C(s)^2 \Big( | \Psi^i |_{ s}  | \widehat \Psi |_{s_0} |\Psi^j |_{s_0} + | \Psi^i |_{s_0}  | \widehat \Psi |_{s} | \Psi^j |_{s_0}   + | \Psi^i |_{ s_0} | \widehat \Psi |_{s_0} | \Psi^j |_{ s}  \Big) \nonumber\\
& \stackrel{\eqref{Mnab}}{\leq} 2 k C(s)^{k + 1} \Big(  | \widehat \Psi |_{ s} + | \Psi |_{ s} | \widehat \Psi |_{s_0}   \Big)\,. \label{derivata Psi n}
\end{align}
Hence 
\begin{align}
| \partial_\Psi \Phi [\widehat \Psi] |_{ s} & \leq   \sum_{k \geq 1} \dfrac{| \partial_\Psi (\Psi^k)[\widehat \Psi] |_{s}}{k!}  \stackrel{\eqref{derivata Psi n}}{\leq}  \sum_{k \geq 1} \dfrac{2 k C(s)^{k + 1}}{k!} \Big(|\widehat \Psi|_s +  | \Psi |_{s}| \widehat \Psi|_{s_0} 
 \Big) \nonumber\\
& \leq_s    | \widehat \Psi |_{ s} + | \Psi |_{ s} | \widehat \Psi |_{s_0}\,,
\end{align}
which is the estimate \eqref{derivata-inversa-Phi}. The estimates \eqref{PhINV2}-\eqref{derivata-inversa-Phi2}, can be proved arguing as above, using that $\Phi_{\geq 2} = \sum_{k \geq 2} \frac{\Psi^k}{k !}$.
\end{proof}

\bigskip

\noindent
Given $ N \in \N $, we define the smoothing operator $\Pi_N {\mathcal R}$, for any operator ${\mathcal R}$ as in \eqref{action toplitz operator}
\be\label{SM}
\big(\Pi_N {\mathcal R} \big)_j^{j'}(\ell - \ell'):= 
\begin{cases}
{\mathcal R}_j^{j'}(\ell - \ell') \qquad \, {\rm if} \  | \ell - \ell'| \leq N  \\
0 \quad \qquad \qquad {\rm otherwise,} 
\end{cases}
\ee
or equivalently, using the block-matrix representation \eqref{notazione a blocchi}, \eqref{definizione blocco operatore} 
\be\label{SM-block-matrix}
\big(\Pi_N {\bf R} \big)_j^{j'}(\ell - \ell'):= 
\begin{cases}
{\bf R}_j^{j'}(\ell - \ell') \qquad \, {\rm if} \  | \ell - \ell'| \leq N  \\
0 \quad \qquad \qquad {\rm otherwise,} 
\end{cases}
\ee

\begin{lemma}\label{lemma smoothing decay}
The operator $ \Pi_N^\bot := {\rm Id} - \Pi_N $ satisfies 
\be\label{smoothingN}
| \Pi_N^\bot {\mathcal R} |_{s} \leq N^{- \mathtt b} |  {\mathcal R} |_{s+\mathtt b} \, , \quad 
| \Pi_N^\bot {\mathcal R} |_{s}^\Lipg \leq N^{- \mathtt b} |  {\mathcal R} |_{s+\mathtt b}^\Lipg \, , 
\quad \mathtt b \geq 0,
\ee
where in the second inequality ${\mathcal R}$ is Lipschitz with respect to the parameter $\omega \in \Omega_o \subseteq \R^\nu$. 
\end{lemma}
\begin{proof}
The proof follows easily by the definitions \eqref{decadimento Kirchoff}, \eqref{norma decadimento lipschitz} and hence it is omitted. 
\end{proof}

\begin{lemma}\label{decadimento operatori di proiezione}
Let us define the operator 
\begin{equation}\label{operatore forma buona resto}
{\mathcal R} h (\vphi, x):= q(\vphi, x) \int_{\T } g(\vphi, x) h (\vphi, x)\, dx \,, \quad q\,, g \in H^s_0(\T^{\nu + 1})\,, \quad s \geq s_0\,.
\end{equation}
Then 
\begin{equation}\label{kartoffen 2}
|{\mathcal R}|_s \leq_s \| g \|_{s_0} \| q \|_{s} + \| g \|_{s } \| q \|_{s_0}\,.
\end{equation}
Moreover if the functions $g$ and $q$ are Lipschitz with respect to the parameter $\omega \in \Omega_o \subseteq \R^\nu$, then the same estimate holds replacing $| \cdot |_s$ by $|\cdot |_s^\Lipg$ and $\| \cdot \|_s$ by $\| \cdot \|_s^\Lipg$. 
\end{lemma}
\begin{proof}
A direct calculation shows that for all $\ell \in \Z^\nu$ and for all $k, k' \in \Z \setminus \{ 0 \}$
$$
{\mathcal R}_k^{k'}(\ell) = \sum_{\ell' \in \Z^\nu} \widehat q_{ k}(\ell - \ell') \widehat g_{ - k'}(\ell')\,.
$$
Using definition \eqref{norma L2 blocco} we get that for all $\ell \in \Z^\nu$, $j, j' \in \N$,   
\begin{align}
\| {\bf R}_j^{j'}(\ell)\|^2 & = \sum_{\begin{subarray}{c}
k = \pm j \\
k' = \pm j'
\end{subarray}} |{\mathcal R}_k^{k'}(\ell)|^2 \leq \Big( \sum_{\begin{subarray}{c}
k = \pm j \\
k' = \pm j'
\end{subarray}} |{\mathcal R}_k^{k'}(\ell)| \Big)^2  \nonumber\\
& \leq \Big( \sum_{\ell' \in \Z^\nu} \sum_{\begin{subarray}{c}
k = \pm j \\
k' = \pm j'
\end{subarray}} |\widehat q_{ k}(\ell - \ell')| |\widehat g_{ - k'}(\ell')| \Big)^2. \nonumber\\
& \lessdot \Big( \sum_{\ell' \in \Z^\nu} \| \widehat{\bf q}_{j}(\ell - \ell') \|_{L^2} \| \widehat{\bf g}_{ j'}(\ell')  \|_{L^2} \Big)^2 \label{norma L2 blocco proiettore}
\end{align}
where the last inequality holds, since, recalling \eqref{bf h ell alpha}, for any $k = \pm j$, $k'= \pm j'$, $|\widehat q_k(\ell - \ell')| \leq \| \widehat{\bf q}_j(\ell - \ell') \|_{L^2}$ and $|\widehat g_{- k'}(\ell')| \leq \| \widehat{\bf g}_{j'}(\ell')\|_{L^2}$. 
Now for all $j' \in \N$, 
\begin{align}
\sum_{\begin{subarray}{c}
\ell \in \Z^\nu \\
j \in \N
\end{subarray}} \langle \ell, j -  j' \rangle^{2 s} \| {\bf R}_j^{j'}(\ell) \|^2 & \stackrel{\eqref{norma L2 blocco proiettore}}{\leq} \sum_{\begin{subarray}{c}
\ell \in \Z^\nu \\
j \in \N
\end{subarray}} \Big( \sum_{\ell' \in \Z^\nu} \langle \ell, j - j' \rangle^{ s}   \| \widehat{\bf q}_{ j}(\ell - \ell')\|_{L^2} \| \widehat{\bf g}_{ j'}(\ell') \|_{L^2} \Big)^2\,. \label{londra 0}
\end{align}
Using that $\langle \ell, j - j' \rangle^{ s} \leq_s \langle \ell - \ell' , j\rangle^s + \langle \ell', j' \rangle^s $, one gets $\eqref{londra 0} \leq_s (A) +(B)$, where 
\begin{equation}\label{(A) proiettore}
(A) := \sum_{\begin{subarray}{c}
\ell \in \Z^\nu \\
j \in \N
\end{subarray}} \Big( \sum_{\ell' \in \Z^\nu} \langle \ell - \ell', j  \rangle^{ s}   \| \widehat{\bf q}_{ j}(\ell - \ell')\|_{L^2} \| \widehat{\bf g}_{ j'}(\ell') \|_{L^2} \Big)^2\,, \quad 
\end{equation}
\begin{equation}\label{(B) proiettore}
(B) := \sum_{\begin{subarray}{c}
\ell \in \Z^\nu \\
j \in \N
\end{subarray}} \Big( \sum_{\ell' \in \Z^\nu} \langle \ell', j' \rangle^{ s}   \| \widehat{\bf q}_{ j}(\ell - \ell')\|_{L^2} \| \widehat{\bf g}_{ j'}(\ell') \|_{L^2} \Big)^2\,.
\end{equation}
By the Cauchy-Schwartz inequality, using that $\sum_{\ell' \in \Z^\nu} \langle \ell' \rangle^{- 2 s_0} = C(s_0)$ (recall that $s_0 = [(\nu + 1)/2] + 1> (\nu + 1)/2$), one gets 
\begin{align}
(A) & \leq_s \sum_{\begin{subarray}{c}
\ell \in \Z^\nu \\
j \in \N
\end{subarray}}  \sum_{\ell' \in \Z^\nu} \langle \ell - \ell', j  \rangle^{ 2 s}   \| \widehat{\bf q}_{ j}(\ell - \ell')\|_{L^2}^2  \langle \ell' \rangle^{2 s_0} \| \widehat{\bf g}_{ j'}(\ell') \|_{L^2}^2 \stackrel{\eqref{altro modo norma s}}{\leq_s} \| q \|_s \| g \|_{s_0}\,. \nonumber
\end{align}
By similar arguments one can prove that $(B) \leq_s \| q \|_{s_0} \| g \|_s$ and the claimed estimate follows by taking the sup over $j' \in \N$ in \eqref{londra 0}. The Lipschitz estimates follow by applying \eqref{kartoffen 2} to  
$$
\frac{{\mathcal R}(\omega_1)- {\mathcal R}(\omega_2)}{\omega_1 - \omega_2} = \frac{q(\omega_1) - q(\omega_2)}{\omega_1 - \omega_2} \langle \cdot, g(\omega_1) \rangle_{L^2_x}  + g(\omega_2) \Big\langle\frac{g(\omega_1) - g(\omega_2)}{\omega_1 - \omega_2}\,,\, \cdot \Big\rangle_{L^2_x}
$$
and passing to the sup for $\omega_1, \omega_2 \in \Omega_o$ with $\omega_1 \neq \omega_2$. 
\end{proof}
As we already mentioned, a T\"oplitz in time operator ${\mathcal R}$ in \eqref{action toplitz operator} may be regarded as $\vphi$-dependent family acting on the space of functions depending only on the $x$-variable 
$$
{\mathcal R}(\vphi) = \big( {\mathcal R}_j^{j'}(\vphi) \big)_{j, j' \in \Z \setminus \{ 0 \}}\,,
$$
and it admits the block representation 
$$
{\mathcal R}(\vphi) = \big( {\bf R}_j^{j'}(\vphi) \big)_{j, j' \in \N}\,, \quad  {\bf R}_j^{j'}(\vphi) = \begin{pmatrix}
{\mathcal R}_j^{j'}(\vphi) & {\mathcal R}_j^{- j'}(\vphi)\\
{\mathcal R}_{- j}^{j'}(\vphi) & {\mathcal R}_{- j}^{- j'}(\vphi)
\end{pmatrix}  \quad \forall \vphi \in \T^\nu\,, \quad \forall j, j' \in \N\,.
$$
The $2 \times 2$ matrix ${\bf R}_j^{j'}(\vphi)$ may be regarded as a linear operator in ${\mathcal L}({\bf E}_{j'}, {\bf E}_j)$, given by
$$
{\bf R}_j^{j'}(\vphi)[u] = \sum_{\begin{subarray}{c}
k = \pm j \\
k' = \pm j'
\end{subarray}} {\mathcal R}_k^{k'}(\vphi) u_{k'} e^{\ii k x}\,, \qquad \forall u(x) = u_{j'} e^{\ii j' x} + u_{- j'} e^{- \ii j' x} \in {\bf E}_{j'}\,.
$$
For the operator ${\mathcal R}(\vphi)$, we denote by $|{\mathcal R}(\vphi)|_{s, x}$ the block-decay norm (only with respect to the $x$-variable)
\begin{equation}\label{decadimento Kirchoff x}
|{\mathcal R}(\vphi)|_{s, x} := \sup_{j' \in \N} \Big(\sum_{j \in \Z} \langle j - j' \rangle^{2 s} \|{\bf R}_{j}^{j'}(\vphi)\|^2 \Big)^{\frac12}\,.
\end{equation}
If ${\mathcal R}$ is an operator of the form \eqref{operatore matriciale decadimento}, we define 
\begin{equation}\label{decadimento operatore matriciale spazio}
|{\mathcal R}(\vphi)|_{s, x} :={\rm max}\{ |{\mathcal R}_1(\vphi)|_{s, x}, |{\mathcal R}_2(\vphi)|_{s, x} \}\,.
\end{equation}
The following Lemma holds: 
\begin{lemma}\label{lemma decadimento Kirchoff in x}
Let ${\mathcal R}$ be a T\"oplitz in time operator. Then the following holds: 

\noindent
$(i)$ Let $s \geq 1$. If for any $\vphi \in \T^\nu$, $|{\mathcal R}(\vphi)|_{s, x} < + \infty$, then for any ${u} \in H^s_0(\T_x)$ 
$$
\| {\mathcal R}(\vphi) u \|_{H^s_x} \leq_s |{\mathcal R}(\vphi)|_{1, x} \|  u \|_{H^s_x} +  |{\mathcal R}(\vphi)  |_{s, x} \| u \|_{H^{1}_x}\,. 
$$

\noindent
$(ii)$ $|{\mathcal R}(\vphi)|_{s, x} \lessdot |{\mathcal R}|_{s + s_0}\,.$
\end{lemma}
\begin{proof}
The proof of item $(i)$ is similar to the one of Lemma \ref{stime tame operatori decadimento Kirchoff}, hence it is omitted. Item $(ii)$ follows since, expanding ${\bf R}_{j}^{j'}(\vphi) = \sum_{\ell \in \Z^\nu} {\bf R}_{j}^{j'}(\ell) e^{\ii \ell \cdot \vphi}$, applying the Cauchy-Schwartz inequality and using that $\sum_{\ell \in \Z^\nu} \langle \ell\rangle^{- 2 s_0} = C(s_0)$, one has that for all $j' \in \N$, 
\begin{align*}
\sum_{j \in \N} \langle j - j' \rangle^{2 s} \|{\bf R}_{j}^{j'}(\vphi)\|^2 & \lessdot \sum_{\begin{subarray}{c}
\ell \in \Z^\nu \\
j \in \N
\end{subarray}} \langle j - j' \rangle^{2 s} \langle \ell \rangle^{2 s_0} \| {\bf R}_{j}^{j'}(\ell)\|^2 \lessdot |{\mathcal R}|_{s + s_0}\,,
\end{align*}
which implies the claimed estimate passing to the supremum on $j' \in \N$. 
\end{proof}

\section{A reduction on the zero mean value functions}
For any function $u \in L^2(\T)$, we define 
\begin{equation}\label{proiettore media}
\pi_0 u := \frac{1}{2 \pi} \int_{\T} u(x)\, d x\,, \qquad \pi_0^\bot := {\rm Id} - \pi_0\,
\end{equation}
and
\begin{equation}\label{proiettore media vettoriale}
\Pi_0 := \begin{pmatrix}
\pi_0 & 0 \\
0 & \pi_0
\end{pmatrix}\,, \qquad \Pi_0^\bot := \begin{pmatrix}
\pi_0^\bot & 0 \\
0 & \pi_0^\bot
\end{pmatrix}\,.
\end{equation}
Given a function $v \in H^s(\T^{\nu + 1})$, $v(\vphi, x) = \sum_{j \in \Z} v_j(\vphi) e^{\ii j   x}$, we write 
\begin{equation}\label{splitting v media nulla costanti}
v(\vphi, x) = v_0(\vphi ) + v_\bot(\vphi, x)\,, 
\end{equation}
$$
v_0(\vphi) : = \pi_0 v(\vphi, x)\,, \quad  v_\bot(\vphi, x) := \pi_0^\bot v(\vphi, x) = \sum_{j \neq 0} v_j(\vphi) e^{\ii j x}\,. 
$$
Then according to the splitting \eqref{splitting v media nulla costanti}, applying the projection $\Pi_0$, $\Pi_0^\bot$ to the nonlinear map ${ F}$ defined in \eqref{operatore non lineare} and setting $u := \pi_0^\bot v$, $\psi := \pi_0^\bot p$, the equation $F(v, p) = { F}(\e, \omega, v, p) = 0$ is decomposed in 
\begin{equation}\label{sistema su media nulla}
\begin{cases}
\omega \cdot \partial_\vphi  u -  \psi = 0 \\
\omega \cdot \partial_\vphi \psi - \Big(1 + \e \int_{\T} |\partial_x  u|^2\, d x \Big) \partial_{xx}  u - \e  f_\bot = 0\,,
\end{cases}
\end{equation}
\begin{equation}\label{sistema sul modo 0}
\begin{cases}
\omega \cdot \partial_\vphi v_0 - p_0 = 0 \\
\omega \cdot \partial_\vphi p_0 - \e f_0 = 0\,
\end{cases}
\end{equation}
(we have used that $\partial_x v = \partial_x v_\bot = \partial_x u$ in \eqref{sistema su media nulla}). The above two systems are completely decoupled, hence they can be solved separately. In the next lemma, we solve explicitly the second system \eqref{sistema sul modo 0}. We use the hyphothesis \eqref{condizione media f} on the forcing term $f(\vphi, x)$. 
\begin{lemma}\label{lemma equazione sulle medie nulle}
Let $\gamma, \tau > 0$ and $q > 2 \tau$.
Then, for all $\omega \in \Omega_{\gamma, \tau}$ (see \eqref{diofantei Kn}), there exists a solution $v_0(\cdot; \omega, \e)\,, p_0(\cdot; \omega, \e)\in H^{q - 2 \tau}(\T^\nu, \R)$ of the system \eqref{sistema sul modo 0} with $\int_{\T^\nu} p_0(\vphi)\, d \vphi = \int_{\T^\nu} v_0(\vphi)\, d \vphi = 0$ and satisfying the estimates 
\begin{equation}\label{caccoletta media}
\| v_0\|_s \lessdot  \e \gamma^{- 2} \|  f\|_{s + 2 \tau}\,, \quad \| p_0\|_s \lessdot \e \gamma^{- 1} \| f \|_{s + \tau}\,, \qquad \forall 0 \leq s \leq q - 2 \tau\,.
\end{equation}
\end{lemma}
\begin{proof}
Since 
$$
\int_{\T^\nu } f_0(\vphi)\,  d \vphi = \int_{\T^{\nu + 1}} f(\vphi, x)\, d \vphi\, d x \stackrel{\eqref{condizione media f}}{=} 0 \,,
$$
the second equation in \eqref{sistema sul modo 0} can be solved by taking $p_0 := \e (\omega \cdot \partial_\vphi)^{- 1} f_0$ where, since $\omega \in \Omega_{\gamma, \tau}$, the operator $(\omega \cdot \partial_\vphi)^{- 1}$ is well defined by \eqref{om d vphi inverso}. Then we can solve the second equation in \eqref{sistema sul modo 0} by defining $v_0 := (\omega \cdot \partial_\vphi)^{- 1} p_0 {=} \e (\omega \cdot \partial_\vphi)^{- 2} f_0\,.$ Clearly $\int_{\T^\nu} v_0(\vphi)\, d \vphi = \int_{\T^\nu} p_0(\vphi)\, d \vphi = 0$ and the claimed estimates follow by applying \eqref{stima om d vphi inverso}. 
\end{proof}
In all the rest of the paper, we will study the equation \eqref{sistema su media nulla} on the zero mean value functions in $x$. 
We will find zeros of the nonlinear operator ${\mathcal F}(\e, \omega, \cdot ) : { H}^s_0(\T^{\nu + 1}, \R^2) \to {H}^{s - 2}_0(\T^{\nu + 1}, \R^2)$ (recall \eqref{Sobolev media nulla}), defined as
\begin{equation}\label{operatore su media nulla}
 {\mathcal F}(\e, \omega, u, \psi ) := \begin{pmatrix}
\omega \cdot \partial_\vphi  u -  \psi  \\
\omega \cdot \partial_\vphi \psi - \Big(1 + \e \int_{\T} |\partial_x  u|^2\, d x \Big) \partial_{xx}  u - \e f_\bot \,,
\end{pmatrix}\,.
\end{equation}
Note that, setting ${\bf u} := (u, \psi)$, ${\mathcal F}({\bf u}) = {\mathcal F}(\e, \omega, {\bf u})$, one has 
\begin{equation}
{\mathcal F}( {\bf u}) = \omega \cdot \partial_\vphi {\bf u} - J \nabla_{{\bf u}} {\mathcal H}_\e ({\bf u})\,, \qquad J := \begin{pmatrix}
0 & 1 \\
- 1 & 0
\end{pmatrix}\,, \quad 
\end{equation}
where $J \nabla_{{\bf u}} {\mathcal H}_\e$ is the Hamiltonian vector field 
$$
J \nabla_{{\bf u}} {\mathcal H}_\e = \begin{pmatrix}
0 & 1 \\
- 1 & 0
\end{pmatrix} \begin{pmatrix}
\nabla_u {\mathcal H}_\e \\
\nabla_\psi {\mathcal H}_\e
\end{pmatrix} = \begin{pmatrix}
- \nabla_\psi {\mathcal H}_\e \\
\nabla_u {\mathcal H}_\e
\end{pmatrix}
$$
generated by the Hamiltonian 
\begin{equation}\label{hamiltoniana media nulla}
{\mathcal H}_\e (u, \psi) := \frac12 \int_{\T} \big( \psi^2 + |\partial_x u|^2 \big)\, dx  + \e \Big( \frac12  \int_{\T} |\partial_x u|^2 \, dx \Big)^2 - \e \int_{\T} f_\bot u\, dx\,,
\end{equation}
 defined on the phase space $H_0^1(\T_x, \R) \times L^2_0(\T_x, \R)$. The Hamiltonian ${\mathcal H}_\e$ is simply the restriction of the Hamiltonian $H$ in \eqref{Hamiltoniana Kirchoff} to the space of the functions with zero average in $x$. We look for the zeros of \eqref{operatore su media nulla} by means of an implicit function Theorem of Nash-Moser type. The Theorem \ref{main theorem kirchoff} will be deduced by Lemma \ref{lemma equazione sulle medie nulle} and by the following Theorem
\begin{theorem}\label{main theorem kirchoff 1}
There exist  $q := q(\nu) > 0$, $s := s(\nu) > 0$ such that: for any $f \in {\mathcal C}^q(\T^\nu \times \T, \R)$, there exists $\e_0 = \e_0(\nu, f) > 0$ small enough such that for all $\e \in (0, \e_0)$, there exists a Cantor set ${\mathcal C}_\e \subseteq \Omega$ of asimptotically full Lebesgue measure i.e. 
$$
|{\mathcal C}_\e| \to |\Omega | \qquad \text{as} \qquad  \e \to 0\,,
$$
such that for any $\omega \in {\mathcal C}_\e$ there exists ${\bf u}(\e, \omega) = (u(\e, \omega), \psi(\e, \omega)) \in {H}^s_0(\T^{\nu + 1}, \R^2)$ satisfying ${\mathcal F}(\e, \omega, {\bf u}(\e, \omega)) = 0$
where the nonlinear operator ${\mathcal F}$ is defined in \eqref{operatore su media nulla} and 
$$
\| {\bf u}(\e, \omega)\|_s \to 0 \qquad \text{as} \qquad \e \to 0\,.
$$
\end{theorem}

\noindent
Theorem \ref{main theorem kirchoff 1} is based on a Nash-Moser iterative scheme implemented in Section \ref{sec:NM}. 
he key ingredient in the proof---which also implies the linear stability of the quasi-periodic solutions---is the {\it reducibility} of the linear operator ${\mathcal L} = {\mathcal L}({\bf u}) = \partial_{\bf u} {\mathcal F}({\bf u}) $
obtained by linearizing \eqref{operatore su media nulla} at any approximate (or exact) solution $ {\bf u} = (u, \psi) $. This is the content of Sections \ref{riduzione linearizzato}, \ref{sec:redu}. The proof of the invertibility of ${\mathcal L}$ and the tame estimates for its inverse is provided in Section \ref{sezione invertibilita cal L}. The measure estimate of the set ${\mathcal C}_\e$ of the {\it good parameters} is provided in Section \ref{sezione stime in misura}.

\section{Regularization of the linearized operator}\label{riduzione linearizzato}
For any family $\omega \in \Omega_o({\bf u}) \mapsto {\bf u}(\cdot; \omega) := (u(\cdot; \omega), \psi(\cdot; \omega)) \in H^S_0(\T^{\nu + 1}, \R^2)$, we consider the linearized operator ${\mathcal L} = {\mathcal L}({\bf u}) = {\mathcal L}(\omega, {\bf u}(\omega)) := \partial_{\bf u} {\mathcal F}(\e, \omega, {\bf u}(\omega)) : H^s_0(\T^{\nu + 1}, \R^2) \to H^{s - 2}_0(\T^{\nu + 1}, \R^2)$ for $2 \leq s \leq S - 2$ (recall \eqref{Sobolev media nulla}).
It has the form 
\begin{equation}\label{operatore linearizzato}
{\mathcal L} [\widehat u, \widehat \psi] := \begin{pmatrix}
\omega \cdot \partial_\vphi \widehat u - \widehat \psi \\
\omega \cdot \partial_\vphi \widehat \psi - a(\vphi) \partial_{xx} \widehat u + {\mathcal R} \widehat u
\end{pmatrix}
\end{equation}
where 
\begin{equation}\label{a cal R}
a(\vphi) :=1 + \e \int_{\T} |\partial_x u(\vphi, x)|^2\,dx\,,\qquad {\mathcal R} \widehat u :=  2 \e \partial_{xx} u \int_{\T} (\partial_{xx} u)  \,\widehat u\,dx.
\end{equation}
Along this section, we will always assume the following hyphothesis, which will be verified along the Nash-Moser nonlinear iteration of Section \ref{sec:NM}.
\begin{itemize}
\item {\sc Assumption}. The function ${\bf u} := (u, \psi)$ depends in a Lipschitz way on the parameter $\omega \in \Omega_o := \Omega_o({\bf u}) \subset \Omega_{\gamma, \tau}$ with $\gamma \in (0, 1)$, $\tau > 0$ (recall \eqref{diofantei Kn}) and for some $\mu := \mu (\tau, \nu) > 0$, for some $S \geq s_0 + \mu$, the map $\omega \in \Omega_o({\bf u}) \mapsto {\bf u}(\cdot; \omega) \in { H}^S_0(\T^{\nu + 1}, \R^2)$ satisfies
\begin{equation}\label{ansatz}
 \|{\bf u} \|_{s_0 + \mu}^\Lipg \leq 1 \quad \text{and}\quad \e \gamma^{- 1} \ll 1
\end{equation}
where we recall that $s_0 := [(\nu  +  1)/2] + 1$, so that $H^{s_0}(\T^{\nu + 1})$ is compactly embedded in ${\mathcal C}^0(\T^{\nu + 1})$. We remark that in Sections \ref{riduzione linearizzato}-\ref{sec:NM}, the constant $\tau >0$ is independent from the number of frequencies $\nu$. It will be fixed as a function of $\nu$ only in Section \ref{sezione stime in misura} for the measure estimates (see \eqref{valori finali tau tau*}).
\end{itemize}
The function $a$ and the operator ${\mathcal R}$ in \eqref{a cal R} depend only on the first component $u$ of the function ${\bf u} = (u, \psi)$. We denote by $\partial_u a [h]$, $\partial_u {\mathcal R}[h]$ their derivatives with recpect to $u$ in the direction $h$. 

\noindent
Note that, since $a(\vphi)$ is a real valued function and ${\mathcal R}$ is symmetric, the operator ${\mathcal L}$ is Hamiltonian in the sense of the definition \ref{campo Hamiltoniano reale}. Let us give some estimates on $a$ and ${\mathcal R}$ defined in \eqref{a cal R}. 
\begin{lemma}\label{stime coefficienti linearizzato}
Assume \eqref{ansatz}, with $\mu = 2$. Then for any $s_0 \leq s \leq S - 2$ the following holds: 
\begin{equation}\label{estimates a}
\| a - 1 \|_s \leq_s \e \|  u \|_{s + 1}\,, \qquad  \| a - 1\|_s^\Lipg \leq_s \e \| u \|_{s + 1}^\Lipg\,,\qquad  
\end{equation}
\begin{equation}\label{stime derivata a}
\| \partial_{ u} a [{ h}]\|_s \leq_s \e  (\|  h \|_{s + 1} + \|u  \|_{s + 1} \|h \|_{s_0 + 1})\,,
\end{equation}
The operator ${\mathcal R}$ in \eqref{a cal R} has the form \eqref{operatore forma buona resto}, with $q $ and $g $ satisfying the estimates 
\begin{align}
& \| q \|_s \leq_s  \e \| u \|_{s + 2}\,,\quad \| g \|_s \leq_s   \| u \|_{s + 2}\,, \label{estimates cal R 1} \\
&  \| q \|_s^\Lipg \leq_s  \e \| u \|_{s + 2}^\Lipg\,,\quad \| g \|_s^\Lipg \leq_s   \| u \|_{s + 2}^\Lipg\,, \label{estimates cal R 11}\\
& \| \partial_u q [h] \|_s \leq_s  \e\| h \|_{s + 2 }\,,\quad \| \partial_u g [h] \|_s \leq_s  \| h \|_{s + 2 }\,. \label{estimates cal R 2}
\end{align}
\end{lemma}
\begin{proof}
The estimates \eqref{estimates a}, \eqref{stime derivata a} follow by the definition \eqref{a cal R} and by the interpolation Lemma \ref{interpolazione C1 gamma}, using the condition \eqref{ansatz}. The estimates \eqref{estimates cal R 1}-\eqref{estimates cal R 2} follow since ${\mathcal R}$ is an operator of the form \eqref{operatore forma buona resto}, with $q :=  2 \e \partial_{xx} u$ and $g := \partial_{xx} u$.
\end{proof}

\noindent
{\it Notation.} In the following, with a slight abuse of notations, For any function $a(\vphi)$, we simply denote  by $a = a(\vphi)$, the multiplication operator $h(\vphi, x) \mapsto a(\vphi) h(\vphi, x)$, acting on the space of functions with zero average in $x$. 
\subsection{Symplectic symmetrization of the highest order}\label{step 1 riduzione}
We start by simmetryzing the highest order of the operator 
$$
{\mathcal L} = \begin{pmatrix}
\omega \cdot \partial_\vphi & - 1 \\
 - a \partial_{xx} + {\mathcal R}& \omega \cdot \partial_\vphi
\end{pmatrix}\,.
$$
Let us consider the transformation  
\begin{equation}\label{cal S1}
 {\mathcal S} = {\mathcal S}(\vphi) : = \begin{pmatrix}
\beta(\vphi) |D|^{- \frac12} & 0 \\
0 & \beta(\vphi)^{- 1} |D|^{\frac12}
\end{pmatrix}
\end{equation}
where $\beta : \T^\nu \to \R$ is a Sobolev function close to $1$ to be determined (recall also the definition \eqref{definizione |D| m}). 
The inverse of the operator ${\mathcal S}$ (acting on Sobolev spaces of zero average functions in $x$) is given by 
\begin{equation}\label{cal S1 inverso}
{\mathcal S}^{- 1}  = {\mathcal S}(\vphi)^{- 1} :=\begin{pmatrix}
\beta(\vphi)^{- 1} |D|^{ \frac12} & 0 \\
0 & \beta(\vphi) |D|^{- \frac12}
\end{pmatrix} \,.
\end{equation} 
Using that for any function $a = a(\vphi)$ depending only on time, the commmutators $[a,  |D|^m] = 0$, $[a, {\mathcal R}] = 0$ where ${\mathcal R}$ is defined in \eqref{a cal R} and since $- \partial_{xx} = |D|^2$, we have 
\begin{equation}
{\mathcal S}^{- 1} {\mathcal L} {\mathcal S} = \begin{pmatrix}
\omega \cdot \partial_\vphi + \beta^{- 1}(\omega \cdot \partial_\vphi \beta) & - \beta^{- 2} |D| \\
 a \beta^2 |D| + \beta^2 |D|^{- \frac12}{\mathcal R} |D|^{- \frac12}   & \omega \cdot \partial_\vphi + \beta \omega \cdot \partial_\vphi ( \beta^{- 1} )   
\end{pmatrix}\,.
\end{equation}
We choose $\beta(\vphi)$ so that $\beta^{- 2}(\vphi) = a(\vphi) \beta^2(\vphi)$, namely we define 
\begin{equation}\label{definition beta}
\beta(\vphi) := \frac{1}{[a(\vphi)]^{\frac14}}\,.
\end{equation}
Since $\beta \omega \cdot \partial_\vphi (\beta^{- 1}) = -  \beta^{- 1} \,\omega \cdot \partial_\vphi \beta$,
we get that 
\begin{equation}\label{cal L1}
{\mathcal L}_1 := {\mathcal S}^{- 1} {\mathcal L} {\mathcal S} = \begin{pmatrix}
\omega \cdot \partial_{\vphi} + a_0 & - a_1 |D| \\
 a_1|D| + {\mathcal R}^{(1)} &  \omega \cdot \partial_\vphi - a_0
\end{pmatrix}\,,
\end{equation}
where 
\begin{equation}\label{definitions a0 a1}
a_0 := \frac{\omega \cdot \partial_\vphi \beta}{\beta} \,, \quad a_1 := \sqrt{a}\,, \quad {\mathcal R}^{(1)} := \beta^2 |D|^{- \frac12}{\mathcal R} |D|^{- \frac12}\,.
\end{equation}
Since $\beta(\vphi)$ is a real-valued function and the operators $|D|^{\pm \frac12}$ are real operators, the operator ${\mathcal S}$ is real. A direct verification shows that it is also symplectic (see Definition \ref{trasformazione simplettica reale}). Hence the transformed operator ${\mathcal L}_1$ is still real and Hamiltonian (see Definition \ref{campo Hamiltoniano reale}).
Now we give some estimates on the coefficients of the operator ${\mathcal L}_1$. 
\begin{lemma}\label{Lemma dopo simmetrizzazione}
Assume \eqref{ansatz}, with $\mu = 2$. Then for any $s_0 \leq s \leq S - 2$ the following holds: 
 the maps ${\mathcal S}^{\pm 1} : H^{s + \frac12}_0(\T^{\nu + 1}, \R^2) \to H^{s }_0(\T^{\nu + 1}, \R^2)$ satisfy the estimates 
 \begin{equation}\label{stime cal S1 cal S2 sup}
\| {\mathcal S}^{\pm 1} {\bf h} \|_s \leq_s \| {\bf h}  \|_{s + \frac12} + \| { u} \|_{s + 1} \| {\bf h} \|_{s_0 + \frac12}\,, \quad {\bf h} \in H^{s + \frac12}(\T^{\nu + 1}, \R^2)\,.
\end{equation}
For any family ${\bf h}(\cdot; \omega) \in  H^{s + \frac12}_0(\T^{\nu + 1}, \R^2)$, $\omega \in \Omega_o$,
\begin{equation}\label{stime cal S1 cal S2}
\| {\mathcal S}^{\pm 1} {\bf h} \|_s^\Lipg \leq_s  \| {\bf h}  \|_{s + \frac12}^\Lipg + \| { u} \|_{s + 1}^\Lipg \| {\bf h} \|_{s_0 + \frac12}^\Lipg\,.
\end{equation}
The functions $a_0$, $a_1$ defined in \eqref{definitions a0 a1} satisfy the estimates 
\begin{align}
& \| a_1 - 1\|_s\,,\, \| a_0\|_s \leq_s \e  (1 + \| { u} \|_{s + 2}) \,,  \label{estimates a0 a1} \\
 & \| a_1 - 1\|_s^\Lipg\,, \, \| a_0\|_s^\Lipg \leq_s \e (1 + \| { u} \|_{s + 2}^\Lipg)\,, \label{estimates a0 a11} \\
  & \| \partial_u a_k [h]\|_s \leq_s  \e \Big( \| h \|_{s + 2} + \| {u} \|_{s + 2} \| h \|_{s_0 + 2}\Big)\,, \qquad k = 0,1\,. \label{estimates derivatives a0 a1}
\end{align}
The remainder ${\mathcal R}^{(1)}$ in \eqref{definitions a0 a1} has the form \eqref{operatore forma buona resto}, with $ q = q_1$, $g = g_1$ satisfying the estimates 
\begin{align}
& \| q_1 \|_s \leq_s  \e \| { u}  \|_{s + 2 }\,, \quad \| g_1 \|_s \leq_s \| {u} \|_{s + 2}\,, \label{estimate cal R2 1} \\
& \| q_1 \|_s^{\Lipg} \leq_s  \e \| { u}  \|_{s + 2 }^\Lipg\,, \quad \| g_1 \|_s^\Lipg \leq_s \| {u} \|_{s + 2}^\Lipg\,, \label{estimate cal R2 11} \\
& \|\partial_{u} q_1[h]\|_s \leq_s \e \Big( \| h \|_{s + 2} + \| {u} \|_{s + 2} \| h \|_{s_0 + 2}\Big)\,,\,\, \|\partial_{u} g_1[h]\|_s \leq_s  \| h \|_{s + 2}\,. \label{estimate cal R2 2}
\end{align}
\end{lemma}
\begin{proof}
The estimates \eqref{stime cal S1 cal S2 sup}-\eqref{estimates derivatives a0 a1} follow by the definitions \eqref{cal S1}, \eqref{cal S1 inverso}, \eqref{definition beta}, \eqref{definitions a0 a1}, by the estimates \eqref{estimates a} and by Lemmata \ref{interpolazione C1 gamma}, \ref{Moser norme pesate}. Let us prove the estimates \eqref{estimate cal R2 1}-\eqref{estimate cal R2 2}. 
By \eqref{a cal R}, \eqref{definitions a0 a1}, using that $|D|^{- \frac12}$ is symmetric, one has that ${\mathcal R}^{(1)} h = q_1\, \int_{\T} g_1\,h\,d x $ with
\begin{equation}\label{funzioni q2 g2 cal R2}
 q_1 := 2 \e \beta^2 ( |D|^{- \frac12}\partial_{xx} u)\,, \qquad g_1 := |D|^{- \frac12}\partial_{xx} u\,.
\end{equation}
One can estimate the function $\beta$ in \eqref{definition beta} by using Lemma \ref{Moser norme pesate} and the estimate \eqref{estimates a}. Applying the interpolation Lemma \ref{interpolazione C1 gamma}, the claimed estimates follow.  
\end{proof}
\begin{lemma}\label{stime Hs x 1}
The operators ${\mathcal S}^{\pm 1}$ defined in \eqref{cal S1}, \eqref{cal S1 inverso} can be regarded as an operator acting on the Sobolev space of the functions in $x$, namely for any $s \geq 1$, for any $\vphi \in \T^\nu$, 
\begin{align*}
& {\mathcal S}(\vphi) \in {\mathcal L}\Big(  H^{s - \frac12}_0(\T_x, \R^2), H^s_0(\T_x, \R) \times H^{s - 1}_0(\T_x, \R) \Big)\,,\\
&{\mathcal S}(\vphi)^{- 1} \in {\mathcal L}\Big(H^s_0(\T_x, \R) \times H^{s - 1}_0(\T_x, \R), H^{s - \frac12}_0(\T_x, \R^2) \Big)\,.
\end{align*}
\end{lemma}
\begin{proof}
By the definition of the function $\beta(\vphi)$ in \eqref{definition beta}, using the estimate \eqref{estimates a} on $a(\vphi)$, the Lemma \ref{Moser norme pesate} and the ansatz \eqref{ansatz}, one gets $\| \beta^{\pm 1}\|_{L^\infty(\T^\nu)} \lessdot 1\,.$
Moreover $\| |D|^{\frac12} h \|_{H^s_x} \leq \| h \|_{H^{s + \frac12}_x}\,, \quad \| |D|^{- \frac12} h \|_{H^s_x} \leq \| h \|_{H^{s - \frac12}_s}\,$ and then the Lemma follows.
\end{proof}
\subsection{Complex variables}\label{variabili complesse}
Now we consider the complex variables $z := \frac{\widehat u + \ii \widehat \psi}{\sqrt{2}}$
introduced in \eqref{prima volta variabili complesse 0}, \eqref{prima volta variabili complesse 1} in order to write the operator ${\mathcal L}_1$ defined in \eqref{cal L1} in complex coordinates. More precisely, we consider the transformations 
\begin{equation}\label{matrice coordinate complesse}
{\mathcal B} := \frac{1}{\sqrt{2}} \begin{pmatrix}
1 & 1 \\
\frac{1}{\ii} & - \frac{1}{\ii}
\end{pmatrix}\, \qquad {\mathcal B}^{- 1} = \frac{1}{\sqrt{2}} \begin{pmatrix}
1 & \ii \\
1  & - {\ii}
\end{pmatrix}
\end{equation} and we get that the conjugated operator ${\mathcal L}_2 := {\mathcal B}^{- 1}{\mathcal L}_1 {\mathcal B}$ is given by
\begin{equation}\label{cal L2 complex coordinates}
{\mathcal L}_2  = \begin{pmatrix}
\omega \cdot \partial_\vphi + \ii a_1 |D| + \ii {\mathcal R}^{(2)} & a_0+ \ii {\mathcal R}^{(2)} \\
a_0 - \ii {\mathcal R}^{(2)} & \omega \cdot \partial_\vphi - \ii a_1 |D| - \ii {\mathcal R}^{(2)}
\end{pmatrix}\,,
\end{equation}
with ${\mathcal R}^{(2)}  := \frac{{\mathcal R}^{(1)}}{{2}}$. Since $a_1$ and $a_0$ are real valued functions and ${\mathcal R}^{(1)}$ (and then ${\mathcal R}^{(2)}$) is symmetric and real, the operator ${\mathcal L}_2$ is a Hamiltonian operator in complex coordinates, in the sense of the Definition \eqref{operatore Hamiltoniano coordinate complesse}. Note that the transformations ${\mathcal B}^{\pm 1}$ satisfy for all $s \geq 0$
\begin{equation}\label{proprieta coordinate complesse}
{\mathcal B } : {\bf H}^s_0(\T^{\nu + 1}) \to H^s_0(\T^{\nu + 1}, \R^2)\,, \quad {\mathcal B}^{- 1} : H^s_0(\T^{\nu + 1}, \R^2) \to {\bf H}^s_0(\T^{\nu + 1})\,, 
\end{equation}
\begin{equation}\label{proprieta coordinate complesse x}
{\mathcal B } : {\bf H}^s_0(\T_x) \to H^s_0(\T_x, \R^2)\,, \quad {\mathcal B}^{- 1} : H^s_0(\T_x, \R^2) \to {\bf H}^s_0(\T_x)
\end{equation}
where we recall that the real subspace ${\bf H}^s_0(\T^{\nu + 1})$, resp. ${\bf H}^s_0(\T_x)$ of $H^s_0(\T^{\nu + 1}, \C^2)$, resp. $H^s_0(\T_x, \C^2)$, is defined in \eqref{definizione bf H s0}.  
\subsection{ Change of variables}\label{sezione diffeo del toro}
The aim of this Section is to reduce to constant coefficients the highest order term $a_1(\vphi) |D|$ in the operator ${\mathcal L}_2$ defined in \eqref{cal L2 complex coordinates}.   In order to do this, let us consider a diffeomorphism of the torus $\T^\nu$ of the form
$$
\vphi \in \T^\nu \mapsto \vphi + \omega \alpha(\vphi) \in \T^\nu\,,\quad 
$$
where $\alpha : \T^\nu \to \R$ has to be determined. This diffeomorphism of the torus induces on the space of functions $h(\vphi, x)$ a linear operator 
\begin{equation}\label{def cal A}
({\mathcal A}h)(\vphi, x) := h(\vphi + \omega \alpha(\vphi), x)\,,
\end{equation} 
whose inverse has the form 
\begin{equation}\label{inverso cal A}
{\mathcal A}^{-1} h(\vartheta, x) := h(\vartheta + \omega \tilde \alpha(\vartheta), x)\,,
\end{equation}
where $\vartheta \to \vartheta + \omega \tilde \alpha(\vartheta)$ is the inverse diffeomorphism of $\vphi \to \vphi + \omega \alpha(\vphi)$.
One has 
$$
{\mathcal A}^{- 1} (\omega \cdot \partial_\vphi) {\mathcal A} = {\mathcal A}^{- 1}[1 + \omega \cdot \partial_\vphi \alpha] \omega \cdot \partial_\vartheta \,, \quad  {\mathcal A}^{- 1} |D| {\mathcal A} = |D|\,, \quad {\mathcal A}^{- 1} a {\mathcal A} = {\mathcal A}^{- 1}[a]
$$
where we recall that $a$ denotes the multiplication operator $h \to a h$. Recalling that $
 {\mathbb I}_2 := \begin{pmatrix}
 {\rm Id}_0 & 0 \\
 0 & {\rm I d}_0
 \end{pmatrix}\,, \quad 
$
where ${\rm Id}_0 : L^2_0 \to L^2_0$ is the identity and defining
\begin{equation}\label{definition rho}
\rho := {\mathcal A}^{- 1}[1 + \omega \cdot \partial_\vphi \alpha]\,,
\end{equation}
we get 
\begin{align*}
& {\mathcal A}^{- 1} {\mathbb I}_2 {\mathcal L}_2 {\mathcal A} {\mathbb I}_2 \\
& = \begin{pmatrix}
\rho \omega \cdot \partial_\vartheta + \ii {\mathcal A}^{- 1}[a_1] |D|  + \ii {\mathcal A}^{- 1} {\mathcal R}^{(2)} {\mathcal A} & {\mathcal A}^{- 1}[a_0] + \ii {\mathcal A}^{- 1} {\mathcal R}^{(2)} {\mathcal A} \\
{\mathcal A}^{- 1}[a_0] - \ii {\mathcal A}^{- 1} {\mathcal R}^{(2)} {\mathcal A} & \rho\omega \cdot \partial_\vartheta - \ii {\mathcal A}^{- 1}[a_1] |D| - \ii {\mathcal A}^{- 1} {\mathcal R}^{(2)} {\mathcal A}
\end{pmatrix}\,.
\end{align*}
We want to choose the function $\alpha$ so that the coefficient $\rho$ in front of $\omega \cdot \partial_{\vartheta}$ is proportional to the coefficient ${\mathcal A}^{- 1}[a_1]$ in front of the operator $|D|$. To this aim it is enough to solve the equation
\begin{equation}\label{equation for alpha}
m \big(1 + \omega \cdot \partial_\vphi \alpha(\vphi) \big) = a_1(\vphi)\, \quad m \in \R\,.
\end{equation}
Integrating on $\T^\nu$ we fix the value of $m$ as 
\begin{equation}\label{definition m}
m := \frac{1}{(2 \pi)^\nu} \int_{\T^\nu} a_1(\vphi)\, d \vphi
\end{equation}
and then, since $\omega \in \Omega_o \subseteq  \Omega_{\gamma, \tau}$, recalling the definitions \eqref{om d vphi inverso}, \eqref{diofantei Kn} we get
\begin{equation}\label{definition alpha}
\alpha(\vphi) = (\omega \cdot \partial_\vphi)^{- 1} \Big[ \frac{a_1}{m} - 1 \Big](\vphi)\,.
\end{equation}
Note that, since the function $a_1$ is real valued, $m$ is real and then $\alpha$ is a real valued function. We have ${\mathcal A}^{- 1} {\mathbb I}_2 {\mathcal L}_2 {\mathcal A} {\mathbb I}_2 = \rho {\mathcal L}_3 \,,$
with 
\begin{equation}\label{cal L3}
{\mathcal L}_3 := \begin{pmatrix}
\omega \cdot \partial_\vartheta + \ii m |D| + \ii {\mathcal R}^{(3)} & b_0 + \ii {\mathcal R}^{(3)} \\
b_0 - \ii {\mathcal R}^{(3)} & \omega \cdot \partial_\vartheta - \ii m |D| - \ii {\mathcal R}^{(3)}
\end{pmatrix}\,, 
\end{equation}
\begin{equation}\label{definitions a2 cal R3}
b_0  := \rho^{- 1} {\mathcal A}^{- 1}[a_0] \,, \qquad {\mathcal R}^{(3)} := \rho^{- 1} {\mathcal A}^{- 1} {\mathcal R}^{(2)} {\mathcal A}\,.
\end{equation}
Note that the operator ${\mathcal L}_3$ is still Hamiltonian in the sense of the definition \eqref{operatore Hamiltoniano coordinate complesse}, since $m \in \R$, $|D|$ is a symmetric real operator, $b_0$ is a real valued function and ${\mathcal R}^{(3)}$ is a real and symmetric operator, implying that $({\mathcal R}^{(3)})^* = \overline{({\mathcal R}^{(3)})^T} = {\mathcal R}^{(3)}$.  
\begin{lemma}\label{stime step 1}
There exists a constant $\sigma = \sigma(\tau, \nu) > 2$ such that if \eqref{ansatz} holds with $\mu = \sigma$, then for all $s_0 \leq s \leq S - \sigma$ the following estimates hold:
\begin{equation}\label{stime m}
|m - 1| \lessdot \e \,, \quad |m - 1|^\Lipg \lessdot \e\,,\quad |\partial_u m [h]| \lessdot \e \| h \|_{s_0 + 2}\,.
\end{equation}
The transformations ${\mathcal A}^{\pm 1} : H^s_0(\T^{\nu + 1}, \C) \to H^s_0(\T^{\nu + 1}, \C)$ satisfy 
\begin{align}
& \| {\mathcal A}^{\pm 1} h \|_s \leq_s \| h \|_{s } + \| u \|_{s + \s} \|  h\|_{s_0 + 1}\,, \label{stima cal A}\\
& \| {\mathcal A}^{\pm 1} h \|_s^\Lipg \leq_s \| h \|_{s + 1}^\Lipg + \| u \|_{s + \s}^\Lipg \|  h\|_{s_0 + 2}^\Lipg\,, \label{stima cal A1}\\
& \| \partial_u ({\mathcal A}^{\pm 1} h) g \|_s \leq_s \e \gamma^{-1} \Big( \| h \|_{s + \s} \| g \|_{s_0 + \s} + \| h \|_{s_0 + \s} \| g \|_{s + \s} \nonumber\\
& \qquad \qquad \qquad  + \| u \|_{s + \s} \|  h\|_{s_0 + \s} \| h \|_{s_0 + \s} \Big)\,. \label{stima derivata cal A}
\end{align}
The function $\rho$ defined in \eqref{definition rho} satisfies 
\begin{equation}\label{stima rho}
\| \rho^{\pm 1} - 1\|_s \leq_s \e (1 + \| u \|_{s + \sigma})\,, \qquad \| \rho^{\pm 1} - 1\|_s^\Lipg \leq_s \e (1 + \| u \|_{s + \sigma}^\Lipg)\,, 
\end{equation}
\begin{equation}\label{stima derivata rho}
\quad \| \partial_u \rho^{\pm 1}[h] \|_s \leq_s \e \big( \| h \|_{s + \sigma} + \| h \|_{s + \sigma} \| h \|_{s_0 + \sigma} \big)\,.
\end{equation}
The function $b_0$ defined in \eqref{definitions a2 cal R3} satisfies 
\begin{equation}\label{stime a2}
\| b_0 \|_s \leq_s \e(1 +  \| u  \|_{s + \s})\,,\quad \| b_0 \|_s^\Lipg \leq_s \e(1 +  \| u  \|_{s + \s}^\Lipg)\,,
\end{equation}
\begin{equation}\label{stime derivata b0}
\| \partial_u b_0 [h]\|_s \leq_s  \e \Big( \| h \|_{s + \s} + \| u \|_{s + \s} \|h \|_{s_0 + \s}\Big)\,.
\end{equation}
The remainder ${\mathcal R}^{(3)}$ defined in \eqref{definitions a2 cal R3} satisfies the estimates
\begin{equation}\label{stime cal R3}
| {\mathcal R}^{(3)} |D| |_s \leq_s \e (1 + \| u \|_{s + \s })\,,\,\, | {\mathcal R}^{(3)} |D| |_s^\Lipg \leq_s \e (1 + \| u \|_{s + \s }^\Lipg)\,,\quad 
\end{equation}
\begin{equation}\label{stime derivata cal R3}
|\partial_u {\mathcal R}^{(3)} [h] |_s \leq_s \e \Big( \| h \|_{s + \s} + \| u \|_{s + \s} \|h \|_{s_0 + \s} \Big)\,.
\end{equation}
\end{lemma}
\begin{proof}
The estimates \eqref{stime m} follow by the formula \eqref{definition m} and using the estimates \eqref{estimates a0 a1}, \eqref{estimates derivatives a0 a1}. 
The transformation ${\mathcal A}$ has been also used in \cite{BBM-Airy}, \cite{BBM-auto}, \cite{BBM-mKdV}, \cite{Berti-Montalto}, \cite{Feola-Procesi}. The proof of the estimates \eqref{stima cal A}-\eqref{stima derivata rho} can be done by using Lemma \ref{lemma:utile} as in these papers. For a detailed proof see for instance \cite{Feola-Procesi}, Pages 25-26.

\noindent
Let us prove the estimates \eqref{stime cal R3}, \eqref{stime derivata cal R3}. One has  
\begin{align*}
{\mathcal R}^{(3)}h = \rho^{- 1}{\mathcal A}^{- 1} {\mathcal R}^{(2)} {\mathcal A} h \stackrel{\eqref{cal L2 complex coordinates}}{=} \frac{1}{2} \rho^{- 1}{\mathcal A}^{- 1} {\mathcal R}^{(1)} {\mathcal A} h\stackrel{\eqref{funzioni q2 g2 cal R2}}{=} q_3 \, \int_{\T} g_3 \, h \, d x\,, \qquad  
\end{align*}
with 
$$
q_3 := \frac{1}{2}\rho^{- 1} {\mathcal A}^{- 1}(q_1)\,, \qquad g_3 := {\mathcal A}^{- 1}(g_1)\,.
$$
Therefore, the functions $q_3$ and $g_3$, can be estimated by using \eqref{estimate cal R2 1}, \eqref{estimate cal R2 2}, \eqref{stima cal A}-\eqref{stima derivata rho} and Lemma \ref{interpolazione C1 gamma}. The estimates in \eqref{stime cal R3} then follow by applying Lemma \ref{decadimento operatori di proiezione}. The estimate for $\partial_u {\mathcal R}^{(3)}[h]$ follows by differentiating the expression of ${\mathcal R}^{(3)}$, $q_3$, $g_3$ given above and applying again Lemma \ref{decadimento operatori di proiezione}. 
\end{proof}

\subsection{Descent method}\label{descent method}
Introducing the notation 
\begin{equation}\label{operator T def}
 \quad T := \begin{pmatrix}
{\rm I d}_0 & 0 \\
0 & - {\rm Id}_0
\end{pmatrix}
\end{equation}
we can write the operator ${\mathcal L}_3$ in \eqref{cal L3} as 
\begin{equation}\label{cal L3 notazioni compatte}
{\mathcal L}_3 = \omega \cdot \partial_\vphi {\mathbb I}_2 + \ii m T |D| + B_0 + {\mathcal R}_3\,,
\end{equation}
where 
\begin{equation}\label{A0 bf R0}
 B_0(\vphi) := \begin{pmatrix}
0 & b_0(\vphi) \\
b_0 (\vphi) & 0
\end{pmatrix}\,,\quad {\mathcal R}_3 := \ii \begin{pmatrix}
{\mathcal R}^{(3)} & {\mathcal R}^{(3)} \\
- {\mathcal R}^{(3)} & - {\mathcal R}^{(3)}
\end{pmatrix}\,.
\end{equation}
Our aim is to eliminate from the operator ${\mathcal L}_4$ the terms of order $|D|^0$, namely, since ${\mathcal R}^{(3)}$ is an operator of the form \eqref{operatore forma buona resto} (then arbitrarily regularizing), we only need to remove the multiplication operator by the matrix valued function $B_0(\vphi)$. 

\noindent
For this purpose, we consider the operator 
\begin{equation}\label{cal W0}
{\mathcal V} = {\mathcal V}(\vphi) := {\rm exp}\big( \ii V(\vphi) |D|^{- 1} \big)\,, \quad V(\vphi) := \begin{pmatrix}
0 & v(\vphi) \\
- v(\vphi) & 0
\end{pmatrix}\,,
\end{equation}
where $v : \T^\nu \to \R$ is a real valued function to be determined. 
Note that ${\mathcal V}$ is symplectic, since $\ii V(\vphi) |D|^{- 1}$ is a Hamiltonian vector field.  We write 
$$
{\mathcal V} = {\mathbb I}_2 + \ii V |D|^{- 1} + {\mathcal V}_{ \geq 2}\,, \qquad {\mathcal V}_{ \geq 2} := \sum_{k \geq 2} \frac{\ii^k}{k !}V^k |D|^{- k}\,,
$$
hence 
\begin{align}
{\mathcal L}_3 {\mathcal V} & = {\mathcal V} \big( \omega \cdot \partial_\vphi {\mathbb I}_2 + \ii m T |D| \big) + [\ii m T |D|, \ii V |D|^{- 1}] + B_0 + B_0({\mathcal V} - {\mathbb I}_2)  \nonumber\\ 
& \quad + [\ii m T |D|, {\mathcal V}_{ \geq 2}]  +  \ii \omega \cdot \partial_\vphi({\mathcal V} - {\mathbb I}_2)  + {\mathcal R}_3 {\mathcal V} \,. \label{quasi-coniugio cal L3 cal W0} 
\end{align}
The term of order $|D|^0$ is given by
$$
[\ii m T |D|, \ii V(\vphi) |D|^{- 1}] + B_0(\vphi) = \begin{pmatrix}
0 & - 2 m v(\vphi) + b_0(\vphi) \\
- 2 m v(\vphi) + b_0(\vphi) & 0
\end{pmatrix} \,.
$$
In order to remove it, we choose 
\begin{equation}\label{definizione w0(vphi)}
v(\vphi) := \frac{b_0(\vphi)}{2 m}
\end{equation}
and we get 
\begin{equation}\label{cal L4}
{\mathcal L}_4 : ={\mathcal V}^{- 1} {\mathcal L}_4 {\mathcal V}  = \omega \cdot \partial_\vphi {\mathbb I}_2 + \ii m T |D|  + {\mathcal R}_4\,,
\end{equation}
\begin{equation}\label{definition bf R1}
{\mathcal R}_4 := {\mathcal V}^{- 1} \Big(B_0(\vphi)({\mathcal V} - {\mathbb I}_2)  + [\ii m T |D|, {\mathcal V}_{ \geq 2}] +  \omega \cdot \partial_\vphi ({\mathcal V} - {\mathbb I}_2) \Big) + {\mathcal V}^{- 1}{\mathcal R}_3 {\mathcal V}\,.
\end{equation}
Note that, since ${\mathcal L}_3$ is Hamiltonian and ${\mathcal V}$ is symplectic, we have that ${\mathcal L}_4$ is still a Hamiltonian operator. 
In the next lemma we provide some estimates on the transformation ${\mathcal V}$ and on the remainder ${\mathcal R}_4$. 
\begin{lemma}\label{stime prima riducibilita}
There exists $\overline \sigma = \overline \sigma(\tau, \nu) > \sigma  > 0$, where $\sigma$ is the {\it loss of derivatives} in Lemma \ref{stime step 1}, such that if \eqref{ansatz} holds with $\mu = \overline \sigma$, then for any $s_0 \leq s \leq S - \overline \sigma$, ${\mathcal V}^{\pm 1} : {\bf H}^s_0(\T^{\nu + 1}) \to {\bf H}^s_0(\T^{\nu + 1})$ (recall \eqref{definizione bf H s0}) and the following estimates hold: 
\begin{align}
& | ({\mathcal V}^{\pm 1} - {\mathbb I}_2) |D| |_{ s} \leq_s \e (1 + \|u \|_{s + \overline\sigma})\,, \label{stima cal W0} \\
& | ({\mathcal V}^{\pm 1} - {\mathbb I}_2) |D| |_{ s}^\Lipg \leq_s \e (1 + \|u \|_{s + \overline\sigma}^\Lipg)\,, \label{stima cal W01} \\
& | \partial_u {\mathcal V}^{\pm 1} h  |_{ s} \leq_s \e (\| h \|_{s + \overline\sigma} + \| u \|_{s + \overline\sigma}\| h \|_{s_0 + \overline \sigma})\,,\label{stima derivata cal V} \\
& | {\mathcal R}_4 |D| |_{ s} \leq_s \e (1 + \|u \|_{s + \overline\sigma})\,, \quad | {\mathcal R}_4 |D| |_{ s}^\Lipg \leq_s \e (1 + \|u \|_{s + \overline\sigma}^\Lipg)\,, \label{stima bf R1} \\
& \quad | \partial_u {\mathcal R}_4 [h] |_{ s} \leq_s \e (\| h \|_{s + \overline\sigma} + \| u \|_{s + \overline\sigma}\| h \|_{s_0 + \overline \sigma})\,. \label{stima derivata cal R5}
\end{align}
\end{lemma}

\begin{proof}
{\sc Proof of \eqref{stima cal W0}-\eqref{stima derivata cal V}.} 
By Lemma \ref{decadimento moltiplicazione} one has 
\begin{equation}\label{spappolare - 1}
|(V |D|^{- 1}) |D||_s = \| V\|_s \leq \| v \|_s \stackrel{\eqref{stime m}, \eqref{stime a2}}{\leq_s} \e (1 + \| u \|_{s + \sigma})\,.
\end{equation}
By \eqref{ansatz} we have that $|(V |D|^{- 1}) |D||_{s_0} = \| V\|_{s_0} \lessdot \e \leq 1$, for $\e$ small enough, then Lemma \ref{lem:inverti} can be applied and the claimed estimate \eqref{stima cal W0} follows. The estimate \eqref{stima derivata cal V} follows by applying the estimate \eqref{derivata-inversa-Phi} and using that by  \eqref{definizione w0(vphi)}, \eqref{stime m}, \eqref{stime derivata b0}, $||D|^{- 1}|_s \leq 1$
$$
|\partial_u V |D|^{- 1}|_s \lessdot \| \partial_u v[h] \|_s \leq_s \e (\| h \|_{s + \sigma} +\| u \|_{s + \sigma} \| h \|_{s_0 + \sigma})\,.
$$

\noindent
{\sc Proof of \eqref{stima bf R1}, \eqref{stima derivata cal R5}.} The claimed estimates follow by the definition \eqref{definition bf R1},  by applying Lemmata \ref{elementarissimo decay}, \ref{interpolazione decadimento Kirchoff}, \ref{decadimento moltiplicazione}, \ref{lem:inverti}. and by the estimates \eqref{stime a2}-\eqref{stime derivata cal R3} and \eqref{stima cal W0}-\eqref{stima derivata cal V}. 
\end{proof}
\begin{lemma}\label{stime Hs x cal v}
Assume \eqref{ansatz} with $\mu = \overline \sigma + s_0$. Then for any $s_0 \leq s \leq S - \overline \sigma - s_0$ for any $\vphi \in \T^\nu$, ${\mathcal V}^{\pm 1}(\vphi) : {\bf H}^s_0(\T_x) \to {\bf H}^s_0(\T_x)$ (recall \eqref{definizione bf H s01}) and 
$$
|{\mathcal V}^{\pm 1}(\vphi)|_{s, x} \leq_s 1 + \| u \|_{s + \overline \sigma + s_0}\,.
$$
\end{lemma}
\begin{proof}
The claimed estimate follows by applying Lemma \ref{lemma decadimento Kirchoff in x}-$(ii)$ and by the estimates \eqref{stima cal W0}. 
\end{proof}
\section{$2 \times 2$ block-diagonal reduction}\label{sec:redu}
The goal of this section is to block-diagonalize the linear Hamiltonian operator ${\mathcal L}_5$  obtained in \eqref{cal L4}. We are going to perform an iterative Nash-Moser reducibility scheme for the linear Hamiltonian operator  
\be\label{L0}
{\mathcal L}_{0} := {\mathcal L}_4 = \o \cdot \partial_{\vphi} {\mathbb  I}_2  + {\mathcal D}_{0} + {\mathcal R}_{0} : {\bf H}_0^1(\T^{\nu + 1}) \to {\bf L}^2_0(\T^{\nu + 1}) \, , 
\ee
 where 
 \begin{equation}\label{cal D0 riducibilita}
 {\mathcal D}_0 = \ii \begin{pmatrix}
 {\mathcal D}_0^{(1)} & 0 \\
 0 & - {\mathcal D}_0^{(1)}
 \end{pmatrix}\,, \quad {\mathcal D}_0^{(1)} :=  m |D| = {\rm diag}_{j \in \Z \setminus \{ 0 \}}m |j|  
 \end{equation}
  and ${\mathcal R}_0 := {\mathcal R}_4$ is a Hamiltonian operator of the form 
  \begin{equation}\label{cal R0 riducibilita}
 {\mathcal R}_0 = \ii \begin{pmatrix}
 {\mathcal R}_0^{(1)} & {\mathcal R}_0^{(2)} \\
 - \overline{\mathcal R}_0^{(2)} & - \overline{\mathcal R}_0^{(1)}
 \end{pmatrix}\,, \quad {\mathcal R}_0^{(1)} = \big( {\mathcal R}_0^{(1)} \big)^*\,, \quad {\mathcal R}_0^{(2)} = ( {\mathcal R}_0^{(2)} )^T
 \end{equation}
 satisfying, by \eqref{stima bf R1}, for any $s_0 \leq s \leq S - \bar \sigma$ the estimates 
 \begin{equation}\label{stime primo resto KAM}
 |{\mathcal R}_0 |D| |_s^\Lipg \leq_s \e (1 + \| {u}\|_{s + \overline \sigma})\,,\, |\partial_{u} {\mathcal R}_0 [h]|_s \leq_s \e \Big( \| h \|_{s +  \overline \sigma} + \| { u}\|_{s +  \overline \sigma} \| h \|_{s_0 + \overline \sigma} \Big)\,,
 \end{equation}
where $\overline \sigma$ is the loss of derivatives given in Lemma \ref{stime prima riducibilita}. We define
\be\label{defN}
N_{-1} := 1 \, ,  \quad 
N_{\nu} := N_{0}^{\chi^{\nu}} \  \forall \nu \geq 0 \, ,  \quad  
\chi := 3 /2  
\ee
(then $ N_{\nu+1} = N_{\nu}^\chi $, $ \forall \nu \geq 0 $) and 
\begin{equation}\label{definizione alpha}
\mathtt a := 6 \tau + 4\,, \quad \mathtt b := \mathtt a + 1\,.
\end{equation}
  We assume that \eqref{ansatz} holds with $\mu = \bar \sigma + \mathtt b$, so that by \eqref{stime primo resto KAM}
  \begin{equation}\label{stima R0 s0 + mathtt b}
  |{\mathcal R}_0 |D||_{s_0 + \mathtt b}^\Lipg \lessdot \e \,, \qquad |\partial_u {\mathcal R}_0 [h]|_{s_0 + \mathtt b} \lessdot \e \|  h \|_{s_0 + \bar \sigma + \mathtt b}\,.
  \end{equation}
For the reducibility Theorem below, we use the $2 \times 2$ block representation of linear operators, given in Section \ref{sezione rappresentazione 2 per 2}. According to \eqref{notazione a blocchi} and recalling also \eqref{notazione operatore diagonale a blocchi}, the operator ${\mathcal D}_0^{(1)}$ can be written as 
 \begin{equation}\label{rappresentazione a blocchi cal D0 (1)}
 {\mathcal D}_0^{(1)} = {\rm diag}_{j \in \N} m j {\bf I}_j\,, 
 \end{equation}
   where ${\bf I}_j : {\bf E}_j \to {\bf E}_j$ is the identity, ${\bf E}_j = {\rm span}\{ e^{\ii j x}, e^{- \ii j x}\}$ is the two dimensional space \eqref{bf E alpha} and the real constant $m$ satisfies the estimates \eqref{stime m}. We also recall the definition of the space ${\mathcal S}({\bf E}_j)$ given in \eqref{cal S E alpha} which is isomorphic to the space of the $2 \times 2$ self-adjoint matrices, the definition of the norm $\| \cdot \|_{{\rm Op}(j, j')}$ given in \eqref{norma operatoriale su matrici alpha beta}, the identity ${\bf I}_{j, j'}$ in \eqref{operatore identita matrici alpha beta}, the definition of $M_L(A)$ in \eqref{definizione moltiplicazione sinistra matrici} and the definition of $M_R(B)$ in \eqref{definizione moltiplicazione destra matrici}. Now we are ready to state the following
 \begin{theorem}{{\bf (KAM reducibility)}} \label{thm:abstract linear reducibility}  
Let $\gamma \in (0, 1)$ and $\tau > 0$. Assume \eqref{ansatz} with $\mu = \bar \sigma + \mathtt b$ and with $S \geq s_0 + \overline \sigma + \mathtt b$. 
There exist $ N_{0 } = N_0(S, \tau, \nu) >0$ large enough, $\delta_0 = \delta_0(S, \tau, \nu) \in (0, 1)$ small enough, such that, if 
\begin{equation}\label{piccolezza1}
\e \gamma^{- 1} \leq \delta_0 
\end{equation}
then:

\begin{itemize}
\item[${\bf(S1)_{\nu}}$] 
For all $\nu \geq 0$, there exists an operator 
\be\label{def:Lj}
{\mathcal L}_\nu := \o \cdot \partial_{\vphi} {\mathbb I}_2 + {\mathcal D}_\nu + {\mathcal R}_\nu \quad 
\ee
\begin{equation}\label{cal D nu}
{\mathcal D}_\nu = \ii \begin{pmatrix}
 {\mathcal D}^{(1)}_\nu & 0 \\
0 & -  \overline{\mathcal D}^{(1)}_\nu
\end{pmatrix}\,,\qquad {\mathcal D}_\nu^{(1)} := {\rm diag}_{j \in \N } {\bf D}_j^\nu\,, 
\end{equation}
\be\label{mu-j-nu}
{\bf D}_j^\nu := {\bf D}_j^\nu (\omega ) = {\bf D}_j^0(\omega) + \widehat {\bf D}_j^\nu(\omega)\,, \quad {\bf D}_j^0 := m j {\bf I}_j\,, \quad \forall j \in \N\,,
\ee
(with $\widehat{\bf D}_j^0 = 0$) defined for all $ \omega \in \Omega_{\nu}^{\g}({\bf u})$, where 
$\Omega_{0}^{\g}({\bf u}) :=  \Omega_o = \Omega_o({\bf u})$, and for $\nu \geq 1$, $\Omega_\nu^\gamma = \Omega_\nu^\gamma({\bf u})$ is defined by
\begin{align}
\Omega_{\nu}^{\g} & := \Big\{\omega \in \Omega_{\nu - 1}^{\gamma} : \| {\bf A}_{\nu - 1}^{-}(\ell, j, j')^{- 1} \|_{{\rm Op}(j, j')} \leq \frac{\langle \ell \rangle^\tau}{\gamma \langle j - j'\rangle}\,, \forall (\ell,j, j') \neq (0, j, j)\,, \nonumber\\
& \qquad |\ell| \leq N_{\nu - 1}\,, \, \| {\bf A}_{\nu - 1}^+(\ell, j, j')^{- 1} \|_{{\rm Op}(j, j')} \leq \frac{\langle \ell \rangle^\tau}{\gamma \langle j + j' \rangle}\,, \nonumber\\
& \qquad \qquad  \forall (\ell, j, j')\,,  |\ell| \leq N_{\nu - 1} \Big\}\,, \label{Omgj}
\end{align}
where the operators ${\bf A}^{\pm}_{\nu - 1}(\ell, j, j') : {\mathcal L}({\bf E}_{j'}, {\bf E}_j) \to  {\mathcal L}({\bf E}_{j'}, {\bf E}_j)$ are defined by
\begin{equation}\label{bf A nu - (ell,alpha,beta)}
{\bf A}^{-}_{\nu - 1}(\ell,j, j') := \omega \cdot \ell {\bf I}_{j,  j'} + M_L({\bf D}_j^{\nu - 1}) - M_R({\bf D}_{j'}^{\nu - 1})\,,
\end{equation}
\begin{equation}\label{bf A nu + (ell,alpha,beta)}
{\bf A}^{+}_{\nu - 1}(\ell,  j, j') := \omega \cdot \ell {\bf I}_{j,  j'} + M_L({\bf D}_j^{\nu - 1}) + M_R(\overline{\bf D}_{j'}^{\nu - 1})\,.
\end{equation}
For $\nu \geq 0$, for all $j \in \N$, the $2 \times 2$ self-adjoint block $\widehat {\bf D}_j^\nu \in {\mathcal S}({\bf E}_j)$ satisfies 
\begin{equation}  \label{rjnu bounded}
 \| \widehat {\bf D}_j^\nu \|^\Lipg  \lessdot \e  j^{- 1} \, \quad \forall j \in \N\,. 
\end{equation}
The Hamiltonian remainder $  {\mathcal R}_\nu : {\bf H}^s_0(\T^{\nu + 1}) \to {\bf H}^s_0(\T^{\nu + 1}) $ satisfies $ \forall s \in [ s_0, S - \overline\s - \mathtt b ] $,  
\be\label{Rsb} 
|{\mathcal R}_{\nu} |D| |_{s}^\Lipg \leq \frac{|{\mathcal R}_{0} |D | |_{s+\mathtt b}^\Lipg}{ N_{\nu - 1}^{\mathtt a}} \, ,\,
|{\mathcal R}_{\nu} |D| |_{s + \mathtt b}^\Lipg \leq {|{\mathcal R}_{0} |D| |_{s+ \mathtt b}^\Lipg}{N_{\nu - 1}}\,.
\ee
Moreover, for any $ \nu \geq 1 $, 
\be	\label{Lnu+1}
{\mathcal L}_{\nu} = \Phi_{\nu-1}^{-1} {\mathcal L}_{\nu-1} \Phi_{\nu-1} \, , \quad \Phi_{\nu-1} := {\rm exp}(\Psi_{\nu - 1}) \, , 
\ee
where $\Psi_{\nu - 1}$ is Hamiltonian, $\Phi_{\nu - 1}$ is symplectic and they satisfy $ \Psi_{\nu-1}, \Phi_{\nu - 1}^{\pm 1} :{\bf H}^s_0(\T^{\nu + 1}) \to {\bf H}^s_0(\T^{\nu + 1}) $, 
\be\label{Psinus}  
 |\Psi_{\nu - 1}|_s^\Lipg\,, \,\left|\Psi_{\nu-1}|D| \right|_{s}^\Lipg \leq 
|{\mathcal R}_{0} |D| |_{s+\mathtt b}^\Lipg \gamma^{-1} N_{\nu-1}^{2 \t+1} N_{\nu - 2}^{- \mathtt a}  \, . 
\ee

\item[${\bf(S2)_{\nu}}$] 
For all $ j \in \N$,  there exists a Lipschitz extension to the whole parameter space $\Omega_o$, 
$ \widetilde{\bf D}_{j}^{\nu}(\cdot ): \Omega_o \to {\mathcal S}({\bf E}_j) $ of  $ {\bf D}_{j}^{\nu}(\cdot ) : \Omega_\nu^\g \to {\mathcal S}({\bf E}_j)$ 
satisfying,  for $\nu \geq 1$, 
\be\label{lambdaestesi}  
\|\widetilde{\bf D}_{j}^{\nu} -  \widetilde{\bf D}_{j}^{\nu-1} \|^\Lipg \lessdot  j^{- 1} | {\mathcal R}_{\nu-1} |D | |^\Lipg_{s_0}  \lessdot  N_{\nu - 2}^{- \mathtt a} \e j^{- 1}  \,.
\ee
\item[ ${\bf (S3)_{\nu}}$] 
Let $ {\bf u}_i(\omega) = (u_i(\omega), \psi_i(\omega)) $, $i = 1,2$ be Lipschitz families of Sobolev functions in $H^{s_0 + \bar \sigma + \mathtt b}(\T^{\nu + 1}, \R^2)$, 
defined for $\omega \in \Omega_o$ satisfying \eqref{ansatz} with $\mu = \bar \sigma + \mathtt b$. Then there exists a constant $K_0> 0$ such that, for $ \nu \geq 0 $,   $\forall \omega \in \Omega_{\nu}^{\g_1}({\bf u}_1) \cap  \Omega_{\nu}^{\g_2}({\bf u}_2)$, 
with $ \g_1, \g_2 \in [\g/2, 2\g]$, 
\begin{align}
& | {\mathcal R}_{\nu}(u_1) - {\mathcal R}_\nu(u_2)|_{ s_{0}}\leq  K_0 N_{\nu - 1}^{-\mathtt a} \e \Vert u_1 - u_2 \Vert_{s_0 + \overline\s + \mathtt b}, \label{derivate-R-nu}\\
& | {\mathcal R}_{\nu}(u_1) - {\mathcal R}_\nu(u_2)|_{ s_{0}+\mathtt b} \leq  K_0  N_{\nu - 1} \e \Vert u_1 - u_2 \Vert_{s_0 + \overline \s + \mathtt b} \, .\label{derivate-R-nu1}
\end{align} 
Moreover, for $\nu \geq 1$,  
$\forall j \in\N $, 
\begin{align}
& \big\|\big(\widehat{\bf D}_{j}^{\nu}(u_2) - \widehat{\bf D}_{j}^{\nu}(u_1)\big) - \big(\widehat{\bf D}_{j}^{\nu-1}(u_2) - \widehat{\bf D}_{j}^{\nu-1}(u_1)\big) \big\| \nonumber\\
& \qquad \leq  \vert {\mathcal R}_{\nu -1}(u_2) - \mR_{\nu -1}(u_1) \vert_{s_0} \,, \label{deltarj12}
\end{align}
\be\label{Delta12 rj}
\| \widehat{\bf D}_j^{\nu}(u_2) - \widehat{\bf D}_j^{\nu}(u_1) \| \lessdot  \e  \|  u_1 - u_2 \|_{s_0 + \overline \s + \mathtt b} \, . 
\ee 

\item[ ${\bf (S4)_{\nu}}$]
Let ${\bf u}_1, {\bf u}_2$ like in $({\bf S3})_\nu $ and $ 0 < \rho < \g / 2 $. For all $ \nu \geq 0 $
\begin{equation}\label{legno}
\e K_1 N_{\nu - 1}^{\tau} \Vert u_1 - u_2 \Vert_{s_0 + \overline \sigma + \mathtt b}^{\rm sup} \leq \rho 
\quad \Longrightarrow \quad
 \Omega_{\nu }^{\gamma}({\bf u}_1) \subseteq 
 \Omega_{\nu}^{\gamma - \rho}({\bf u}_2) \, ,
\end{equation}
where $K_1$ is a suitable constant depending on $\tau$ and $\nu$.
\end{itemize}
\end{theorem}

\subsection{Proof of Theorem \ref{thm:abstract linear reducibility}} \label{subsec:proof of thm abstract linear reducibility}

\noindent
{\sc Proof  of ${\bf ({S}i)}_{0}$, $i=1,\ldots,4$.} 
Properties \eqref{def:Lj}-\eqref{Rsb} in ${\bf({S}1)}_0$
hold by \eqref{L0}-\eqref{cal R0 riducibilita} with $ {\bf D}^0_j $ defined in \eqref{mu-j-nu} and
 $ \widehat{\bf D}_j^0(\omega) = 0 $ (for \eqref{Rsb} recall  that $ N_{-1} := 1 $, see \eqref{defN}). 
Moreover, since $m$ is a real function, ${\bf D}_j^0$ is self-adjoint.  
Then there is nothing else to verify.  

${\bf({S}2)}_0 $ holds, since the function $ m(\omega)= m(\omega, u(\omega))$ is already defined for all $\omega \in \Omega_o = \Omega_o({\bf u})$. 

${\bf({S}3)}_0 $ follows by the estimate \eqref{stima R0 s0 + mathtt b} and by the mean value Theorem, by taking $K_0 > 0$ large enough. 

${\bf({S}4)}_0 $ is trivial because, by definition, $\Omega_0^\g({\bf u}_1) : = \Omega_o = :\Omega_0^{\g-\rho}({\bf u}_2)$.

\subsection{The reducibility step}
We now describe the inductive step, showing how to define a symplectic transformation 
$ \Phi_\nu := \exp(  \Psi_\nu ) $ 
so that  $ {\mathcal L}_{\nu+1 } = \Phi_\nu^{- 1} {\mathcal L}_\nu \Phi_\nu $ has the desired properties.
To simplify notations, in this section we drop the index $ \nu $ and we write $ + $ for $ \nu + 1$.
At each step of the iteration we have a Hamiltonian operator 
\begin{equation}\label{operator KAM step}
{\mathcal L} = \omega \cdot \partial_\vphi {\mathbb I}_2 + {\mathcal D} + {\mathcal R}\,
\end{equation}
where
\begin{equation}\label{diagonal operator KAM step}
{\mathcal D} := 
\ii \begin{pmatrix}
 {\mathcal D}^{(1)} & 0 \\
0 & -  \overline{\mathcal D}^{(1)}
\end{pmatrix}\,, \qquad {\mathcal D}^{(1)} := {\rm diag}_{j \in \N} {\bf D}_j \,,\quad 
\end{equation}
${\bf D}_j \in {\mathcal S}({\bf E}_j)$, $\forall j \in \N$ (recall the definiton \eqref{cal S E alpha}) and ${\mathcal R}$ is a Hamiltonian operator, namely it has the form 
\begin{equation}\label{cal R Hamiltoniano}
{\mathcal R} = \ii \begin{pmatrix}
{\mathcal R}^{(1)} & {\mathcal R}^{(2)} \\
- \overline{\mathcal R}^{(2)} & - \overline{\mathcal R}^{(1)}
\end{pmatrix}\,, \qquad {\mathcal R}^{(1)} = ({\mathcal R}^{(1)})^*\,, \qquad {\mathcal R}^{(2)} = ({\mathcal R}^{(2)})^T\,.
\end{equation}
Let us consider a transformation 
\begin{equation}\label{Phi Psi}
\Phi := {\rm exp}(\Psi)\,,\qquad \Psi :=  \ii \begin{pmatrix}
\Psi^{(1)} & \Psi^{(2)} \\
- \overline\Psi^{(2)} & - \overline\Psi^{(1)}
\end{pmatrix}\,,
\end{equation}
with $\Psi^{(1)} = (\Psi^{(1)})^*$, $\Psi^{(2)}= (\Psi^{(2)})^T$. Writing 
\begin{equation}\label{espansione Phi}
\Phi = {\mathbb I}_2 + \Psi + \Phi_{\geq 2}\,,\quad \Phi_{\geq 2} := \sum_{k \geq 2} \frac{\Psi^k}{k !}\,,
\end{equation}
we have 
\begin{eqnarray}
{\mathcal L} \Phi & = & \Phi \Big( \omega \cdot \partial_\vphi {\mathbb I}_2 + {\mathcal D} \Big) + \Big(\omega \cdot \partial_\vphi \Psi + [{\mathcal D}, \Psi] + \Pi_N {\mathcal R} \Big) + \Pi_N^\bot {\mathcal R} \nonumber\\
& & + \omega \cdot \partial_\vphi \Phi_{\geq 2} + [{\mathcal D}, \Phi_{\geq 2}] + {\mathcal R} (\Phi - I)\,, \label{quasi coniugio riducibilita}
\end{eqnarray}
We want to determine the operator $\Psi$ so that 
\begin{equation}\label{Homological equation}
\omega \cdot \partial_\vphi \Psi + [{\mathcal D}, \Psi] + \Pi_N {\mathcal R} = [{\mathcal R}],
\end{equation}
where, recalling the notation \eqref{notazione operatore diagonale a blocchi}, 
\begin{equation}\label{media di cal R}
[{\mathcal R}] := \ii \begin{pmatrix}
[{\mathcal R}^{(1)}] & 0 \\
0 &- [{\mathcal R}^{(1)}]
\end{pmatrix}\,,\qquad 
[{\mathcal R}^{(1)}] := {\rm diag}_{j \in \N} ({\bf R}^{(1)})_j^j(0)\,.
\end{equation}
We recall that, according to \eqref{notazione a blocchi}, the operator $({\bf R}^{(1)})_j^j(0)$, $j \in \N$ is identified with its $2 \times 2$ matrix representation
$$
({\bf R}^{(1)})_j^j(0) = \big( ({\mathcal R}^{(1)})_k^{k'}(0) \big)_{k, k' = \pm j}\,.
$$
Since ${\mathcal R}^{(1)}$ is self-adjoint, all the $2 \times 2$ blocks $({\bf R}^{(1)})_j^j(0)$ are self-adjoint and then also $[{\mathcal R}^{(1)}]$ is self-adjoint.
\begin{lemma} {\bf (Homological equation)}\label{homologica equation} 
For all $ \omega \in \Omega_{\nu + 1}^{{\gamma}}$ (see \eqref{Omgj}),  there exists a unique solution
$ \Psi  $ of the homological equation \eqref{Homological equation}, which is Hamiltonian and satisfies 
\be\label{PsiR}
|\Psi |D| |_ s^\Lipg \lessdot   N^{2 \tau + 1} \gamma^{-1} |{\mathcal R} |D| |_s^\Lipg \, . 
\ee
Moreover if $\gamma/2 \leq \gamma_1, \gamma_2 \leq 2 \gamma$, and if ${\bf u}_i(\omega) = (u_i(\omega), \psi_i(\omega)) \in H^{s_0 + \overline \sigma + \mathtt b}(\T^{\nu + 1}, \R^2)$, $i = 1, 2$ are Lipschitz families, then for all $s \in [s_0, s_0 + \mathtt b]$, for all $\omega \in \Omega_{\nu + 1}^{\gamma_1}({\bf u}_1) \cap \Omega_{\nu + 1}^{\gamma_2}({\bf u}_2)$ 
\begin{align}
& |\Psi(u_1) - \Psi(u_2)|_s \label{Psi(u1) - Psi(u2)}\\
& \leq_s N^{2 \tau + 1} \gamma^{- 1} \Big(|{\mathcal R}(u_1)|_s \|u_1 - u_2 \|_{s_0 + \overline \sigma + \mathtt b} + |{\mathcal R}(u_1) - {\mathcal R}(u_2)|_s \Big)\,. \nonumber
\end{align}
\end{lemma} 
\begin{proof}
Recalling \eqref{cal R Hamiltoniano}, \eqref{Phi Psi}, The equation \eqref{Homological equation} is splitted in the two equations 
\begin{equation}\label{prima equazione omologica}
\ii \omega \cdot \partial_\vphi \Psi^{(1)} + [\Psi^{(1)}, {\mathcal D}^{(1)}] + \ii \Pi_N {\mathcal R}^{(1)} = \ii [{\mathcal R}^{(1)}]\,,
\end{equation}
\begin{equation}\label{seconda equazione omologica}
\ii \omega \cdot \partial_\vphi \Psi^{(2)} - ({\mathcal D}^{(1)} \Psi^{(2)} + \Psi^{(2)} \overline{\mathcal D}^{(1)}) + \ii \Pi_N {\mathcal R}^{(2)} = 0\,.
\end{equation}
Using the decomposition \eqref{notazione a blocchi}, the equations \eqref{prima equazione omologica}, \eqref{seconda equazione omologica} become,
for all $j, j' \in \N$, $\ell \in \Z^\nu$ such that $|\ell| \leq N$, 
\begin{equation}\label{equazione omologica blocco psi diagonale}
 \omega \cdot \ell ({\bf \Psi}^{(1)})_j^{j'}(\ell) + {\bf D}_j ({\bf \Psi}^{(1)})_j^{j'}(\ell) - ({\bf \Psi}^{(1)})_j^{j'}(\ell){\bf D}_{j'}    =\ii ({\bf R}^{(1)})_j^{j'}(\ell) - \ii [{\bf R}^{(1)}]_j^{j'}\,,
\end{equation}
\begin{equation}\label{equazione omologica blocco psi fuori diagonale}
 \omega \cdot \ell ({\bf \Psi}^{(2)})_j^{j'}(\ell) + {\bf D}_j({\bf\Psi}^{(2)})_j^{j'}(\ell) + ({\bf \Psi}^{(2)})_j^{j'}(\ell) \overline{\bf D}_{j'} =  \ii ({\bf R}^{(2)})_j^{j'}(\ell)\,.
\end{equation}
By the Definitions \eqref{bf A nu - (ell,alpha,beta)}, \eqref{bf A nu + (ell,alpha,beta)}, 
the equations \eqref{equazione omologica blocco psi diagonale}, \eqref{equazione omologica blocco psi fuori diagonale} can be written in the form 
$$
{\bf A}^{-}(\ell, j, j') ({\bf \Psi}^{(1)})_j^{j'}(\ell)  = \ii ({\bf R}^{(1)})_j^{j'}(\ell) - \ii [{\bf R}^{(1)}]_j^{j'}\,, \quad
$$
$$
 {\bf A}^{+}(\ell,j, j') ({\bf \Psi}^{(2)})_j^{j'}(\ell) = \ii ({\bf R}^{(2)})_j^{j'}(\ell)\,.
$$ 
Then, since $\omega \in \Omega_{\nu + 1}^{\gamma}$, we can define, $\forall (\ell, j, j') \in \Z^\nu \times \N \times \N\,,\quad  (\ell, j, j') \neq (0, j, j)\,, \quad |\ell| \leq N$,
\begin{equation}\label{soluzione equazione omologica 1}
({\bf \Psi}^{(1)})_j^{j'}(\ell) = \ii {\bf A}^{-}(\ell, j, j')^{- 1} ({\bf R}^{(1)})_j^{j'}(\ell)\,, \quad 
\end{equation}
with the normalization $({\bf \Psi}^{(1)})_j^j(0) = 0$, and $\forall (\ell, j, j') \in \Z^\nu \times \N \times \N\,, \quad |\ell| \leq N$,
\begin{equation}\label{soluzione equazione omologica 2}
({\bf \Psi}^{(2)})_j^{j'}(\ell) = \ii {\bf A}^{+}(\ell,j , j')^{- 1} ({\bf R}^{(2)})_j^{j'}(\ell)\,.
\end{equation}
Since 
$$
\|{\bf A}^{-}(\ell, j, j')^{- 1} \|_{{\rm Op}(j, j')} \leq \frac{\langle \ell \rangle^\tau}{\gamma \langle j-  j' \rangle}\,, \quad \|{\bf A}^{+}(\ell, j, j')^{- 1} \|_{{\rm Op}(j, j')} \leq \frac{\langle \ell \rangle^\tau}{\gamma \langle j +j' \rangle}\,,
$$
(recall \eqref{norma operatoriale su matrici alpha beta}) we get immediately that 
\begin{equation}\label{stima norma sup coefficienti Psi}
\| ({\bf \Psi}^{(i)})_j^{j'}(\ell)\| \leq N^\tau \gamma^{- 1} \| ({\bf R}^{(i)})_j^{j'}(\ell)\|\,, \quad i = 1, 2\,.
\end{equation}
Now, let $\omega_1, \omega_2  \in \Omega_{\nu + 1}^{\gamma}$. As a notation for any function $f = f(\omega)$, we write $\Delta_{\omega} f := f(\omega_1) - f(\omega_2)\,.$
 
 \noindent
 By \eqref{soluzione equazione omologica 1}, one has 
\begin{align}
\Delta_{\omega} ({\bf \Psi}^{(1)})_j^{j'}(\ell ) & = \ii \big\{ \Delta_\omega {\bf A}^{-}(\ell, j, j')^{- 1}  \big\} ({\bf R}^{(1)})_j^{j'}(\ell; \omega_1)   \nonumber\\
& \quad +\ii {\bf A}^{-}(\ell, j, j'; \omega_2)^{- 1}  \big\{\Delta_\omega ({\bf R}^{(1)})_j^{j'}(\ell ) \big\}\,. \label{Delta omega 12 coefficienti Psi}
\end{align}
 The second term in the above formula satisfies
\begin{equation}\label{stima primo pezzo Delta omega 12 coefficienti Psi}
\| {\bf A}^{-}(\ell, j, j'; \omega_2)^{- 1}  \big\{ \Delta_\omega ({\bf R}^{(1)})_j^{j'}(\ell) \big\} \| \leq N^\tau \gamma^{- 1} \| \Delta_\omega ({\bf R}^{(1)})_j^{j'}(\ell)  \|\,,
\end{equation}
hence it remains to estimate only the first term in \eqref{Delta omega 12 coefficienti Psi}. We have 
\begin{align}
& \Delta_\omega {\bf A}^{-}(\ell, j, j')^{- 1} \label{bernardino 0} \\
& = - {\bf A}^{-}(\ell, j, j' ; \omega_1)^{- 1}\big\{ \Delta_\omega {\bf A}^{-}(\ell, j, j' )  \big\}{\bf A}^{-}(\ell, j, j' ; \omega_2)^{- 1}\,,\nonumber
\end{align}
therefore 
\begin{equation}\label{norma A - omega 12 inverso}
\|\Delta_\omega  {\bf A}^{-}(\ell, j, j')^{- 1} \|_{{\rm Op}(j, j')}  \leq \frac{N^{2 \tau} }{\gamma^2 \langle j - j' \rangle^2} \| \Delta_\omega {\bf A}^{-}(\ell, j, j')  \|_{{\rm Op}(j, j')} \,.
\end{equation}
Moreover 
\begin{align}
\Delta_\omega {\bf A}^{- }(\ell, j, j') & = (\omega_1 - \omega_2) \cdot \ell \,\,{\bf I}_{j,  j'} + M_L(\Delta_\omega{\bf D}_j) - M_R(\Delta_\omega{\bf D}_{j'})\label{Delta omega 12 bf A -}
\end{align}
and using that, by \eqref{mu-j-nu}, \eqref{rjnu bounded}
\begin{equation}\label{stima bf widehat D alpha (omega)}
{\bf D}_j(\omega) =  m(\omega) \,j  {\bf I}_j + \widehat{\bf D}_j(\omega)\,, \quad \text{with} \quad \| \widehat{\bf D}_j \|^\Lipg \lessdot  \e j^{- 1}\,, \quad \forall  j\in \N\,,
\end{equation}
we get 
\begin{align}
M_L(\Delta_\omega {\bf D}_j) - M_R(\Delta_\omega{\bf D}_{j'}) & = (\Delta_\omega m)\, (j - j'){\bf I}_{j,  j'} + M_L(\Delta_\omega \widehat{\bf D}_j) - M_R(\Delta_\omega \widehat{\bf D}_{j'})\,. \nonumber
\end{align}
By \eqref{stime m}, \eqref{stima bf widehat D alpha (omega)} and using the property \eqref{norma operatoriale ML MR} one gets 
\begin{align}
\| M_L(\Delta_\omega {\bf D}_j) - M_R(\Delta_\omega{\bf D}_{j'}) \|_{{\rm Op}(j, j')} & {\lessdot}  \e \gamma^{- 1} \langle j - j' \rangle |\omega_1 - \omega_2 |\,. \label{ML MR Delta omega 12}
\end{align}
Recalling \eqref{Delta omega 12 bf A -}, we get the estimate 
$$
\| \Delta_\omega {\bf A}^{- }(\ell, j, j') \|_{{\rm Op}(j, j')}  \lessdot \Big(\langle \ell \rangle + \e \gamma^{- 1} \langle j - j' \rangle \Big) |\omega_1 - \omega_2|\,, 
$$
which implies, by \eqref{norma A - omega 12 inverso}, for $\e \gamma^{- 1} \leq 1$ that 
$$
\| \Delta_\omega  {\bf A}^{-}(\ell, j, j')^{- 1} \|_{{\rm Op}(j, j')}  \leq N^{2 \tau + 1} \gamma^{- 2} |\omega_1 - \omega_2| \,.
$$
By \eqref{Delta omega 12 coefficienti Psi}, \eqref{stima primo pezzo Delta omega 12 coefficienti Psi}, we get the estimate 
\begin{align}\label{stima Delta omega 12 coefficienti Psi finale}
 \| \Delta_\omega ({\bf \Psi}^{(1)})_j^{j'}(\ell)\| & \leq N^\tau \gamma^{- 1} \| \Delta_\omega ({\bf R}^{(1)})_j^{j'}(\ell)  \| \\
 & \quad + N^{2 \tau + 1} \gamma^{- 2} \| ({\bf R}^{(1)})_j^{j'}(\ell; \omega_1) \|\,. \nonumber
\end{align}
Thus \eqref{stima norma sup coefficienti Psi}, \eqref{stima Delta omega 12 coefficienti Psi finale} and the definition \eqref{decadimento Kirchoff} imply 
$$
|\Psi^{(1)} |D| |_s^\Lipg \lessdot N^{2 \tau + 1} \gamma^{- 1} |{\mathcal R}^{(1)} |D||_s^\Lipg\,.
$$
The estimate of $\Psi^{(2)}$ in terms of ${\mathcal R}^{(2)}$ follows by similar arguments and then \eqref{PsiR} follows. 

\noindent
Now we prove the estimate \eqref{Psi(u1) - Psi(u2)}. As a notation, we write $\Delta_{12} A:= A(u_1) - A(u_2)$, for any operator $A$ depending on $u$. 
We prove the estimate \eqref{Psi(u1) - Psi(u2)} for the operator $\Psi^{(1)}$. The estimate for $\Psi^{(2)}$ is analogous. By \eqref{soluzione equazione omologica 1}, for all $ j, j' \in \N$, $\ell \in \Z^\nu$, $(\ell, j, j') \neq (0,j , j)$, $|\ell| \leq N$ one has 
\begin{align}\label{australia - 1}
\Delta_{12} ({\bf \Psi}^{(1)})_j^{j'}(\ell) & = \ii \big\{ \Delta_{12} {\bf A}^{-}(\ell, j, j')^{- 1}  \big\} ({\bf R}^{(1)})_j^{j'}(\ell; u_1) \\
& \quad +  \ii {\bf A}^{-}(\ell, j, j'; u_2)^{- 1} \big\{\Delta_{1 2} ({\bf R}^{(1)})_{j}^{j'}(\ell) \big\}\,. \nonumber
\end{align}
Since $\omega \in \Omega_{\nu + 1}^{\gamma_1}({\bf u}_1) \cap \Omega_{\nu + 1}^{\gamma_2}({\bf u}_2)$ and $\gamma/2 \leq \gamma_1, \gamma_2 \leq 2 \gamma$, we have 
\begin{equation}\label{australia 0}
\|  {\bf A}^{-}(\ell, j, j'; u_2)^{- 1}  \big\{\Delta_{1 2} ({\bf R}^{(1)})_j^{j'}(\ell) \big\} \| \leq N^\tau \gamma^{- 1} \|  \Delta_{1 2} ({\bf R}^{(1)})_j^{j'}(\ell)  \|\,.
\end{equation}
Moreover, arguing as in \eqref{bernardino 0}, \eqref{norma A - omega 12 inverso} (replacing $\omega_1$ resp. $\omega_2$ by $u_1$ resp. $u_2$), one has 
\begin{equation}\label{australia 1}
\|\Delta_{1 2} {\bf A}^{-}(\ell, j, j')^{- 1} \|_{{\rm Op}(j, j')} \leq  \frac{N^{2 \tau}}{\gamma^2 \langle j - j' \rangle^2} \| \Delta_{1 2} {\bf A}^{-}(\ell, j, j')\|_{{\rm Op}(j, j')}\,.
\end{equation}
By the definition \eqref{bf A nu - (ell,alpha,beta)}, we get 
\begin{align}
& \Delta_{12} {\bf A}^{-}(\ell, j, j')  = M_L(\Delta_{12} {\bf D}_j) -M_R(\Delta_{1 2}{\bf D}_{j'}) \nonumber\\
& \stackrel{\eqref{stima bf widehat D alpha (omega)}}{=} (\Delta_{12}m) (j - j')\,{\bf I}_{j,  j'} + M_L(\Delta_{1 2} \widehat{\bf D}_j) - M_R(\Delta_{1 2} \widehat{\bf D}_{j'})\,, 
\end{align}
therefore by \eqref{stime m}, \eqref{norma operatoriale ML MR}, \eqref{Delta12 rj}
\begin{align}
\| \Delta_{12} {\bf A}^{-}(\ell, j, j') \|_{{\rm Op}(j, j')} & \lessdot  \e \langle j - j' \rangle \|u_1 - u_2 \|_{s_0 + \overline \sigma + \mathtt b}\,. \label{australia 2}
\end{align}
Then \eqref{australia 1}, \eqref{australia 2}, $\e \gamma^{- 1} \leq 1$ imply that 
\begin{align*}
& \| \big\{ \Delta_{12} {\bf A}^{-}(\ell, j, j')^{- 1}  \big\} ({\bf R}^{(1)})_j^{j'}(\ell; u_1)  \|   \lessdot N^{2 \tau} \gamma^{- 1} \| ({\bf R}^{(1)})_j^{j'}(\ell; u_1)\| \| u_1 - u_2 \|_{s_0 + \overline\sigma + \mathtt b}\, 
\end{align*}
and recalling \eqref{australia - 1}, \eqref{australia 0}, we obtain the estimate
$$
\| \Delta_{12} ({\bf \Psi}^{(1)})_j^{j'}(\ell)\| \lessdot  N^{2 \tau} \gamma^{- 1} \Big( \| ({\bf R}^{(1)})_j^{j'}(\ell; u_1) \| \| u_1 - u_2 \|_{s_0 + \overline\sigma + \mathtt b} + \| \Delta_{1 2} ({\bf R}^{(1)})_j^{j'}(\ell) \| \Big)\,.
$$
This last estimate imply the estimate \eqref{Psi(u1) - Psi(u2)} for $\Delta_{12}\Psi^{(1)}$, by using the definiton of the norm $| \cdot |_s$ in \eqref{decadimento Kirchoff}. The estimate for $\Delta_{12} \Psi^{(2)}$ follows by similar arguments and then the proof is concluded. 
 \end{proof}
By \eqref{quasi coniugio riducibilita}, \eqref{Homological equation}, \eqref{media di cal R}, we get 
\begin{equation}\label{cal L+}
{\mathcal L}_+ := \Phi^{- 1} {\mathcal L} \Phi = \omega \cdot \partial_\vphi {\mathbb I}_2 + {\mathcal D}_+ + {\mathcal R}_+\,,
\end{equation}
where 
\begin{align} 
& {\mathcal D}_+ := {\mathcal D} + [{\mathcal R}]\,, \nonumber\\
& {\mathcal R}_+ :=  (\Phi^{- 1} - {\mathbb I}_2) {\mathcal R} + \Phi^{- 1} \Big( \Pi_N^\bot {\mathcal R} + \omega \cdot \partial_\vphi \Phi_{\geq 2} + [{\mathcal D}, \Phi_{\geq 2}] + {\mathcal R}(\Phi - {\mathbb I}_2) \Big)\,. \nonumber
\end{align}

\begin{lemma}[{\bf The new $2 \times 2$ block-diagonal part}]\label{the new diagonal part}
The new block-diagonal part is 
$$
{\mathcal D}_+ := {\mathcal D} + [{\mathcal R}] = \ii \begin{pmatrix}
 {\mathcal D}_+^{(1)} & 0 \\
0 & -  \overline{\mathcal D}_+^{(1)}
\end{pmatrix}\,,\qquad {\mathcal D}_+^{(1)} := {\mathcal D}^{(1)} + [{\mathcal R}^{(1)}] = {\rm diag}_{j \in \N } {\bf D}_j^+\,,
$$
where
$$
 {\bf D}_j^+ := {\bf D}_j + ({\bf R}^{(1)})_j^j(0) = m\, j {\bf I}_j  + \widehat{\bf D}_j + ({\bf R}^{(1)})_j^j(0) = m\,j {\bf I}_j + \widehat{\bf D}_j^+\,,\quad 
$$
\begin{equation}\label{espansione blocco D+ alpha}
\widehat{\bf D}_j^+ := \widehat{\bf D}_j + ({\bf R}^{(1)})_j^j(0)\,, \qquad \forall j \in \N\,,
\end{equation}
and 
\begin{equation}\label{mu + - mu}
 \|{\bf D}_j^+ - {\bf D}_j\|^\Lipg \lessdot  j^{- 1} |{\mathcal R} |D||_{s_0}^\Lipg\,.
\end{equation}
Moreover, if ${\bf u}_i(\omega) = (u_i(\omega), \psi_i(\omega))$, $i = 1, 2$ are families of Sobolev functions, for all $\omega \in \Omega_\nu^{\gamma_1}({\bf u}_1) \cap \Omega_\nu^{\gamma_2}({\bf u}_2)$, $\gamma /2 \leq \gamma_1, \gamma_2 \leq 2 \gamma$, for all $j \in \N$,  
\begin{equation}\label{Delta 12 mu + - mu}
\| \big(\widehat{\bf D}_j^+(u_1) - \widehat{\bf D}_j^+(u_2) \big) - \big( \widehat{\bf D}_j(u_1) - \widehat{\bf D}_j(u_2) \big) \|_{L^2} \lessdot |{\mathcal R}(u_1) - {\mathcal R}(u_2)|_{s_0}\,.
\end{equation}

\end{lemma}

\begin{proof}
Notice that, since ${\mathcal R}^{(1)}(\vphi)$ is selfadjoint, the operators $({\bf R}^{(1)})_j^j(0) : {\bf E}_j \to {\bf E}_j$ are self-adjoint for all $j \in \N$, i.e. $({\bf R}^{(1)})_j^j(0) \in {\mathcal S}({\bf E}_j)$. Since ${\bf D}_j, \widehat{\bf D}_j$ are self-adjoint, we get that ${\bf D}_j^+, \widehat{\bf D}_j^+$ are self-adjoint for all $j \in \N$. Furthermore, by Lemma \ref{elementarissimo decay}
\begin{align*}
 \|{\bf D}_j^+ - {\bf D}_j\|^\Lipg & = \|\widehat{\bf D}_j^+ - \widehat{\bf D}_j\|^\Lipg = \| ({\bf R}^{(1)})_j^j(0)\|^\Lipg \\
 & \leq  j^{- 1} {\rm sup }_{k \in \N}\|  ({\bf R}^{(1)})_k^k(0) k \|^\Lipg  \leq j^{- 1} |{\mathcal R} |D||_{s_0}^\Lipg\,,
\end{align*}
which is the estimate \eqref{mu + - mu}. 

\noindent
Since, by \eqref{espansione blocco D+ alpha} we have $\big(\widehat{\bf D}_j^+(u_1) - \widehat{\bf D}_j^+(u_2) \big) - \big( \widehat{\bf D}_j(u_1) - \widehat{\bf D}_j(u_2) \big) = ({\bf R}^{(1)})_j^j(0; u_1) - ({\bf R}^{(1)})_j^j(0 ; u_2)\,,$
the estimate \eqref{Delta 12 mu + - mu} follows since, applying again Lemma \ref{elementarissimo decay}, for all $j \in \N$  
$$
\| ({\bf R}^{(1)})_j^j(0; u_1) - ({\bf R}^{(1)})_j^j(0 ; u_2) \| \lessdot |{\mathcal R}^{(1)}(u_1) - {\mathcal R}^{(1)}(u_2)|_{s_0}\,.
$$
\end{proof}
\subsection{The iteration}
Let $\nu \geq 0$ and let us suppose that $({\bf Si})_\nu$ are true. We prove $({\bf Si})_{\nu + 1}$. To simplify notations, in this proof we write $| \cdot |_s$ for $| \cdot |_s^\Lipg$.

{\sc Proof of $({\bf S1})_{\nu + 1}$}. Since the self-adjoint $2 \times 2$ blocks ${\bf D}_j^\nu \in {\mathcal S}({\bf E}_j)$ are defined on $\Omega_\nu^\gamma$, the set $\Omega_{\nu + 1}^\gamma$ is well-defined and by Lemma \ref{homologica equation}, the following estimates hold on $\Omega_{\nu + 1}^\gamma$
\begin{equation}\label{Psi nu norma alta}
| \Psi_\nu | D | |_s \leq_s  N_\nu^{2 \tau + 1} \gamma^{- 1} |{\mathcal R}_\nu | D ||_s 
\stackrel{\eqref{Rsb}}{\leq_s} N_\nu^{2 \tau + 1} N_{\nu - 1}^{- \mathtt a} \gamma^{- 1} | {\mathcal R}_0 |D | |_{s + {\mathtt b}}  \,,
\end{equation}
and in particular, by \eqref{stima R0 s0 + mathtt b}, \eqref{piccolezza1}, \eqref{definizione alpha}, \eqref{defN}, taking $\delta_0$ small enough, 
\begin{equation}\label{Psi nu norma bassa}
| \Psi_\nu |D | |_{s_0} \leq  1\,.
\end{equation}
By \eqref{Psi nu norma bassa}, we can apply Lemma \ref{lem:inverti} to the map $\Phi_\nu^{\pm 1} := {\rm exp}(\pm \Psi_\nu)$ and using also Lemma \ref{elementarissimo decay}-$(ii)$ we obtain that
\begin{equation}\label{Phi nu bassa alta}
 | \Phi_\nu^{\pm 1} - {\mathbb I}_2  |_s \leq | (\Phi_\nu^{\pm 1} - {\mathbb I}_2) |D|  |_s \leq_s | \Psi_\nu |D| |_s\,. 
\end{equation}

By \eqref{cal L+} we get  ${\mathcal L}_{\nu + 1} := \Phi_\nu^{- 1} {\mathcal L}_\nu \Phi_\nu = \omega \cdot \partial_\vphi {\mathbb I}_2 + {\mathcal D}_{\nu + 1} + {\mathcal R}_{\nu + 1}$, where ${\mathcal D}_{\nu + 1} := {\mathcal D}_\nu + [{\mathcal R}_\nu]$ and 
\begin{align}
 {\mathcal R}_{\nu + 1} & := (\Phi_\nu^{- 1} - {\mathbb I}_2) [{\mathcal R}_\nu] \nonumber\\
 &+  \Phi_\nu^{- 1} \Big( \Pi_{N_\nu}^\bot {\mathcal R}_\nu + \omega \cdot \partial_\vphi \Psi_{\nu , \geq 2} + [{\mathcal D}_\nu, \Psi_{\nu , \geq 2}] +  {\mathcal R}_\nu (\Phi_\nu - {\mathbb I}_2)  \Big)\,. \label{cal L nu + 1}
\end{align}
Note that, since ${\mathcal R}_\nu$ is defined on $\Omega_\nu^\gamma$ and $\Psi_\nu$ is defined on $\Omega_{\nu + 1}^\gamma$, the remainder ${\mathcal R}_{\nu + 1}$ is defined on $\Omega_{\nu + 1}^\gamma$ too. 
Since the remainder ${\mathcal R}_\nu : {\bf H}^s_0(\T^{\nu + 1}) \to {\bf H}^s_0(\T^{\nu + 1})$ is Hamiltonian, the map $\Psi_\nu : {\bf H}^s_0(\T^{\nu + 1}) \to {\bf H}^s_0(\T^{\nu + 1})$ is Hamiltonian, then $\Phi_\nu : {\bf H}^s_0(\T^{\nu + 1}) \to {\bf H}^s_0(\T^{\nu + 1})$ is symplectic and the operator ${\mathcal L}_{\nu + 1} : {\bf H}^s_0(\T^{\nu + 1}) \to {\bf H}^{s - 1}_0(\T^{\nu + 1})$ is still Hamiltonian.  

Now let us prove the estimates \eqref{Rsb} for ${\mathcal R}_{\nu + 1}$.  Applying Lemmata \ref{elementarissimo decay}, \ref{interpolazione decadimento Kirchoff} and the estimates  \eqref{Psi nu norma bassa}, \eqref{Psi nu norma alta}, \eqref{Phi nu bassa alta}, we get 
\begin{align}
& |(\Phi_\nu^{- 1} - {\mathbb I}_2) [{\mathcal R}_\nu] |D||_s\,,\, |\Phi_\nu^{- 1}{\mathcal R}_\nu (\Phi_\nu - {\mathbb I}_2) |D| |_s \nonumber \\
&    \leq_s N_\nu^{2 \tau + 1}\gamma^{- 1}  |{\mathcal R}_\nu |D| |_s |{\mathcal R}_\nu |D||_{s_0}\, \label{primo pezzo cal R nu + 1}
\end{align}
and 
\begin{equation}\label{secondo pezzo cal R nu + 1}
|\Phi_{\nu}^{- 1} \Pi_{N_\nu}^\bot {\mathcal R}_\nu |D||_s \leq_s |\Pi_{N_\nu} {\mathcal R}_\nu |D||_s + N_\nu^{2 \tau + 1}\gamma^{- 1}  |{\mathcal R}_\nu |D| |_s |{\mathcal R}_\nu |D||_{s_0}\,.
\end{equation}
Then, it remains to estimate the term $\Phi_\nu^{- 1} \big( \omega \cdot \partial_\vphi \Phi_{\nu , \geq 2} + [{\mathcal D}_\nu, \Phi_{\nu , \geq 2}] \big) |D|$ in \eqref{cal L nu + 1}. A direct calculation shows that for all $n \geq 2$
\begin{align}
\omega \cdot \partial_\vphi (\Psi_\nu^n) + [{\mathcal D}_\nu, \Psi_\nu^n] & = \sum_{i + k = n - 1} \Psi_\nu^i (\omega \cdot \partial_\vphi \Psi_\nu + [{\mathcal D}_\nu, \Psi_\nu]) \Psi_\nu^k \nonumber\\
&  \stackrel{\eqref{Homological equation}}{=} \sum_{i + k = n - 1} \Psi_\nu^i ([{\mathcal R}_\nu] - {\mathcal R}_\nu) \Psi_\nu^k\,, \label{vesuvio}
\end{align}
therefore using \eqref{Psi nu norma bassa}, \eqref{Psi nu norma alta},  \eqref{Mnab} and that $|\Psi_\nu|_s \stackrel{Lemma\,\ref{elementarissimo decay}-(ii)}{\leq} |\Psi_\nu|D||_s$ we get that for all $n \geq 2$ 
\begin{align}
& \Big| \Big(\omega \cdot \partial_\vphi (\Psi_\nu^n) + [{\mathcal D}_\nu, \Psi_\nu^n] \Big) |D | \Big|_s \nonumber\\
&  \leq  n^2 C(s)^n \Big(( |\Psi_\nu |D||_{s_0})^{n - 1} |{\mathcal R}_\nu |D||_s + (|\Psi_\nu |D||_{s_0})^{n - 2} |\Psi_\nu |D||_s |{\mathcal R}_\nu |D||_{s_0}  \Big)  \nonumber\\
& \leq 2 n^2 C(s)^n N_\nu^{2 \tau + 1} \gamma^{- 1} |{\mathcal R}_\nu |D||_s |{\mathcal R}_\nu |D||_{s_0}\,, \label{peppino 0}
\end{align}
for some constant $C(s) > 0$. Thus 
\begin{align}
& \Big| \Big(\omega \cdot \partial_\vphi \Psi_{\nu , \geq 2} + [{\mathcal D}_\nu, \Psi_{\nu , \geq 2}] \Big)  |D| \Big|_s  \leq \sum_{n \geq 2} \frac{1}{n !} \Big| \Big(\omega \cdot \partial_\vphi (\Psi_\nu^n) + [{\mathcal D}_\nu, \Psi_\nu^n] \Big) |D |  \Big|_s \nonumber\\
& \stackrel{\eqref{peppino 0}}{\leq} N_\nu^{2 \tau + 1} \gamma^{- 1} |{\mathcal R}_\nu |D||_s |{\mathcal R}_\nu |D||_{s_0} 2 \sum_{n \geq 2}\frac{C(s)^n n^2}{n !} \nonumber\\
& \leq_s N_\nu^{2 \tau + 1} \gamma^{- 1} |{\mathcal R}_\nu |D||_s |{\mathcal R}_\nu |D||_{s_0}\,. \label{peppino 1}
\end{align}
The estimates \eqref{Phi nu bassa alta}, \eqref{peppino 1} and Lemma \ref{interpolazione decadimento Kirchoff} imply that  
\begin{align}
& \Big|\Phi_\nu^{- 1} \big( \omega \cdot \partial_\vphi \Psi_{\nu , \geq 2} + [{\mathcal D}_\nu, \Psi_{\nu , \geq 2}] \big) |D|  \Big|_s \nonumber\\
& \leq_s N_\nu^{2 \tau + 1} \gamma^{- 1} |{\mathcal R}_\nu |D||_{s} |{\mathcal R}_\nu |D||_{s_0}\,. \label{quarto pezzo cal R nu + 1}
\end{align}
Collecting the estimates \eqref{primo pezzo cal R nu + 1}-\eqref{quarto pezzo cal R nu + 1} we obtain the estimate
\begin{equation}\label{stima induttiva cal R nu 1}
| {\mathcal R}_{\nu + 1} |D| |_s \leq_s | \Pi_{N_\nu} {\mathcal R}_\nu |D | |_s + N_\nu^{2 \tau + 1} \gamma^{- 1} | {\mathcal R}_\nu |D| |_s | {\mathcal R}_\nu |D| |_{s_0}\,,
\end{equation}
which implies (using the smoothing property \eqref{smoothingN} and \eqref{piccolezza1}, \eqref{Rsb}  )
\begin{equation}\label{stima induttiva cal R nu 2 bassa}
| {\mathcal R}_{\nu + 1} |D | |_s \leq_s N_\nu^{- \mathtt b } | {\mathcal R}_\nu |D | |_{s + \mathtt b} + N_\nu^{2 \tau + 1} \gamma^{- 1} | {\mathcal R}_\nu |D | |_s | {\mathcal R}_\nu |D | |_{s_0}\,,
\end{equation}
\begin{equation}\label{stima induttiva cal R nu 2 alta}
| {\mathcal R}_{\nu + 1} |D | |_{s + \mathtt b} \leq C(s + \mathtt b) | {\mathcal R}_\nu |D | |_{s + \mathtt b}\,.
\end{equation}
Hence
$$
| {\mathcal R}_{\nu + 1} | D | |_{s + \mathtt b} \ \stackrel{\eqref{stima induttiva cal R nu 2 alta},\eqref{Rsb}}{\leq} C(s + \mathtt b) | {\mathcal R}_0 |D| |_{s + \mathtt b} N_{\nu - 1} \leq |{\mathcal R}_0 |D| |_{s + \mathtt b} N_\nu\,,
$$
for $N_0 := N_0(s, \mathtt b) > 0$ large enough and then the second inequality in \eqref{Rsb} for ${\mathcal R}_{\nu + 1}$ has been proved. Let us prove the first inequality in \eqref{Rsb} at the step $\nu + 1$. We have 
\begin{align*}
| {\mathcal R}_{\nu + 1}|D | |_s & \stackrel{\eqref{stima induttiva cal R nu 2 alta}, \eqref{Rsb}}{\leq_s} N_\nu^{- \mathtt b} N_{\nu - 1} | {\mathcal R}_0 |D | |_{s + \mathtt b} + N_\nu^{2 \tau + 1}  N_{\nu - 1}^{- 2 \mathtt a} \gamma^{- 1} | {\mathcal R}_0|D| |_{s_0 + \mathtt b} | {\mathcal R}_0 |D | |_{s + \mathtt b} \\
& \leq | {\mathcal R}_0 |D | |_{s + \mathtt b} N_\nu^{- \mathtt a}\,,
\end{align*}
provided 
$$
N_{\nu}^{\mathtt b - \mathtt a} N_{\nu - 1}^{- 1} \geq 2 C(s)\,,\quad \gamma^{- 1} | {\mathcal R}_0 |D | |_{s_0 + \mathtt b} \leq \frac{ N_{\nu - 1}^{2 \mathtt a} N_\nu^{- {\mathtt a} - 2 \tau - 1}}{2 C(s)}\,,
$$
which are verified by \eqref{defN}, \eqref{definizione alpha}, \eqref{stima R0 s0 + mathtt b} and \eqref{piccolezza1}, taking $N_0$ large enough and $\delta_0$ small enough. 

\noindent
The estimate \eqref{rjnu bounded} for $\widehat{\bf D}_j^{\nu +1}$ follows by a telescopic argument, since $\widehat{\bf D}_j^{\nu + 1} = \sum_{k = 0}^\nu  \widehat{\bf D}_j^{k + 1} - \widehat{\bf D}_j^k$ and applying the estimates \eqref{mu + - mu}, \eqref{Rsb}, \eqref{stima R0 s0 + mathtt b}.

\noindent
{\sc Proof of $({\bf S2})_{\nu + 1}$} We now construct a Lipschitz extension of the function $\omega \in \Omega_{\nu + 1}^\gamma \mapsto {\bf D}_j^{\nu + 1}(\omega ) \in {\mathcal S}({\bf E}_j)$, for all $j \in \N$. We apply Lemma M.5 in \cite{KP}. Note that the space ${\mathcal S}({\bf E}_j)$, defined in \eqref{cal S E alpha}, is a Hilbert subspace of ${\mathcal L}({\bf E}_j)$ equipped by the scalar product defined in \eqref{prodotto scalare traccia matrici}. By the inductive hyphothesis, there exists a Lipschitz function $\widetilde{\bf D}^\nu_j : \Omega_o \to {\mathcal S}({\bf E}_j)$, satisfying $\widetilde{\bf D}_j^\nu(\omega) = {\bf D}_j^\nu(\omega)$, for all $\omega \in \Omega_\nu^\gamma$. Now we construct a self-adjoint extension of the self-adjoint operator ${\bf D}_j^{\nu + 1} = {\bf D}_j^\nu + [{\bf D}]_j^\nu $, where $ [{\bf D}]_j^\nu := ({\bf R}_\nu^{(1)})_j^j(0)$. 
By \eqref{mu + - mu}, for all $j \in \N$, one has that 
\begin{align*}
 \| [{\bf D}]_j^\nu \|^\Lipg & = \|{\bf D}_j^{\nu + 1} - {\bf D}_j^\nu  \|^\Lipg \lessdot j^{- 1} |{\mathcal R}_\nu |D| |_{s_0}^\Lipg  \stackrel{\eqref{Rsb}}{\lessdot}  N_{\nu - 1}^{- \mathtt a}   |{\mathcal R}_0 |D||_{s_0 + \mathtt b} j^{- 1}\\
 & \stackrel{\eqref{stima R0 s0 + mathtt b}}{\lessdot} N_{\nu - 1}^{- \mathtt a} \e j^{- 1}
\end{align*}
and then by Lemma M.5 in \cite{KP} there exists a Lipschitz extension $[\widetilde{\bf D}]_j^\nu : \Omega_o \to {\mathcal S}({\bf E}_j)$  of $[{\bf D}]_j^\nu : \Omega_\nu^\gamma \to {\mathcal S}({\bf E}_j)$ still satisfying the above estimate. Therefore we define $\widetilde{\bf D}_j^{\nu + 1} := \widetilde{\bf D}_j^\nu + [\widetilde{\bf D}]_j^\nu$. 

\noindent
{\sc Proof of ${\bf (S3)}_{\nu + 1}$}. As a notation we write $\Delta_{1 2} A := A(u_1) - A(u_2)$, for any operator $A = A(u)$ depending on $u$. 
Now we will estimate the operator $\Delta_{12} {\mathcal R}_{\nu + 1}$, where ${\mathcal R}_{\nu + 1}$ is defined in \eqref{cal L nu + 1}. Moreover, we define 
$$
R_\nu(s) := {\rm max}\{ |{\mathcal R}_\nu(u_1)|_{s}, |{\mathcal R}_\nu(u_2)|_{s} \}\,, \qquad \forall s \in [s_0, s_0 + \mathtt b]\,. 
$$
Note that, by \eqref{Rsb} and \eqref{stima R0 s0 + mathtt b} and by Lemma \ref{elementarissimo decay}-$(ii)$, one gets 
\begin{equation}\label{R(s0) R(s0 + mathtt b)}
R_\nu(s_0) \lessdot \e N_{\nu - 1}^{- \mathtt a}\,, \qquad R_\nu(s_0 + \mathtt b) \lessdot \e N_{\nu - 1}\,.
\end{equation}
By \eqref{PsiR}, \eqref{Psi(u1) - Psi(u2)}, \eqref{R(s0) R(s0 + mathtt b)}, \eqref{derivate-R-nu}, \eqref{derivate-R-nu1}, Lemma\,\ref{elementarissimo decay}-$(ii)$, one has 
\begin{align}
|\Delta_{12} \Psi_\nu|_{s_0} & {\lessdot}  N_\nu^{2 \tau + 1} N_{\nu - 1}^{- \mathtt a} \e \gamma^{- 1} \| u_1 - u_2 \|_{s_0 + \overline \sigma + \mathtt b}\,,\label{delta 12 Psi nu s0} \\
|\Delta_{12} \Psi_\nu|_{s_0 + \mathtt b} &\lessdot N_\nu^{2 \tau + 1} N_{\nu - 1} \e \gamma^{- 1} \| u_1 - u_2 \|_{s_0 + \overline\sigma + \mathtt b}\,,\label{delta 12 Psi nu s0 + beta} \\
|\Psi_\nu(u_1)|_{s_0}\,, |\Psi_\nu(u_2)|_{s_0} &  \lessdot N_\nu^{2 \tau + 1} N_{\nu - 1}^{- \mathtt a} \e \gamma^{- 1}\,, \label{Psi nu u1 u2 s0} \\
|\Psi_\nu(u_1)|_{s_0 + \mathtt b}\,, |\Psi_\nu(u_2)|_{s_0 + \mathtt b} &  \lessdot N_\nu^{2 \tau + 1} N_{\nu - 1} \e \gamma^{- 1}\,. \label{Psi nu u1 u2 s0 + mathtt b}
\end{align}
By the estimates \eqref{Phi nu bassa alta} (applied to $\Phi_\nu = \Phi_\nu(u_i)$, $i = 1, 2$ ), \eqref{Psi nu u1 u2 s0}-\eqref{Psi nu u1 u2 s0 + mathtt b} and using also \eqref{derivata-inversa-Phi}, one gets
\begin{align}
& |\Phi_\nu^{\pm 1}(u_i) - {\mathbb I}_2|_{s_0} \lessdot N_\nu^{2 \tau + 1} N_{\nu - 1}^{- \mathtt a} \e \gamma^{- 1}\,, \label{Phi (u1 u2) - I2} \\
&  |\Phi_\nu^{\pm 1}(u_i) - {\mathbb I}_2|_{s_0 + \mathtt b} \lessdot N_\nu^{2 \tau + 1} N_{\nu - 1} \e \gamma^{- 1}\,, \label{Phi (u1 u2) - I21} \\
& |\Delta_{12} \Phi_\nu^{\pm 1} |_{s_0}  \lessdot  N_\nu^{2 \tau + 1} N_{\nu - 1}^{- \mathtt a} \e \gamma^{- 1} \| u_1 - u_2 \|_{s_0 + \overline \sigma + \mathtt b}\,,  \\
& |\Delta_{12} \Phi_\nu^{\pm 1} |_{s_0 + \mathtt b}  \lessdot   N_\nu^{2 \tau + 1} N_{\nu - 1} \e \gamma^{- 1}  \| u_1 - u_2\|_{s_0 + \overline \sigma + \mathtt b}\,. \label{Delta 12 Phi nu norma bassa}
\end{align}
We estimate separately the terms of $\Delta_{12} {\mathcal R}_{\nu + 1}$, where, by \eqref{cal L nu + 1} 
\begin{equation}\label{R nu + 1 differenze}
{\mathcal R}_{\nu + 1} = (\Phi_\nu^{- 1} - {\mathbb I}_2) [{\mathcal R}_\nu] +  \Phi_\nu^{- 1} {\mathcal H}_\nu\,, \qquad   
\end{equation}
$$
{\mathcal H}_\nu := \Pi_{N_\nu}^\bot {\mathcal R}_\nu + \omega \cdot \partial_\vphi \Psi_{\nu , \geq 2} + [{\mathcal D}_\nu, \Psi_{\nu , \geq 2}] +  {\mathcal R}_\nu (\Phi_\nu - {\mathbb I}_2)\,.
$$
In the following we will use that by \eqref{defN}, \eqref{definizione alpha}, \eqref{piccolezza1} (choosing $\delta_0$ small enough) 
\begin{equation}\label{condizione u1 - u2 stima delta 12 Phi}
N_\nu^{2 \tau + 1} N_{\nu - 1}^{- \mathtt a} \e \gamma^{- 1} \leq 1\,, \quad \forall \nu \geq 0\,.
\end{equation}
 Lemma \ref{elementarissimo decay}-$(iii)$, Lemma \ref{interpolazione decadimento Kirchoff} and the estimates \eqref{Phi (u1 u2) - I2}-\eqref{Delta 12 Phi nu norma bassa}, \eqref{R(s0) R(s0 + mathtt b)}, \eqref{derivate-R-nu}, \eqref{derivate-R-nu1}, \eqref{condizione u1 - u2 stima delta 12 Phi} imply that 
 \begin{align}\label{stima secondo pezzo delta 12 R nu + 1 s0}
& |\Delta_{12} \big\{ (\Phi_\nu^{- 1} - {\mathbb I}_2) [{\mathcal R}_\nu] \big\}|_{s_0}\,,\, |\Delta_{12}\big\{ {\mathcal R}_\nu (\Phi_\nu - {\mathbb I}_2) \big\} |_{s_0}  \\
& \lessdot N_\nu^{2 \tau + 1} N_{\nu - 1}^{- 2 \mathtt a} \e^2 \gamma^{- 1} \|u_1 - u_2 \|_{s_0 +\overline \sigma + \mathtt b}\,, \nonumber\\
& \quad |\Delta_{12} \big\{ (\Phi_\nu^{- 1} - {\mathbb I}_2) [{\mathcal R}_\nu] \big\}|_{s_0 + \mathtt b}\,,\,  |\Delta_{12}\big\{ {\mathcal R}_\nu (\Phi_\nu - {\mathbb I}_2) \big\} |_{s_0 + \mathtt b} \label{stima secondo pezzo delta 12 R nu + 1 s0 + beta}\\
& \lessdot N_{\nu - 1} \e \| u_1 - u_2 \|_{s_0 + \overline \sigma + \mathtt b}\,,\nonumber
\end{align}
 Moreover by \eqref{smoothingN}, \eqref{derivate-R-nu}, \eqref{derivate-R-nu1} one gets 
 \begin{align}
 |\Pi_{N_\nu}^\bot\Delta_{12} {\mathcal R}_\nu|_{s_0} & \lessdot N_{\nu}^{- \mathtt b} N_{\nu - 1} \e \|  u_1 - u_2\|_{s_0 + \overline \sigma + \mathtt b}\,, \label{Pi N Delta 12 R nu bassa} \\
 |\Pi_{N_\nu}^\bot\Delta_{12} {\mathcal R}_\nu|_{s_0 + \mathtt b} & \lessdot  N_{\nu - 1} \e \|  u_1 - u_2\|_{s_0 + \overline \sigma + \mathtt b}\,, \label{Pi N Delta 12 R nu alta} 
 \end{align}
It remains to estimate only the term $\Delta_{12} \big\{ \omega \cdot \partial_\vphi \Phi_{\nu, \geq 2} + [{\mathcal D}_\nu, \Phi_{\nu, \geq 2}] \big\}$, where we recall that $\Phi_{\nu, \geq 2} = \sum_{n \geq 2} \frac{\Psi_\nu^n}{n !}$.
By \eqref{vesuvio}, for all $n \geq 2$ we have 
\begin{equation}\label{alfredo}
\Delta_{12} \big\{ \omega \cdot \partial_\vphi \Psi_\nu^n + [{\mathcal D}_\nu, \Psi_\nu^n] \big\} = \sum_{i + k = n - 1} \Delta_{12}\big\{  \Psi_\nu^i ([{\mathcal R}_\nu] - {\mathcal R}_\nu) \Psi_\nu^k \big\}\,.
\end{equation}
Iterating the interpolation estimate of Lemma \ref{interpolazione decadimento Kirchoff} and using \eqref{derivate-R-nu}, \eqref{derivate-R-nu1}, \eqref{R(s0) R(s0 + mathtt b)}, \eqref{delta 12 Psi nu s0}-\eqref{Psi nu u1 u2 s0 + mathtt b}, we have that for all $i + k = n - 1$, 
\begin{align}
& |\Delta_{12}\big\{  \Psi_\nu^i ([{\mathcal R}_\nu] - {\mathcal R}_\nu) \Psi_\nu^k \big\} |_{s_0} \nonumber\\
&  \leq n C(s_0)^n \Big( N_\nu^{2 \tau + 1} N_{\nu - 1}^{- \mathtt a} \e \gamma^{- 1}\Big)^{n - 2} N_\nu^{2 \tau + 1} N_{\nu - 1}^{- 2 \mathtt a}\e^2 \gamma^{- 1}  \| u_1 - u_2 \|_{s_0 + \overline \sigma + \mathtt b} \nonumber\\
& \stackrel{\eqref{condizione u1 - u2 stima delta 12 Phi}}{\leq} n C(s_0)^n N_\nu^{2 \tau + 1} N_{\nu - 1}^{- 2 \mathtt a} \e^2 \gamma^{- 1} \| u_1 - u_2 \|_{s_0 + \overline \sigma + \mathtt b} \label{alfredo n s0}
\end{align}
and 
\begin{align}
& |\Delta_{12}\big\{  \Psi_\nu^i ([{\mathcal R}_\nu] - {\mathcal R}_\nu) \Psi_\nu^k \big\} |_{s_0 + \mathtt b}  \nonumber\\ 
& \leq n C(s_0 + \mathtt b)^n \Big( N_\nu^{2 \tau + 1} N_{\nu - 1}^{- \mathtt a} \e \gamma^{- 1}\Big)^{n - 1} N_{\nu - 1} \e \| u_1 - u_2 \|_{s_0 + \overline \sigma + \mathtt b} \nonumber\\
& \stackrel{\eqref{condizione u1 - u2 stima delta 12 Phi}}{\leq} n C(s_0 + \mathtt b)^{n} N_{\nu - 1} \e \|u_1 - u_2 \|_{s_0 + \overline \sigma + \mathtt b}\,. \label{alfredo n s0 + beta}
\end{align}
Hence 
\begin{align}
& \Big| \omega \cdot \partial_\vphi \Phi_{\nu, \geq 2} + [{\mathcal D}_\nu, \Phi_{\nu, \geq 2}] \Big|_{s_0}  
  \stackrel{\eqref{alfredo}}{\leq} \sum_{n \geq 2} \frac{1}{n!} \sum_{i + k = n - 1} \big|\Delta_{12}\big\{  \Psi_\nu^i ([{\mathcal R}_\nu] - {\mathcal R}_\nu) \Psi_\nu^k \big\} \big|_{s_0} \nonumber\\
& \stackrel{\eqref{alfredo n s0}}{\leq} N_\nu^{2 \tau + 1} N_{\nu - 1}^{- 2 \mathtt a} \e^2 \gamma^{- 1} \| u_1 - u_2 \|_{s_0 + \overline \sigma + \mathtt b} \sum_{n \geq 2} \frac{n^2 C(s_0)^n }{n  !} \nonumber\\
& \lessdot N_\nu^{2 \tau + 1} N_{\nu - 1}^{- 2 \mathtt a} \e^2 \gamma^{- 1} \| u_1 - u_2 \|_{s_0 + \overline \sigma + \mathtt b} \label{stima quarto pezzo delta 12 R nu + 1 s0}
\end{align}
and
\begin{align}
& \Big| \omega \cdot \partial_\vphi \Phi_{\nu, \geq 2} + [{\mathcal D}_\nu, \Phi_{\nu, \geq 2}] \Big|_{s_0 + \mathtt b}  \stackrel{\eqref{alfredo}}{\leq} \sum_{n \geq 2} \frac{1}{n!} \sum_{i + k = n - 1} \big|\Delta_{12}\big\{  \Psi_\nu^i ([{\mathcal R}_\nu] - {\mathcal R}_\nu) \Psi_\nu^k \big\} \big|_{s_0 + \mathtt b} \nonumber\\
& \stackrel{\eqref{alfredo n s0 + beta}}{\leq} N_{\nu - 1} \e \| u_1 - u_2 \|_{s_0 + \overline \sigma + \mathtt b} \sum_{n \geq 2} \frac{n^2 C(s_0 + \mathtt b)^n }{n !} \nonumber\\
& \lessdot N_{\nu - 1} \e \| u_1 - u_2 \|_{s_0 + \overline \sigma + \mathtt b}\,. \label{stima quarto pezzo delta 12 R nu + 1 s0 + beta}
\end{align}
Collecting the estimates \eqref{Pi N Delta 12 R nu bassa}, \eqref{Pi N Delta 12 R nu alta}, \eqref{stima quarto pezzo delta 12 R nu + 1 s0}, \eqref{stima quarto pezzo delta 12 R nu + 1 s0 + beta} and recalling the definition of ${\mathcal H}_\nu$ in \eqref{R nu + 1 differenze}, one gets
\begin{align}
|\Delta_{12} {\mathcal H}_\nu|_{s_0} & \lessdot \Big(N_\nu^{- \mathtt b} N_{\nu - 1} \e + N_\nu^{2 \tau + 1} N_{\nu - 1}^{- 2 \mathtt a} \e^2 \gamma^{- 1} \Big) \| u_1 - u_2\|_{s_0 + \overline \sigma + \mathtt b} \label{stima cal H nu + 1 s0} \\
|\Delta_{12} {\mathcal H}_\nu|_{s_0 + \mathtt b} & \lessdot N_{\nu - 1} \e \| u_1 - u_2\|_{s_0 + \overline \sigma + \mathtt b}\,. \label{stima cal H nu + 1 s0 + mathtt b}
\end{align}
Arguing as in the proof of \eqref{Rsb} one can obtain that for $i = 1, 2$
\begin{align}
|{\mathcal H}_\nu(u_i)|_{s_0}  & \lessdot N_\nu^{- \mathtt b} N_{\nu - 1} \e + N_\nu^{2 \tau + 1} N_{\nu - 1}^{- 2 \mathtt a} \e^2 \gamma^{- 1}\,, \\
\quad |{\mathcal H}_\nu(u_i)|_{s_0 + \mathtt b} &   \lessdot N_{\nu - 1} \e\,,  \label{stima cal H nu} 
 \end{align}
 thus, by \eqref{R nu + 1 differenze}, writing $\Phi_\nu^{- 1} {\mathcal H}_\nu = {\mathcal H}_\nu + (\Phi_\nu^{- 1} - {\mathbb I}_2) {\mathcal H}_\nu$, using Lemma \ref{interpolazione decadimento Kirchoff}, and the estimates \eqref{Phi (u1 u2) - I2}-\eqref{Delta 12 Phi nu norma bassa}, \eqref{condizione u1 - u2 stima delta 12 Phi}-\eqref{stima secondo pezzo delta 12 R nu + 1 s0 + beta}, \eqref{stima cal H nu + 1 s0}-\eqref{stima cal H nu}, one obtains 
 \begin{align}
 |\Delta_{12} {\mathcal R}_{\nu + 1}|_{s_0} & \leq C(\tau, \nu) \Big(N_\nu^{- \mathtt b} N_{\nu - 1} \e + N_\nu^{2 \tau + 1} N_{\nu - 1}^{- 2 \mathtt a} \e^2 \gamma^{- 1} \Big) \| u_1 - u_2\|_{s_0 + \overline \sigma + \mathtt b} \nonumber\\
  & \leq K_0 N_\nu^{- \mathtt a} \e\,, \nonumber
 \end{align}
and
 \begin{align}
 |\Delta_{12} {\mathcal R}_{\nu + 1}|_{s_0 + \mathtt b} & \leq C(\tau, \nu) N_{\nu - 1} \e  \| u_1 - u_2\|_{s_0 + \overline \sigma + \mathtt b} \leq K_0   N_\nu \e \| u_1 - u_2\|_{s_0 + \overline \sigma + \mathtt b} \nonumber
 \end{align}
 by \eqref{defN}, \eqref{definizione alpha}, taking $N_0$ large enough and $\e \gamma^{- 1}$ small enough. Then the estimate \eqref{derivate-R-nu}, \eqref{derivate-R-nu1} has been proved at the step $\nu + 1$. The estimates \eqref{deltarj12}, \eqref{Delta12 rj} follow by \eqref{Delta 12 mu + - mu}, \eqref{derivate-R-nu} and by a telescopic argument. 

\noindent
{\sc Proof of $({\bf S4})_{\nu + 1}$} We have to prove that, if 
\begin{equation}\label{caltanisetta}
 K_1 N_\nu^\tau \e \| u_1 - u_2 \|_{s_0 + \overline \sigma + \mathtt b} \leq \rho\,,
 \end{equation}
 for a suitable constant $K_1 = K_1(\tau, \nu) > 0$, then 
$$
\omega \in \Omega_{\nu + 1}^\gamma({\bf u}_1) \qquad  \Longrightarrow \qquad   \omega \in \Omega_{\nu + 1}^{\gamma - \rho}({\bf u}_2)\,.
$$
By the definiton \eqref{Omgj}, we have $\Omega_{\nu + 1}^\gamma({\bf u}_1) \subseteq \Omega_\nu^\gamma({\bf u}_1)$, and by the inductive hyphothesis $\Omega_\nu^\gamma({\bf u}_1) \subseteq \Omega_\nu^{\gamma - \rho}({\bf u}_2)$, hence 
\begin{equation}\label{palermo - 1}
\omega \in \Omega_{\nu + 1}^\gamma({\bf u}_1) \qquad \Longrightarrow \qquad \omega \in \Omega_\nu^{\gamma - \rho}({\bf u}_2) \subseteq \Omega_\nu^{\gamma/2}({\bf u}_2) \,.
\end{equation}
Then for all $j \in \N$, the $2 \times 2$ matrices ${\bf D}_j^\nu(u_2) = {\bf D}_j^\nu(\omega, u_2(\omega))$ are well defined on $\Omega_{\nu + 1}^\gamma({\bf u}_1)$. We set for convenience 
$$
\Delta_{12} {\bf A}^-_\nu(\ell, j, j') := {\bf A}^-_\nu(\ell, j, j' ; u_2)  - {\bf A}^-_\nu(\ell, j, j' ; u_1)\,.
$$
By \eqref{palermo - 1}, on the set $\Omega_{\nu + 1}^\gamma({\bf u}_1)$, both the operators ${\bf A}_\nu^-(\ell, j, j' ; u_1)$ and ${\bf A}_\nu^-(\ell, j, j' ; u_2)$ are well defined. By the estimate \eqref{australia 2}, 
\begin{align}
\| \Delta_{12} {\bf A}^-_\nu(\ell, j, j') \|_{{\rm Op}(j, j')} \lessdot  \e \langle j - j' \rangle \| u_1 - u_2 \|_{s_0 + \overline\sigma + \mathtt b}\,. \label{palermo 0}
\end{align}
Now we write 
\begin{align}
& {\bf A}^-_\nu(\ell,j , j' ; u_2)  = {\bf A}^-_\nu(\ell, j, j' ; u_1) + \Delta_{12} {\bf A}_\nu^-(\ell, j, j') \nonumber\\
& = {\bf A}^-_\nu(\ell, j, j' ; u_1)  \Big\{ {\bf I}_{j,  j'} + {\bf A}^-_\nu(\ell, j, j' ; u_1)^{- 1} \Delta_{12} {\bf A}_\nu^-(\ell, j, j')  \Big\}\,. \label{palermo 2}
\end{align}
For all $|\ell| \leq N_\nu$, we have
\begin{align}
& \| {\bf A}^-_\nu(\ell, j, j' ; u_1)^{- 1} \Delta_{12} {\bf A}_\nu^-(\ell, j, j')  \|_{{\rm Op}(j, j')} \nonumber\\
&  \leq \| {\bf A}^-_\nu(\ell, j, j' ; u_1)^{- 1} \|_{{\rm Op}(j, j')} \| \Delta_{12} {\bf A}_\nu^-(\ell, j, j') \|_{{\rm Op}(j, j')} \nonumber\\
& \stackrel{\eqref{palermo 0}}{\lessdot} \frac{\langle \ell \rangle^\tau}{\gamma \langle j - j' \rangle} \e \langle j - j' \rangle \| u_1 - u_2 \|_{s_0 + \overline\sigma + \mathtt b}  \lessdot \langle \ell \rangle^\tau \e \gamma^{- 1} \| u_1 - u_2 \|_{s_0 + \overline\sigma + \mathtt b} \nonumber\\
& \lessdot   N_\nu^\tau \e \gamma^{- 1} \| u_1 - u_2 \|_{s_0 + \overline\sigma + \mathtt b} \stackrel{\eqref{caltanisetta}}{\leq} \rho \gamma^{- 1}\,.
\end{align}
Since $\rho \leq \gamma/2$, we get that the operator ${\bf I}_{j,  j'} + {\bf A}^-_\nu(\ell, j, j' ; u_1)^{- 1} \Delta_{12} {\bf A}_\nu^-(\ell, j, j')$ is invertible and by Neumann series we get 
\begin{equation}\label{palermo 1}
\Big\| \Big( {\bf I}_{j,  j'} + {\bf A}^-_\nu(\ell, j, j' ; u_1)^{- 1} \Delta_{12} {\bf A}_\nu^-(\ell, j, j') \Big)^{- 1} \Big\|_{{\rm Op}(j, j')} \leq \frac{\gamma}{\gamma - \rho}\,.
\end{equation}
By \eqref{palermo 2}, \eqref{palermo 1} we get 
\begin{align}
\|{\bf A}_\nu^-(\ell, j, j' ; u_2)^{- 1} \|_{{\rm Op}(j, j')} & \leq \frac{\gamma}{\gamma - \rho}  \|{\bf A}_\nu^-(\ell, j, j' ; u_1)^{- 1} \|_{{\rm Op}(j, j')} \nonumber\\
& \leq \frac{\gamma}{\gamma - \rho} \frac{\langle \ell\rangle^\tau}{\gamma \langle j - j' \rangle} \leq \frac{\langle \ell\rangle^\tau}{(\gamma - \rho) \langle j - j' \rangle}\,. \label{pippe 1}
\end{align}
By similar arguments one can also prove that on the set $\Omega_{\nu + 1}^\gamma({\bf u}_1)$ the following estimate holds
\begin{equation}\label{pippe 2}
\| {\bf A}_\nu^+(\ell, j, j' ; u_2)^{- 1}\|_{{\rm Op}(j, j')} \leq \frac{\langle \ell\rangle^\tau}{(\gamma - \rho) \langle j + j' \rangle}\,.
\end{equation}
Summarizing we have proved that if $\omega \in \Omega_{\nu + 1}^\gamma({\bf u}_1)$, then \eqref{pippe 1}, \eqref{pippe 2} hold, implying that $\omega \in \Omega_{\nu + 1}^{\gamma - \rho}({\bf u}_2)$ (recall the definition \eqref{Omgj}). This concludes the proof of $({\bf S4})_{\nu + 1}$. 

\subsection{Conjugation to a $2 \times 2$-block diagonal operator}
In this Section we prove that the operator ${\mathcal L}_0$ in \eqref{L0} can be conjugated to the $2 \times 2$, time independent, block-diagonal operator ${\mathcal L}_\infty$ in \eqref{bf L infinito esplicito}. This will be proved in Theorem \ref{teoremadiriducibilita} and it is a consequence of the KAM reducibility Theorem \ref{thm:abstract linear reducibility}. First, we state some auxiliary results.
\begin{corollary}\label{lem:convPhi}
{\bf (KAM transformation)} 
$ \forall \omega \in  \cap_{\nu \geq 0} \Omega_{\nu}^{\g} $ 
the sequence 
\be\label{Phicompo}
\widetilde{\Phi}_{\nu} := \Phi_{0} \circ \Phi_1 \circ \cdots\circ \Phi_{\nu} 
\ee
converges in $ |\cdot |_{s}^{\Lipg}$ to an operator $\Phi_{\infty} $
and 
\be\label{Phinftys}
\left|\Phi_{\infty}^{\pm 1} - {\mathbb I}_2 \right|_{s}^{\Lipg} \leq_s \left|{\mathcal R}_{0} |D|\right|_{s + \mathtt b}^{\Lipg} \gamma^{-1} \, . 
\ee
Moreover $\Phi_\infty^{\pm 1}$ is symplectic. 
\end{corollary}

\begin{proof}
The proof is similar to the one of Corollary 4.1 in \cite{BBM-Airy} and hence it is omitted.  
\end{proof}
By Theorem \ref{thm:abstract linear reducibility}-${\bf(S2)_{\nu}}$, for all $j \in \N$, the sequence of the $2 \times 2$ blocks $(\widetilde{\bf D}_{j}^{\nu})_{\nu \geq 0}$ (defined for $\omega \in \Omega_o$) is a Cauchy sequence in ${\mathcal S}({\bf E}_j)$ (recall \eqref{cal S E alpha}) with respect to $\| \cdot \|^\Lipg$, then, it converges to a limit ${\bf D}_j^\infty(\omega) \in {\mathcal S}({\bf E}_j)$, for any $\omega \in \Omega_o$. We have
\begin{align}
{\bf D}_j^\infty(\omega)&  : = \lim_{\nu \to + \infty} \widetilde{\bf D}_{j}^{\nu}(\omega) = \widetilde{\bf D}_j^0(\omega) + \widehat{\bf D}_j^\infty(\omega). \label{definizione bf D infinito}\\
\quad \widehat{\bf D}_j^\infty(\omega) & := \sum_{\nu \geq 0} \widetilde{\bf D}_j^{\nu + 1}(\omega) - \widetilde{\bf D}_j^\nu(\omega)\,. \nonumber
\end{align}
It could happen that $ \Omega_{\nu_0}^\g = \emptyset $ (see \eqref{Omgj}) for some $ \nu_0 $. In such a case
the iterative process of Theorem \ref{thm:abstract linear reducibility} stops  after finitely many steps. 
However, we can always set $ \widetilde{\bf D}_{j}^{\nu}  := \widetilde{\bf D}_{j}^{\nu_0} $, $ \forall \nu \geq \nu_0 $,
and the functions $ {\bf D}^{\infty}_{j} : \Omega_o \to {\mathcal S}({\bf E}_j)$  are always well defined.

\begin{corollary} {\bf (Final blocks)} For all $ \nu \geq 0, j \in \N $, 
\be\label{autovcon}
\| { \bf D }_{j}^{\infty} - \widetilde{\bf D}^{\nu}_{j} \|^{\Lipg}
\lessdot 
 N_{\nu-1}^{-\mathtt a} \e j^{- 1}   \, , \ \ 
 \| \widehat{\bf D}_j^\infty \|^{\Lipg}
 \lessdot \e j^{- 1}  \, . 
\ee
\end{corollary}

\begin{proof}
The bound \eqref{autovcon} follows by a telescopic argument, applying the estimate \eqref{lambdaestesi}. 
\end{proof}
Now we define the Cantor set 
\begin{align}
& \Omega_{\infty}^{2 \g}  := 
\Omega_{\infty}^{2 \gamma}({\bf u}) \nonumber\\
& := \Big\{\omega \in \Omega_o : \| {\bf A}_{\infty}^{-}(\ell, j, j' ; \omega)^{- 1} \| \leq \frac{\langle \ell \rangle^\tau}{2 \gamma \langle j - j' \rangle}\,, \quad \forall (\ell, j, j') \in \Z^\nu \times \N \times \N\,, \nonumber\\ 
& \qquad (\ell, j, j')\neq (0, j, j)\,,
 \qquad  \| {\bf A}_{\infty}^+(\ell, j, j'; \omega)^{- 1} \| \leq \frac{\langle \ell \rangle^\tau}{2 \gamma \langle j + j' \rangle}\,, \nonumber\\
 &  \forall (\ell, j, j') \in \Z^\nu \times \N \times \N\Big\}\,, \label{Omegainfty}
\end{align}
where the operators ${\bf A}^{\pm}_{\infty}(\ell, j, j') = {\bf A}^{\pm}_{\infty}(\ell, j, j'; \omega) = {\bf A}^{\pm}_{\infty}(\ell, j, j'; \omega, u(\omega)) : {\mathcal L}({\bf E}_{j'}, {\bf E}_j) \to {\mathcal L}({\bf E}_{j'}, {\bf E}_j)$ are defined by 
\begin{equation}\label{bf A infinito - (ell,alpha,beta)}
{\bf A}^{-}_{\infty}(\ell, j, j')  := \omega \cdot \ell {\bf I}_{j,  j'} + M_L({\bf D}_j^{\infty}) - M_R({\bf D}_{j'}^{\infty})\,,
\end{equation}
\begin{equation}\label{bf A infinito + (ell,alpha,beta)}
{\bf A}^{+}_{\infty}(\ell, j, j')  := \omega \cdot \ell {\bf I}_{j,  j'} + M_L({\bf D}_j^{\infty}) + M_R(\overline{\bf D}_{j'}^{\infty})\,.
\end{equation}

\begin{lemma} {\bf (Cantor set)}
\be\label{cantorinclu}
\Omega_{\infty}^{2 \g} \subset \cap_{\nu \geq 0} \Omega_{\nu}^\g \, . 
\ee 
\end{lemma}

\begin{proof} 
It suffices to show that for any $\nu \geq 0$, $\Omega_\infty^{2 \gamma} \subseteq \Omega_\nu^\gamma$. We argue by induction. For $\nu = 0$, since $\Omega_0^\gamma = \Omega_o$, it follows from the definition \eqref{Omegainfty} that $\Omega_\infty^{2 \gamma} \subseteq \Omega_0^\gamma$.
Assume that $\Omega_\infty^{2 \gamma} \subseteq \Omega_\nu^\gamma$ for some $\nu \geq 0$ and let us prove that $\Omega_\infty^{2 \gamma} \subseteq \Omega_{\nu + 1}^\gamma$. Let $\omega \in \Omega_\infty^{2 \gamma}$. By the inductive hyphothesis $\omega \in \Omega_\nu^\gamma$, hence by Theorem \ref{thm:abstract linear reducibility}, the self-adjoint matrices ${\bf D}_j^\nu(\omega) \in {\mathcal S}({\bf E}_j)$ are well defined for all $j \in \N$ and ${\bf D}_j^\nu(\omega) = \widetilde{\bf D}_j^\nu(\omega)$. By the definitions \eqref{bf A nu - (ell,alpha,beta)}, \eqref{bf A nu + (ell,alpha,beta)}, also the matrices ${\bf A}_\nu^{\pm}(\ell, j, j'; \omega)$ are well defined.
Since $\omega \in \Omega_\infty^{2 \gamma}$, ${\bf A}^-_\infty(\ell, j, j'; \omega)$ is invertible and we may write 
\begin{align*}
{\bf A}_\nu^-(\ell, j, j'; \omega)& = {\bf A}_\infty^- (\ell, j, j'; \omega) + \Delta_\infty^-(\ell, j, j'; \omega) \\
& =  {\bf A}_\infty^- (\ell, j, j'; \omega) \Big( {\bf I}_{j,  j'} + {\bf A}_\infty^- (\ell, j, j'; \omega)^{- 1}\Delta_\infty^-(\ell, j, j'; \omega)  \Big)  
\end{align*}
where
$$
\Delta^-_\infty(\ell, j, j'; \omega) := M_L\big( {\bf D}_j^\nu(\omega) - {\bf D}_j^\infty(\omega) \big) - M_R \big({\bf D}_{j'}^\nu(\omega) - {\bf D}_{j'}^\infty(\omega) \big)\,.
$$
By the property \eqref{norma operatoriale ML MR} and by the estimate \eqref{autovcon}
\begin{equation}\label{copenaghen 0}
\| {\Delta}^-_\infty(\ell, j, j'; \omega) \|_{{\rm Op}(j, j')} \lessdot N_{\nu - 1}^{- \mathtt a} \e j^{- 1}\,.
\end{equation}
Since for all $|\ell| \leq N_\nu$, $j, j' \in \N$, with $(\ell, j, j') \neq (0, j, j)$
\begin{equation}\label{copenaghen 1}
\| {\bf A}_\infty^-(\ell, j, j'; \omega)^{- 1} \Delta_\infty^-(\ell, j, j'; \omega)  \|_{{\rm Op}(j, j')} \lessdot  \frac{N_\nu^\tau N_{\nu - 1}^{-\mathtt a}}{2 \gamma \langle j - j' \rangle } \e  \leq \frac12\,,
\end{equation}
by \eqref{definizione alpha}, and for $\delta_0$ in \eqref{piccolezza1} small enough. Therefore the operator ${\bf A}_\nu^-(\ell, j, j'; \omega)$ is invertible, with inverse given by the Neumann series. For all $|\ell| \leq N_\nu$, $j, j' \in \N$ with $(\ell,j, j') \neq (0, j, j')$
\begin{align*}
\|{\bf A}_\nu^-(\ell, j, j'; \omega)^{- 1} \|_{{\rm Op}(j, j')} & \leq \frac{\| {\bf A}_\infty^-(\ell, j, j'; \omega)^{- 1}\|_{{\rm Op}(j, j')}}{1 - \| {\bf A}_\infty^-(\ell, j, j' ; \omega)^{- 1} {\Delta}_\infty^-(\ell, j, j'; \omega) \|_{{\rm Op}(j, j')}}  \\
& \stackrel{\eqref{copenaghen 1}}{\leq } 2 \| {\bf A}_\infty^-(\ell, j, j'; \omega)^{- 1}\|_{{\rm Op}(j, j')} \stackrel{\eqref{Omegainfty}}{\leq} \frac{\langle \ell \rangle^\tau}{\gamma \langle j - j' \rangle}\,.
\end{align*}
By similar arguments, one can also obtain that for any $(\ell, j, j') \in \Z^\nu \times \N \times \N$ with $|\ell| \leq N_\nu$
$$
\| {\bf A}_\nu^+(\ell, j, j'; \omega)^{- 1}\|_{{\rm Op}(j, j')} \leq \frac{\langle \ell \rangle^\tau}{\gamma \langle j + j' \rangle}\,
$$
and then the lemma follows. 
\end{proof}

\noindent
To state the main result of this section we introduce the operator
\begin{equation}\label{cal D infinito}
{\mathcal D}_\infty := \ii \begin{pmatrix}
 {\mathcal D}_\infty^{(1)} & 0 \\
0 &   - \overline{\mathcal D}_\infty^{(1)}
\end{pmatrix}\,,\quad {\mathcal D}_\infty^{(1)} := {\rm diag}_{j \in \N} {\bf D}_j^\infty\,,
\end{equation}
where the $2 \times 2$ self-adjoint blocks ${\bf D}_j^\infty \in {\mathcal S}({\bf E}_j)$ are defined in \eqref{definizione bf D infinito}. Furthermore, we introduce, for $\omega \in \Omega_o$, the operator
\begin{equation}\label{bf L infinito esplicito}
{\mathcal L}_\infty(\omega) := \omega \cdot \partial_\vphi {\mathbb I}_2 + {\mathcal D}_\infty(\omega).
\end{equation}
Then ${\mathcal L}_\infty(\omega)$ is a $\vphi$-independent block-diagonal bounded linear Hamiltonian operator ${\mathcal L}_\infty(\omega) : {\bf H}^s_0(\T^{\nu + 1}) \to {\bf H}^{s - 1}_0(\T^{\nu + 1})$, for any $s \geq 1$.

\begin{theorem}\label{teoremadiriducibilita}
Under the same assumptions of Theorem \ref{thm:abstract linear reducibility}, the following holds: 

\noindent
$(i)$ For all $\omega \in \Omega_\infty^{2\g}$ and $s \in [s_0, S - \bar \sigma - \mathtt b]$, the transformations $\Phi_\infty^{\pm 1} : {\bf H}^s_0(\T^{\nu + 1}) \to {\bf H}^s_0(\T^{\nu + 1})$ satisfy the estimates
\begin{equation} \label{stima Phi infty}
 \quad |\Phi_\infty^{\pm 1} - {\mathbb I}_2 |_{s}^\Lipg
\leq_s  \e \gamma^{- 1} (1 + \| {\bf u}\|_{s + \bar \sigma + \mathtt b}^\Lipg)  \,.
\end{equation}

\noindent
$(ii)$
On the set $ \Omega_\infty^{2 \gamma}$, the Hamiltonian operator ${\mathcal L}_0$ in \eqref{L0} is conjugated to the Hamiltonian operator ${\mathcal L}_\infty$ by $\Phi_\infty$, namely for all $\omega \in \Omega_\infty^{2 \gamma}$, 
\be\label{Lfinale}
{\mathcal L}_{\infty}(\omega)
= \Phi_{\infty}\inv(\omega) {\mathcal L}_0(\omega)   \Phi_{\infty}(\omega)\,.
\ee
\end{theorem}
\begin{proof} $(i)$ Since $\Omega_\infty^{2 \gamma}({\bf u}) \stackrel{\eqref{cantorinclu}}{\subseteq} \cap_{\nu \geq 0} \Omega_\nu^\gamma({\bf u})$, the estimate \eqref{Phinftys} holds on the set $\Omega_\infty^{2 \gamma}$, and $\forall s_0 \leq s \leq S - \bar \sigma - \mathtt b\,,$
$$
|\Phi_\infty^{\pm 1} - {\mathbb I}_2 |_{s}^\Lipg  {\leq_s}  \gamma^{- 1} |{\mathcal R}_0 |D||_{s + \mathtt b}^\Lipg \stackrel{\eqref{stime primo resto KAM}}{\leq_s}\, \e \gamma^{- 1} (1 + \| {\bf u}\|_{s + \bar \sigma + \mathtt b}^\Lipg)\,, \qquad 
$$
which is the claimed estimate \eqref{stima Phi infty}. 

\noindent
$(ii)$ By \eqref{Lnu+1}, \eqref{Phicompo} we get 
\begin{equation}\label{spanish}
{\mathcal L}_\nu = \widetilde{\Phi}_{\nu - 1}^{- 1} {\mathcal L}_0 \widetilde \Phi_{\nu - 1} = \omega \cdot \partial_\vphi {\mathbb I}_2 + {\mathcal D}_\nu + {\mathcal R}_\nu\,, \qquad \widetilde  \Phi_\nu = \Phi_0 \circ \ldots \circ \Phi_\nu\,.
\end{equation}
Note that, for all $\nu \geq 0$, 
\begin{align}
|{\mathcal D}_\infty^{(1)} - {\mathcal D}_\nu^{(1)}|^\Lipg_s & \stackrel{Lemma\,\ref{elementarissimo decay}-(ii)}{\leq} |({\mathcal D}_\infty^{(1)} - {\mathcal D}_\nu^{(1)}) |D||_s^\Lipg = \sup_{j \in \N} j \| {\bf D}_j^\infty - {\bf D}_j^\nu \|^\Lipg \nonumber\\
& \stackrel{\eqref{autovcon}}{\lessdot}  \e N_{\nu - 1}^{- \mathtt a} \stackrel{\nu \to + \infty}{\to} 0
\end{align}
and for any $s \in [s_0, S - \bar \sigma - \mathtt b]$ 
$$
|{\mathcal R}_\nu|_{s}^\Lipg \stackrel{Lemma\,\ref{elementarissimo decay}-(ii)}{\leq} |{\mathcal R}_\nu |D| |_{s}^\Lipg \stackrel{\eqref{Rsb}, \eqref{stime primo resto KAM}}{\lessdot} \e N_{\nu - 1}^{- \mathtt a} \big(1  +  \| {\bf u} \|_{s+ \bar \sigma + \mathtt b}^{\Lipg} \big) \stackrel{\nu \to + \infty}{\to} 0\,.
$$
Hence, $|{\mathcal L}_\nu - {\mathcal L}_\infty|_s^\Lipg \stackrel{\nu \to + \infty}{\to}0$ for all $s_0 \leq s \leq S - \bar \sigma - \mathtt b$. Since, by Lemma \ref{lem:convPhi}, $\widetilde\Phi_\nu^{\pm 1} \stackrel{\nu \to + \infty}{\to} \Phi_\infty^{\pm 1}$ in the norm $|\cdot |_{s}^\Lipg$, formula \eqref{Lfinale} follows by passing to the limit in \eqref{spanish}. 
\end{proof}
\begin{corollary}\label{stime Hs x Phi infty}
Assume \eqref{ansatz} with $\mu = \overline \sigma + \mathtt b + s_0$, then for any $s_0 \leq s \leq S - \overline \sigma - \mathtt b - s_0$, for any $\vphi \in \T^\nu$, the maps $\Phi_\infty^{\pm 1}(\vphi) : {\bf H}^s_0(\T_x) \to {\bf H}_0^s(\T_x)$ and it satisfies the estimates 
$$
|\Phi_\infty^{\pm 1}(\vphi)|_{s, x} \leq_s 1 + \| u \|_{s + \overline \sigma + \mathtt b + s_0}\,. 
$$ 
\end{corollary}
\begin{proof}
The claimed estimate follows by Lemma \ref{lemma decadimento Kirchoff in x}-$(ii)$ and by the estimate \eqref{stima Phi infty}. 
\end{proof}
\bigskip

\section{Inversion of the operator ${\mathcal L}$}\label{sezione invertibilita cal L}
We define 
\begin{equation}\label{definizione cal W 12}
{\mathcal W}_1 := {\mathcal S} {\mathcal B} ({\mathcal A}{\mathbb I}_2) {\mathcal V} \Phi_\infty\,, \quad {\mathcal W}_2 := {\mathcal S} {\mathcal B} ({\mathcal A}{\mathbb I}_2) \rho {\mathcal V} \Phi_\infty
\end{equation}
(recall the Definitions \eqref{cal S1}, \eqref{def cal A}, \eqref{cal W0}, \eqref{definition rho} and Corollary \ref{lem:convPhi}).
By Sections \ref{riduzione linearizzato}, \ref{sec:redu}, the operator ${\mathcal L}$ in \eqref{operatore linearizzato} may be written as 
\begin{equation}\label{semiconiugio}
{\mathcal L} = {\mathcal W}_2 {\mathcal L}_\infty {\mathcal W}_1^{- 1}\,,
\end{equation}
where the operator ${\mathcal L}_\infty$ is given in \eqref{bf L infinito esplicito}. 

\begin{lemma}\label{stime trasformazioni totali}
There exists $\mu_0 := \mu_0(\tau, \nu) > 0$, $\mu_0 \geq \bar \sigma + \mathtt b$ such that if \eqref{ansatz} holds with $\mu = \mu_0$, then the operators ${\mathcal W}_1$ and ${\mathcal W}_2$ defined in \eqref{definizione cal W 12} satisfies for any $s_0 \leq s \leq S - \mu_0$, $m = 1, 2$ 
$$
{\mathcal W}_m : {\bf H}^{s + \frac12}_0(\T^{\nu + 1}) \to H^s_0(\T^{\nu + 1} ,\R^2)\,, \quad {\mathcal W}_m^{- 1}: {H}^{s + \frac12}_0(\T^{\nu + 1} ,\R^2) \to {\bf H}^{s}_0(\T^{\nu + 1})
$$ 
$$
\| {\mathcal W}_m^{\pm 1} {\bf h}_{\pm} \|_s^\Lipg \leq_s \| {\bf h}_{\pm} \|_{s + \mu_0}^\Lipg + \| {\bf u} \|_{s + \mu_0}^\Lipg \| {\bf h}_{\pm} \|_{s_0 + \mu_0}^\Lipg\,,\quad m = 1,2 
$$
for any Lipschitz family of Sobolev functions ${\bf h}_+ (\cdot ; \omega) \in {\bf H}^{s + \mu_0}_0(\T^{\nu + 1})$, ${\bf h}_-(\cdot; \omega) \in H^{s + \mu_0}(\T^{\nu + 1}, \R^2)$, $\omega \in \Omega_\infty^{2 \gamma}({\bf u})$. 
\end{lemma}
\begin{proof}
The lemma follows by the estimates \eqref{stime cal S1 cal S2}, \eqref{stima cal A}, \eqref{stima cal A1}, \eqref{stima cal W01}, \eqref{stima rho}, \eqref{stima Phi infty} and applying also Lemma \ref{stime tame operatori decadimento Kirchoff}. 
\end{proof}

\noindent
 For all $\ell \in \Z^\nu$, for all $j \in \N$, for all $\omega \in \Omega_o = \Omega_o({\bf u})$, we define  the matrix ${\bf B}_\infty(\ell, j; \omega) = {\bf B}_\infty(\ell, j; \omega, u(\omega))$ as
\begin{equation}\label{matrice prime di Melnikov}
{\bf B}_\infty(\ell, j; \omega) := \omega \cdot \ell \,\,{\bf I}_{j} + {\bf D}_j^\infty(\omega)\,.
\end{equation}
Then Define the set 
\begin{equation}\label{prime e seconde di Melnikov}
\Lambda_{\infty}^{2 \gamma}({\bf u}) := \Big\{ \omega \in \Omega_o({\bf u}) : \| {\bf B}_\infty(\ell, j; \omega)^{- 1} \| \leq \frac{\langle \ell \rangle^\tau}{2 \gamma\, j}\,, \quad \forall (\ell, j) \in \Z^\nu \times \N \Big\}.
\end{equation}
We prove the following 
\begin{lemma}[\bf Invertibility of ${\mathcal L}_\infty$]\label{invertibilita cal D Kn}
For all $\omega \in \Lambda_{\infty}^{2 \gamma}({\bf u})$, the operator ${\mathcal L}_\infty$ is invertible and its inverse ${\mathcal L}_\infty^{- 1} : {\bf H}^{s + 2 \tau + 1}_0(\T^{\nu + 1}) \to {\bf H}^s_0(\T^{\nu + 1})$ satisfies the tame estimate 
$$
\| {\mathcal L}_\infty^{- 1} {\bf h} \|_s^\Lipg \lessdot \gamma^{- 1}\| {\bf h} \|_{s + 2 \tau + 1}^\Lipg
$$
for any Lipschitz family ${\bf h}(\cdot; \omega) \in {\bf H}^{s + 2 \tau + 1}_0(\T^{\nu + 1})$, $\omega \in \Lambda_\infty^{2 \gamma}({\bf u})$ 
\end{lemma}
\begin{proof}
By \eqref{cal D infinito}, \eqref{bf L infinito esplicito} the operator ${\mathcal L}_\infty$ has the form 
$$
{\mathcal L}_\infty = \begin{pmatrix}
{\mathcal L}_\infty^{(1)} & 0 \\
0 & \overline{\mathcal L}_\infty^{(1)}
\end{pmatrix}\,, \qquad {\mathcal L}_\infty^{(1)} := \omega \cdot \partial_\vphi + \ii {\mathcal D}_\infty^{(1)} \,,
$$
then it suffices to prove that ${\mathcal L}_\infty^{(1)}$ is invertible and its inverse satisfies the claimed estimate. 
 Let $\omega \in \Lambda_{\infty}^{2 \gamma}({\bf u})$, and let $h \in H^{s + \tau}_0(\T^{\nu + d})$. Using the block-representation \eqref{azione operator Toplitz a blocchi}, we have 
\begin{equation}\label{rappresentazione blocchi Fourier cal D Kn}
[{\mathcal L}_\infty^{(1)}]^{- 1} h(\vphi, x) = \sum_{(\ell, j) \in \Z^\nu \times \N} {\bf B}_\infty(\ell, j; \omega)^{- 1} [\widehat{\bf h}_j(\ell)] e^{\ii \ell \cdot \vphi}\,,
\end{equation}
where the $2 \times 2$ self-adjoint matrices ${\bf B}_\infty(\ell, j;\omega)$ are defined in \eqref{matrice prime di Melnikov}. Using \eqref{altro modo norma s}, we get 
\begin{align}
\| [{\mathcal L}_\infty^{(1)}]^{- 1} h \|_s^2 &  \leq  \sum_{(\ell, j) \in \Z^\nu \times \N} \langle \ell, j \rangle^{2 s} \| {\bf B}_\infty(\ell, j; \omega)^{- 1} \widehat{\bf h}_j(\ell) \|_{L^2}^2  \nonumber\\
&  \stackrel{\eqref{prime e seconde di Melnikov}}{\lessdot} \sum_{(\ell, j) \in \Z^\nu \times \N} \langle \ell, j \rangle^{2 s} \langle \ell \rangle^{2 \tau} \gamma^{- 2}  \|\widehat{\bf h}_j(\ell)\|_{L^2}^2  \lessdot \gamma^{- 2} \| h \|_{s + \tau}^2\,. \label{stima cal D Kn inv sup} 
\end{align}
Now let us consider a Lipschitz family of Sobolev functions $\omega \in \Lambda_\infty^{2 \gamma}({\bf u}) \to h(\cdot ; \omega) \in H^{s + 2 \tau + 1}_0$. Then, for any $\omega_1, \omega_2 \in \Lambda_\infty^{2 \gamma}({\bf u})$ one has that 
\begin{align}
& [{\mathcal L}_\infty^{(1)}(\omega_1)]^{- 1} h(\omega_1) - [{\mathcal L}_\infty^{(1)}(\omega_2)]^{- 1} h(\omega_2) \nonumber\\
& =  [{\mathcal L}_\infty^{(1)}(\omega_1)]^{- 1} \Big( h(\omega_1) - h(\omega_2) \Big) + \Big({\mathcal L}_\infty^{(1)}(\omega_1) - {\mathcal L}_\infty^{(1)}(\omega_2) \Big) h(\omega_2)\,. \label{decomposizione lip L infinito 1}
\end{align}
Arguing as in \eqref{stima cal D Kn inv sup}, the first term in \eqref{decomposizione lip L infinito 1} satisfies 
\begin{align}
 \Big\| [{\mathcal L}_\infty^{(1)}(\omega_1)]^{- 1} \Big( h(\omega_1) - h(\omega_2) \Big)\Big\|_s & \lessdot \gamma^{- 1} \| h(\omega_1) - h(\omega_2)\|_{s + \tau} \nonumber\\
 & \lessdot \gamma^{- 2} \| h \|_{s + \tau}^\Lipg |\omega_1 - \omega_2|\,. \label{stima primo pezzo decomposizione lip L infinito 1}
\end{align}
Now we estimate the second term in \eqref{decomposizione lip L infinito 1}. One has that 
\begin{align*}
& \Big({\mathcal L}_\infty^{(1)}(\omega_1)^{- 1} - {\mathcal L}_\infty^{(1)}(\omega_2)^{- 1} \Big) h(\omega_2)  \\
& = \sum_{(\ell, j) \in \Z^\nu \times \N} \Big({\bf B}_\infty(\ell, j; \omega_1)^{- 1} - {\bf B}_\infty(\ell, j; \omega_2)^{- 1} \Big)\widehat{\bf h}_j(\ell; \omega_2) e^{\ii \ell \cdot \vphi}\,.
\end{align*} 
Since 
\begin{align*}
&  {\bf B}_\infty(\ell, j; \omega_1)^{- 1} - {\bf B}_\infty(\ell, j; \omega_2)^{- 1} \\
&  = {\bf B}_\infty(\ell, j; \omega_1)^{- 1} \Big( {\bf B}_\infty(\ell, j; \omega_2) - {\bf B}_\infty(\ell, j ; \omega_1) \Big) {\bf B}_\infty(\ell, j; \omega_2)^{- 1}\,,
\end{align*}
using that $\omega_1, \omega_2 \in \Lambda_\infty^{2 \gamma}({\bf u})$, one has that 
\begin{align}
& \Big\|  {\bf B}_\infty(\ell, j; \omega_1)^{- 1} - {\bf B}_\infty(\ell, j; \omega_2)^{- 1} \Big\|  \nonumber\\
& \leq \frac{\langle \ell \rangle^{2 \tau}}{\gamma^2 j^2} \Big\| {\bf B}_\infty(\ell, j ; \omega_2) - {\bf B}_\infty(\ell, j; \omega_1) \Big\|\,. \label{acp}  
\end{align}
Furthermore, by \eqref{definizione bf D infinito}, \eqref{mu-j-nu} one gets 
\begin{align}
& \| {\bf B}_\infty(\ell, j; \omega_2) - {\bf B}_\infty(\ell, j; \omega_1) \| \nonumber\\
& \leq |\omega_1 - \omega_2||\ell| +  | m(\omega_1) - m(\omega_2) | j + \| \widehat{\bf D}_j^\infty(\omega_2) - \widehat{\bf D}_j^\infty(\omega_1) \| \nonumber\\
& \stackrel{ \eqref{stime m}, \eqref{autovcon}}{\lessdot} \big( |\ell| + \e \gamma^{- 1}  j  \big) |\omega_1 - \omega_2|\,. \label{mala femmina 0}
\end{align}
The estimates \eqref{acp}, \eqref{mala femmina 0} (using that $\e \gamma^{- 1} \leq 1$) imply that 
\begin{equation}\label{mala femmina 1}
\Big\|  {\bf B}_\infty(\ell, j; \omega_1)^{- 1} - {\bf B}_\infty(\ell, j; \omega_2)^{- 1} \Big\|  \lessdot  \langle \ell \rangle^{2 \tau + 1} \gamma^{- 2} |\omega_1 - \omega_2|\,.
\end{equation}
Therefore 
\begin{align}
& \Big\| \Big({\mathcal L}_\infty^{(1)}(\omega_1)^{- 1} - {\mathcal L}_\infty^{(1)}(\omega_2)^{- 1} \Big) h(\omega_2) \Big\|_s^2 \nonumber\\
&  \stackrel{\eqref{altro modo norma s}}{\leq}\sum_{(\ell, j) \in \Z^\nu \times \N} \langle \ell, j \rangle^{2 s} \Big\| \Big({\bf B}_\infty(\ell, j; \omega_1)^{- 1} - {\bf B}_\infty(\ell, j; \omega_2)^{- 1} \Big)\widehat{\bf h}_j(\ell; \omega_2) \Big\|_{L^2}^2 \nonumber\\
& \stackrel{\eqref{acp}}{\lessdot} \gamma^{- 4} \sum_{(\ell, j) \in \Z^\nu \times \N} \langle \ell, j \rangle^{2 s} \langle \ell \rangle^{4 \tau + 2} \|\widehat{\bf h}_j(\ell; \omega_2)\|_{L^2}^2  |\omega_1 - \omega_2|^2
\end{align}
which implies (recalling \eqref{altro modo norma s}) that 
\begin{equation}\label{stima secondo pezzo decomposizione lip L infinito 1}
\Big\| \Big({\mathcal L}_\infty^{(1)}(\omega_1) - {\mathcal L}_\infty^{(1)}(\omega_2) \Big) h(\omega_2) \Big\|_s \lessdot \gamma^{- 2} \| h \|_{s + 2 \tau + 1}^{\rm sup} |\omega_1 - \omega_2|\,.
\end{equation}
Hence, by \eqref{decomposizione lip L infinito 1}, \eqref{stima primo pezzo decomposizione lip L infinito 1}, \eqref{stima secondo pezzo decomposizione lip L infinito 1}, one gets 
\begin{equation}\label{acp 1}
\gamma \|[{\mathcal L}_\infty^{(1)}]^{- 1} h \|_s^{\rm lip} \lessdot \gamma^{- 1} \| h \|_{s + 2 \tau + 1}^\Lipg\,.
\end{equation}
Recalling \eqref{stima cal D Kn inv sup} and the definition of the norm $\| \cdot \|^{\Lipg}_s$ in \eqref{def norma Lipg}, the lemma follows. 
\end{proof}

\begin{theorem}[\bf Invertibility of ${\mathcal L}$]\label{partial invertibility of the linearized operator}
Let $\gamma \in (0, 1)$ and $\tau > 0$. There exists a constant 
\begin{equation}\label{costante finale mu 1 linearizzato}
\mu_1 = \mu_1(\tau, \nu) \geq \mu_0 \geq \overline \sigma + \mathtt b
\end{equation}
 where $\mu_0$ is given in Lemma \ref{stime trasformazioni totali}, such that for any Lipschitz family ${\bf u}(\cdot; \omega) \in { H}^S_0(\T^{\nu + 1}, \R^2)$, $S \geq s_0 + \mu_1$, satisfying 
\begin{equation}\label{ansatz finale inversione}
\|{\bf u} \|_{s_0 + \mu_1}^\Lipg \leq 1
\end{equation}
there exists a constant $\delta_1 = \delta_1(S, \tau, \nu) > 0$ (eventually smaller than the constant $\delta_0$ given in Theorem \ref{thm:abstract linear reducibility}) such that if 
\begin{equation}\label{final smallness inversion}
\e \gamma^{- 1} \leq \delta_1\,,
\end{equation}
then for all $\omega \in {\bf \Omega}_{\infty}^{2 \gamma}({\bf u}) := \Omega_\infty^{2 \gamma}({\bf u}) \cap \Lambda_\infty^{2 \gamma}({\bf u})$ $($see \eqref{prime e seconde di Melnikov}, \eqref{Omegainfty}$)$, the operator ${\mathcal L}$ defined in \eqref{operatore linearizzato} is invertible and its inverse ${\mathcal L}^{- 1} : H^{s + \mu_1}_0(\T^{\nu +1}, \R^2) \to H^{s}_0(\T^{\nu +1}, \R^2)$ satisfies $\forall s_0 \leq s \leq S - \mu_1$ the tame estimate  
\begin{equation}\label{stima cal T}
\| {\mathcal L}^{- 1} {\bf h} \|_s^\Lipg \leq_s \gamma^{- 1} \big( \| {\bf h} \|_{s + \mu_1}^\Lipg + \| {\bf u} \|_{s + \mu_1}^\Lipg\| {\bf h} \|_{s_0 + \mu_1}^\Lipg \big)\,, \quad 
\end{equation}
for any Lipschitz family ${\bf h}(\cdot; \omega) \in H^{s + \mu_1}_0(\T^{\nu +1}, \R^2)$, $\omega \in {\bf \Omega}_\infty^{2 \gamma}({\bf u})$. 
\end{theorem}
\begin{proof}
 The estimate \eqref{stima cal T} follows by \eqref{semiconiugio} and by Lemmata \ref{stime trasformazioni totali}, \ref{invertibilita cal D Kn}. 
 \end{proof}

\section{The Nash-Moser iteration}\label{sec:NM}
 
Our next goal is to prove Theorem \ref{main theorem kirchoff 1}. It will be a consequence of Theorem \ref{iterazione-non-lineare} below where we construct iteratively a sequence of better and better approximate solutions of the operator
${\mathcal F}({\bf u} ) = {\mathcal F} (\e, \omega, {\bf u})$, defined in \eqref{operatore su media nulla} and of the Sections \ref{sezione stime in misura}, \ref{sezione conclusiva dim}.

\noindent 
We consider the finite-dimensional subspaces 
$$
{\mathcal H}_n := \Big\{ {\bf u} \in L_0^2(\T^{\nu + 1}, \R^2): \ \  
{\bf u} = { \Pi}_n {\bf u} \Big\},
$$
where $ {\Pi}_n  $ is the projector
\be\label{truncation NM}
{ \Pi}_n {\bf u} := (\Pi_n u, \Pi_n \psi) \,, \qquad \Pi_n h(\vphi, x) := \sum_{|(\ell, j)| \leq N_n} \widehat h_j(\ell) e^{\ii (\ell \cdot \vphi + j \cdot x)}
\ee
with $ N_n = N_0^{\chi^n} $ (see \eqref{defN}).
We also define $ \Pi_n^\bot := {\rm Id} - \Pi_n $.  
The projectors $ \Pi_n $, $ \Pi_n^\bot$ satisfy the following classical  smoothing properties
for the weighted norm $\| \cdot \|_s^\Lipg$, namely 
$ \forall  s, b \geq 0$, 
\begin{align}\label{smoothing-u1}
\|\Pi_{n} {\bf u} \|_{s + b}^\Lipg 
\leq K_{n}^{b} \| {\bf u} \|_{s}^\Lipg \, , 
\qquad \ \,   \|\Pi_{n}^\bot {\bf u} \|_{s}^\Lipg 
& \leq K_n^{- b} \| {\bf u} \|_{s + b}^\Lipg \,   . 
\end{align}
In view of the Nash-Moser Theorem \ref{iterazione-non-lineare} we introduce the constants 
\begin{align}\label{costanti nash moser}
&   \kappa := 6 \mu_1 + 19\,, \quad {\mathtt b}_1 := 2 \mu_1 + 4  + \kappa (1 + \chi^{- 1}) + 1\,,   \\
&   \mathtt a_1 :=  \kappa \chi^{- 1} - 2\mu_1\,, \quad \chi := \frac32
\label{costanti nash moser 1}
\end{align}
where 
 $\mu_1 := \mu_1(\tau, \nu) > 0$ is given in Theorem \ref{partial invertibility of the linearized operator}.  
\begin{theorem}\label{iterazione-non-lineare} 
{\bf (Nash-Moser)} 
Let $\gamma \in (0, 1)$ and $\tau > 0$. Assume that $ f \in {\mathcal C}^q(\T^{\nu} \times \T, \R) $, with $ q \geq  s_0 + \mathtt b_1  $. There exist $ \d \in (0, 1)$ small enough and $N_0 > 0$, $C_* > 0$ large enough such that if
\begin{equation}\label{nash moser smallness condition}  
 \e \gamma^{- 1}  \leq \d 
\end{equation}
 then: 
\begin{itemize}
\item[$({\mathcal P}1)_{n}$] 
 For all $n \geq 0$, there exists  a function ${\bf u}_n := (u_n, \psi_n) : \mG_n \subseteq \Omega \to {\mathcal H}_n$, $\omega \mapsto {\bf u}_n(\omega) = (u_n(\omega), \psi_n(\omega))$,  
with 
\begin{equation}\label{ansatz induttivi nell'iterazione}
 \| {\bf u}_{n} \|_{s_0 + \mu_1}^{\Lipg} \leq 1 \,, \qquad  
 \end{equation}
  $ {\bf u}_0 := 0$, 
where ${\mathcal G}_{n} $ are Cantor like subsets of $\Omega$ defined inductively by: 
\begin{equation}\label{G-n+1}
{\mathcal G}_0 := \Omega_{\gamma, \tau}\,, \quad {\mathcal G}_{n + 1} :=  {\bf \Omega}_\infty^{2 \gamma_n}({\bf u}_n)\,, \quad \forall n \geq 0\,,
\end{equation} 
(recall \eqref{diofantei Kn}) where $ \gamma_{n}:=\gamma (1 + 2^{-n}) $ and the set ${\bf \Omega}^{2 \gamma_n}({\bf u}_n)$ is given in Theorem \ref{partial invertibility of the linearized operator}, with $\Omega_o({\bf u}_n) = {\mathcal G}_n$. There exists a constant $C_*' > C_*$ such that for all $n \geq 1$, the difference ${\bf h}_n := {\bf u}_{n} - {\bf u}_{n-1}$ satisfies
\be  \label{hn}
\| {\bf h}_{n} \|_{s_0 + \mu_1}^\Lipg \leq C_*' \e \gamma^{-1} N_{n }^{-\mathtt a_1}\,.
\ee
\item[$({\mathcal P}2)_{n}$] For all $n \geq 0$, $ \| {\mathcal F}({\bf u}_n) \|_{s_{0}}^{\Lipg} \leq C_* \e N_{n }^{- \kappa}$.

\item[$({\mathcal P}3)_{n}$] For all $n \geq 0$, $ \|{\bf u}_n \|_{\mathfrak s_{0}+ \mathtt b_1}^{\Lipg} 
\leq C_* \e \gamma^{-1} N_{n }^{\kappa}\,, \qquad \| {\mathcal F}({\bf u}_n)\|_{s_0 + \mathtt b_1}^\Lipg \leq C_* \e N_{n }^\kappa$
\end{itemize}
All the Lip norms are defined on $ {\mathcal G}_{n} $.

\end{theorem}

\begin{proof}
To simplify notations, in this proof we write $\| \cdot \|_s $ instead of $\| \cdot \|^\Lipg_s$. 

\medskip

{\sc Step 1:} \emph{Proof of} $({\mathcal P}1, 2, 3)_0$.
They follow, since by \eqref{operatore su media nulla}, 
$$
{\mathcal F}(0) = \begin{pmatrix}
0 \\
\e f_\bot(\vphi, x)
\end{pmatrix}\,, \quad \|{\mathcal F}(0) \|_s \leq \e \| f\|_s
$$
and taking $C_* > {\rm max}\{ \| f \|_{s_0} N_0^\kappa\,,\, \| f \|_{s_0 + \mathtt b_1} \} $.

\medskip

{\sc Step 2:} \emph{Assume that $({\mathcal P}1,2,3)_n$ hold for some $n \geq 0$, and prove $({\mathcal P}1,2,3)_{n+1}$.}
We are going to define the successive approximation ${\bf u}_{n+1} $. By $({\mathcal P}1)_n$, one has $\| {\bf u}_n\|_{s_0 + \mu_1} \leq 1$, moreover the smallness condition \eqref{nash moser smallness condition} implies
the smallness condition \eqref{final smallness inversion} of Theorem \ref{partial invertibility of the linearized operator}, by taking $\delta < \delta_1 = \delta_1(S, \tau, \nu)$ with $S = s_0 + \mu_1 + \mathtt b_1$. Then Theorem \ref{partial invertibility of the linearized operator} can be applied to the linearized operator 
\begin{equation}\label{definizione cal Ln}
{\mathcal L}_n = {\mathcal L}({\bf u}_n) = \partial_{\bf u} {\mathcal F}({\bf u}_n)\,,
\end{equation}
implying that for all $\omega \in {\mathcal G}_{n + 1} = {\bf \Omega}_\infty^{2 \gamma_n}({\bf u}_n)$, the operator ${\mathcal L}_n$ is invertible and its inverse satisfies $\forall s_0 \leq s \leq s_0 + \mathtt b_1$, $\forall {\bf h} \in {H}_0^{s + \mu_1}(\T^{\nu + 1}, \R^2)$, the tame estimate
\begin{equation}\label{stima Tn}
\| {\mathcal L}_n^{- 1} {\bf h} \|_s \leq_s \gamma^{- 1} \Big( \| {\bf h} \|_{s + \mu_1}  + \| {\bf u}_n\|_{s + \mu_1} \| {\bf h} \|_{s_0 + \mu_1}\Big)\,.
\end{equation}
Specializing the above estimate for $s= s_0$, using \eqref{ansatz induttivi nell'iterazione}, one gets 
\begin{align}
& \| {\mathcal L}_n^{- 1} {\bf h} \|_{s_0} \leq_{s_0} \gamma^{- 1} \| {\bf h} \|_{s_0 + \mu_1}\,. \label{stima Tn norma bassa} \\
\end{align}
 We define the successive approximation 
\begin{equation}\label{soluzioni approssimate}
{\bf u}_{n + 1} := {\bf u}_n + {\bf h}_{n + 1}\,, \quad 
{\bf h}_{n + 1}  :=  - {\Pi}_{n + 1 } {\mathcal L}_n^{- 1} \Pi_{n + 1} {\mathcal F}({\bf u}_n) 
\in {\mathcal H}_{n + 1}  
\end{equation}
where  $ {\Pi}_n $ is defined in \eqref{truncation NM}.
We now show that the iterative scheme in \eqref{soluzioni approssimate} is rapidly converging. 
We write  
$$ 
{\mathcal F}({\bf u}_{n + 1}) =  {\mathcal F}({\bf u}_n) + {\mathcal L}_n {\bf h}_{n + 1} + Q_n 
$$ 
where $ {\mathcal L}_n := \partial_{\bf u} {\mathcal F}({\bf u}_n) $ and  
\begin{equation}\label{def:Qn}
Q_n := Q({\bf u}_n, {\bf h}_{n + 1}) \, , \quad 
Q ({\bf u}_n, {\bf h})  :=  {\mathcal F}({\bf u}_n + {\bf h} ) - {\mathcal F}({\bf u}_n) - {\mathcal L}_n {\bf h} \,,
\end{equation}
${\bf h} \in {\mathcal H}_{n + 1}$. Then, by the the definition of $ {\bf h}_{n+1} $ in \eqref{soluzioni approssimate}, writing $\Pi_{n + 1} = {\rm Id} - \Pi_{n + 1}^\bot$,
we have 
\begin{align}
{\mathcal F}({\bf u}_{n + 1}) & = 
 {\mathcal F}({\bf u}_n) - {\mathcal L}_n { \Pi}_{n +1} {\mathcal L}_n^{- 1} \Pi_{n + 1}  {\mathcal F}({\bf u}_n) + Q_n \nonumber\\
 & = \Pi_{n + 1}^\bot {\mathcal F}({\bf u}_n) +  R_n + Q_n\,, \label{relazione algebrica induttiva} 
 \end{align}
where 
\begin{equation}\label{Rn Q tilde n}
R_n := [{\mathcal L}_n \,,\, { \Pi}_{n + 1}^\bot ] {\mathcal L}_n^{- 1} \Pi_{n + 1} {\mathcal F}({\bf u}_n) = [\Pi_{n + 1}\,,\, {\mathcal L}_n] {\mathcal L}_n^{- 1} \Pi_{n + 1} {\mathcal F}({\bf u}_n) \,.
\end{equation}
We first note that,  for all $ \om \in {\mathcal G}_n$, by $({\mathcal P}2)_n$ and by the smallness condition \eqref{nash moser smallness condition}, one has that 
\begin{equation}\label{piccolo ansatz Fn}
 \| {\mathcal F}({\bf u}_n)\|_{s_0} \gamma^{- 1} \leq 1\,.
\end{equation} 
\begin{lemma}
On the set  ${\mathcal G}_{n + 1}$, defining  
\begin{equation}\label{definizione Bn Nash Moser}
B_n :=  \| {\mathcal F}({\bf u}_n)\|_{s_0 + \mathtt b_1} + \e \| {\bf u}_n\|_{s_0 + \mathtt b_1}\,,
\end{equation}
we have
\begin{align}
B_{n + 1} & \leq_{s_0 +\mathtt b_1} N_{n + 1}^{2 \mu_1 + 6} B_n\,, \label{stima induttiva Bn Nash Moser}\\
  \| {\mathcal F}({\bf u}_{n + 1})\|_{s_0}  & \leq_{s_0 + {\mathtt b}_1}    N_{n + 1}^{2 \mu_1 + 4 - {\mathtt b}_1} 
 B_n 
+ N_{n + 1}^{2 \mu_1 + 6} \e \gamma^{- 2} \| {\mathcal F}({\bf u}_n)\|_{s_0}^2 \,.
\label{F(U n+1) norma bassa} 
\end{align}

\end{lemma}

\begin{proof} 
We first estimate $ {\bf h}_{n +1} $ defined in  \eqref{soluzioni approssimate}.

\noindent
{\bf Estimates of $ {\bf h}_{n+1} $.}
 By \eqref{soluzioni approssimate} and \eqref{smoothing-u1}, 
\eqref{stima Tn} (applied for $s = s_0 + \mathtt b_1$), \eqref{stima Tn norma bassa}, \eqref{ansatz induttivi nell'iterazione}, we get 
\begin{align}
\|  {\bf h}_{n + 1} \|_{s_0 + {\mathtt b}_1} 
& \leq_{s_0 + {\mathtt b}_1} \gamma^{- 1}
\big(  \|\Pi_{n + 1}{\mathcal F}({\bf u}_n) \|_{s_0 + \mu_1 +  {\mathtt b}_1}  \nonumber\\
& \quad \quad \quad  + 
\|{\bf u}_n \|_{s_0 + \mu_1 +  {\mathtt b}_1}\| \Pi_{n + 1} {\mathcal F}({\bf u}_n)\|_{s_0 + \mu_1}  \big)
\nonumber \\
& \stackrel{ \eqref{smoothing-u1}}{\leq_{s_0 + {\mathtt b}_1}} N_{n + 1}^{2 \mu_1  } \gamma^{- 1}\big( \| {\mathcal F}({\bf u}_n)\|_{s_0 + \mathtt b_1}  + \|{\bf u}_n \|_{s_0 + \mathtt b_1} \| {\mathcal F}({\bf u}_n)\|_{s_0} \big)\, ,   \nonumber\\
& \stackrel{\eqref{piccolo ansatz Fn}}{\leq_{s_0 + \mathtt b_1}} N_{n + 1}^{2 \mu_1} \Big(\gamma^{- 1} \| {\mathcal F}({\bf u}_n)\|_{s_0 + \mathtt b_1} +  \|{\bf u}_n \|_{s_0 + \mathtt b_1} \Big)\label{H n+1 alta} 
\\
\label{H n+1 bassa}
\|  {\bf h}_{n + 1}\|_{s_0} 
& \leq_{s_0} \gamma^{-1}N_{n + 1}^{\mu_1} \|  {\mathcal F}({\bf u}_n )\|_{s_0} \, .
\end{align}
Now we  estimate the terms $ Q_n $ in \eqref{def:Qn} and $R_n$ in \eqref{Rn Q tilde n}. 
\\[1mm]
{\bf Estimate of $ Q_n $.}
By  \eqref{def:Qn}, \eqref{operatore su media nulla}, \eqref{interpolazione C1 gamma}, \eqref{ansatz induttivi nell'iterazione}, \eqref{smoothing-u1}, 
we have the quadratic estimate
\be\label{stima parte quadratica norma bassa}
\| Q( {\bf u}_n , {\bf h}) \|_{s}  \leq_{s}\e N_{n + 1}^6  \Big( \| {\bf h}\|_{s} \| {\bf h}\|_{s_0} + \| {\bf u}_n\|_{s} \| {\bf h}\|_{s_0}^2 \Big)   \, , 
\ee
$\forall  {\bf h} \in {\mathcal H}_{n + 1}$, $\forall s \geq s_0$. Then the term $ Q_n $ in \eqref{def:Qn} satisfies, 
by \eqref{stima parte quadratica norma bassa}, \eqref{H n+1 alta}, 
 \eqref{H n+1 bassa}, \eqref{piccolo ansatz Fn}
\begin{align} 
\| Q_n\|_{s_0 + \mathtt b_1} & \leq_{s_0 + \mathtt b_1} N_{n + 1}^{2 \mu_1 + 6} \e \Big( \| {\mathcal F}({\bf u}_n)\|_{s_0 + \mathtt b_1} \gamma^{- 1} + \| {\bf u}_n\|_{s_0 + \mathtt b_1} \Big)\,, \label{Qn norma alta} \\
\| Q_n \|_{s_0} 
&  \leq_{s_0}  N_{n + 1}^{2 \mu_1 + 6 } \e \gamma^{- 2} \| {\mathcal F}({\bf u}_n ) \|_{s_0}^2\, . \label{Qn norma bassa}
\end{align}
{\bf Estimate of $ R_n $.} Now we estimate the term $R_n$ defined in \eqref{Rn Q tilde n}. By \eqref{definizione cal Ln}, \eqref{operatore linearizzato}, Lemma \ref{interpolazione C1 gamma}, \eqref{smoothing-u1}, \eqref{ansatz induttivi nell'iterazione}  the operator ${\mathcal L}_n$ satisfies the estimates 
\begin{align}
& \| [{\mathcal L}_n, \Pi_{n + 1}^\bot] {\bf h}\|_{s_0}   \leq_{s_0 + \mathtt b_1} N_{n + 1}^{- \mathtt b_1 + 2} \e \Big( \| {\bf h}\|_{s_0 + \mathtt b_1 } + \| {\bf u}_n\|_{s_0 + \mathtt b_1 } \| {\bf h}\|_{s_0 + 2} \Big)\,,  \nonumber\\
& \| [{\mathcal L}_n, \Pi_{n + 1}^\bot] {\bf h}\|_{s_0 + \mathtt b_1}  = \| [ \Pi_{n + 1}, {\mathcal L}_n] {\bf h}\|_{s_0 + \mathtt b_1} \nonumber\\
&  \leq_{s_0 + \mathtt b_1} N_{n + 1}^2 \e \Big( \| {\bf h}\|_{s_0 + \mathtt b_1 } + \| {\bf u}_n\|_{s_0 + \mathtt b_1 } \| {\bf h}\|_{s_0 + 2} \Big)\,.\nonumber
\end{align}
The above estimates, together with the estimates \eqref{stima Tn}, \eqref{stima Tn norma bassa}, \eqref{smoothing-u1} and using also \eqref{piccolo ansatz Fn} imply 
\begin{align}
\| R_n \|_{s_0} &\leq_{s_0 + \mathtt b_1 } N_{n + 1}^{2 \mu_1 + 4 - \mathtt b_1} \e \Big( \| {\mathcal F}({\bf u}_n)\|_{s_0 + \mathtt b_1} \gamma^{- 1} + \| {\bf u}_n\|_{s_0 + \mathtt b_1} \Big)\,,  \label{stima Rn Nash-Moser} \\
\| R_n \|_{s_0 + \mathtt b_1} &\leq_{s_0 + \mathtt b_1 } N_{n + 1}^{2 \mu_1 + 4} \e \Big( \| {\mathcal F}({\bf u}_n)\|_{s_0 + \mathtt b_1} \gamma^{- 1} + \| {\bf u}_n\|_{s_0 + \mathtt b_1} \Big)\,.  \label{stima Rn Nash-Moser norma alta}
\end{align}

\noindent
{\bf Estimates of ${\bf u}_{n + 1}$.} By \eqref{soluzioni approssimate} and by the estimates \eqref{H n+1 alta} one gets
\begin{align}
\| {\bf u}_{n + 1}\|_{s_0 + \mathtt b_1} & \leq_{s_0 + \mathtt b_1} N_{n + 1}^{2 \mu_1} \Big( \| {\bf u}_n\|_{s_0 + \mathtt b_1} + \| {\mathcal F}({\bf u}_n)\|_{s_0 + \mathtt b_1} \gamma^{- 1} \Big)\,. \label{stima u n+1 alta}
\end{align}

\noindent
Finally, by \eqref{relazione algebrica induttiva}, \eqref{Qn norma alta},  
 \eqref{Qn norma bassa}, \eqref{stima Rn Nash-Moser}, \eqref{stima Rn Nash-Moser norma alta}, \eqref{stima u n+1 alta}, \eqref{smoothing-u1}, $\e \gamma^{- 1} \leq 1$ and recalling the definition \eqref{definizione Bn Nash Moser}, we deduce the estimates \eqref{F(U n+1) norma bassa}, \eqref{stima induttiva Bn Nash Moser}. 
 \end{proof}
 \noindent
The estimates \eqref{F(U n+1) norma bassa}, \eqref{stima induttiva Bn Nash Moser}, together with \eqref{costanti nash moser}, \eqref{costanti nash moser 1}, $({\mathcal P}2)_n, ({\mathcal P}3)_n$, \eqref{nash moser smallness condition},  taking $\delta_1$ small enough and $N_0$ large enough, imply $({\mathcal P}2)_{n + 1}$, $({\mathcal P}3)_{n + 1}$.

\noindent
The estimate \eqref{hn} at the step $n + 1$ follows since 
\begin{align*}
\| {\bf h}_{n + 1}\|_{s_0 + \mu_1} & \stackrel{\eqref{smoothing-u1}}{\leq} N_{n + 1}^{\mu_1} \| {\bf h}_{n + 1}\|_{s_0} \stackrel{\eqref{H n+1 bassa}}{\leq} C(s_0) N_{n + 1}^{2 \mu_1} \gamma^{- 1}\|{\mathcal F}({\bf u}_n) \|_{s_0} \\
&  \stackrel{({\mathcal P}2)_n}{\leq} C(s_0) C_* N_{n + 1}^{2 \mu_1} N_n^{- \kappa} \e \gamma^{- 1}
\end{align*}
which implies the claimed estimate, by \eqref{costanti nash moser 1} and taking $C_*' = C(s_0) C_*$. 

\noindent
The estimate \eqref{ansatz induttivi nell'iterazione} at the step $n + 1$ follows since 
$$
\| {\bf u}_{n + 1}\|_{s_0 + \mu_1} \leq \sum_{k = 0}^{n + 1} \| {\bf h}_{k}\|_{s_0 + \mu_1} \stackrel{\eqref{hn}}{\leq} C_*' \e \gamma^{- 1} \sum_{k \geq 0} N_k^{- \mathtt a_1} \leq C' \e \gamma^{- 1} \stackrel{}{\leq} 1
$$
by taking $\delta$ in \eqref{nash moser smallness condition} small enough. Then $({\mathcal P}1)_{n + 1}$ follows and the proof is concluded.
\end{proof}

\bigskip

\noindent
\section{Measure estimates}\label{sezione stime in misura}
 In this Section we estimate the measure of the set 
 \begin{equation}\label{insieme di Cantor cal G infinito} 
{\mathcal G}_\infty := \cap_{n \geq 0} {\mathcal G}_n\,.
\end{equation}
where the sets $\ldots \subseteq {\mathcal G}_{n + 1} \subseteq {\mathcal G}_n \subseteq \ldots \subseteq {\mathcal G}_0$ are given in Theorem \ref{iterazione-non-lineare}-$({\mathcal P}1)_n$.
 First, let us define the constants  
\begin{equation}\label{valori finali tau tau*}
\tau^* := \nu + 2\,, \qquad \tau := 2 \tau^*+ \nu + 1\,,
\end{equation}
\begin{equation}\label{definizione gamma* tau*}
\gamma^* := 5 \gamma\,, \qquad \gamma_n^* := \gamma^* (1 + 2^{- n})\,, \qquad \forall n \geq 0\,
\end{equation}
and recall also that the constants
\begin{equation}\label{gamma n stime misura}
\gamma_n = \gamma(1 + 2^{- n})\,, \qquad \forall n \geq 0\,,
\end{equation}
are given in Theorem \ref{iterazione-non-lineare}-$({\mathcal P}1)_n$.
We prove the following 
\begin{theorem}\label{stima finale parametri cattivi}
One has 
$$
|\Omega \setminus {\mathcal G}_\infty| \lessdot \gamma\,.
$$
\end{theorem}
we write 
\begin{align}
\Omega \setminus {\mathcal G}_\infty & = (\Omega \setminus {\mathcal G}_0) \bigcup_{n \geq 0} ({\mathcal G}_n \setminus {\mathcal G}_{n + 1})\,. \label{splitting Omega - G infty 1}
\end{align}
Since ${\mathcal G}_0 = \Omega_{\gamma, \tau}$, see \eqref{G-n+1}, it follows by standard volume estimates 
\begin{equation}\label{pupella 0}
\Omega \setminus {\mathcal G}_0  = O(\gamma)\,.
\end{equation}

\noindent
For any $n \geq 0$, we define the set 
\begin{align}\label{prime di Melnikov}
 \Omega_{\gamma^*, \tau^*}^{(I)}({\bf u}_n) & :=  \big\{ \omega \in {\mathcal G}_n  :   |\omega \cdot \ell + m({ u}_n) j  | \geq \frac{\gamma_n^*  \langle j \rangle}{\langle \ell \rangle^{\tau^*}}\,,\\
 &  \quad \forall (\ell, j ) \in (\Z^\nu \times \N_0) \setminus \{(0, 0) \} \big\} \nonumber
\end{align} 
where we recall that the constant $m({ u}_n) = m(\omega, { u}_n(\omega))$ satisfies the estimates \eqref{stime m} (recall also that $\N_0 = \{ 0 \} \cup \N = \{0,1,2, \ldots \}$). 
For all $n \geq 0$, we make the splitting
\begin{equation}\label{pappa pappa}
{\mathcal G}_n \setminus {\mathcal G}_{n + 1} = {\mathcal A}_n^{(1)} \cup {\mathcal A}_n^{(2)}\,,
\end{equation}
where 
\begin{align}
& {\mathcal A}_n^{(1)}  := ({\mathcal G}_n \setminus {\mathcal G}_{n + 1}) \cap \Omega_{\gamma^*, \tau^*}^{(I)}({\bf u}_n)\,, \label{cal An (1)} \\
& \,  {\mathcal A}_n^{(2)}  :=  ({\mathcal G}_n \setminus {\mathcal G}_{n + 1})  \cap ({\mathcal G}_n \setminus \Omega_{\gamma^*, \tau^*}^{(I)}({\bf u}_n))\,. \label{cal An (2)}  
\end{align} 

\noindent
{\bf Estimate of ${\mathcal A}_n^{(1)}$.}

\bigskip

\noindent
By \eqref{cal An (1)}, using the inductive definition of the sets ${\mathcal G}_n$ given in \eqref{G-n+1} and recalling \eqref{prime e seconde di Melnikov}, \eqref{Omegainfty}, one has that for all $n \geq 0$,
\begin{align}
& {\mathcal A}_n^{(1)} =   \bigcup_{\begin{subarray}{c}
(\ell, j) \in \Z^\nu \times \N 
\end{subarray}} Q_{\ell j}({\bf u}_n) \bigcup \bigcup_{\begin{subarray}{c}
\ell \in \Z^\nu \\
j, j' \in \N \\
 (\ell, j, j') \neq (0, j, j)
\end{subarray}} R_{\ell j j'}^{-}({\bf u}_n) \bigcup \bigcup_{\begin{subarray}{c}
\ell \in \Z^\nu \\
j, j' \in \N
\end{subarray}}R_{\ell j j'}^{+}({\bf u}_n) \label{espansione risonanti}
\end{align}
where 
\begin{align}
\label{risonanti prime Melnikov} Q_{\ell j}({\bf u}_n) & := \Big\{ \omega \in {\mathcal G}_n \bigcap \Omega_{\gamma^*, \tau^*}^{(I)}({\bf u}_n) : \text{the \, operator}\, {\bf B}_\infty(\ell, j ; \omega , { u}_n(\omega)) \nonumber\\
&\text{is\,\, not\,\,invertible\,\,or} \,\text{\,\,it \,\,is \,\,invertible\,\,and}\, \nonumber\\
&  \| { \bf B}_\infty(\ell, j ; \omega , { u}_n(\omega))^{- 1}\| > \frac{\langle \ell \rangle^\tau}{2 \gamma_n j}  \Big\}\,,
\end{align}
\begin{align}
R_{\ell j j'}^-({\bf u}_n) & := \Big\{ \omega \in {\mathcal G}_n  \bigcap \Omega_{\gamma^*, \tau^*}^{(I)}({\bf u}_n) :  \text{the \, operator}\, {\bf A}_\infty^-(\ell, j, j' ; \omega , { u}_n(\omega)) \nonumber\\
& \,\, \text{is\,\, not\,\,invertible\,\,or} \,\, \text{\,\,it \,\,is \,\,invertible\,\,and}\, \nonumber\\
& \| { \bf A}_\infty^-(\ell, j, j' ; \omega , { u}_n(\omega))^{- 1}\|_{{\rm Op}(j, j')} > \frac{\langle \ell \rangle^\tau}{2\gamma_n \langle j - j' \rangle}  \Big\}\,, \label{risonanti seconde di melnikov differenza}
\end{align}
\begin{align}
R_{\ell j j'}^+({\bf u}_n) & := \Big\{ \omega \in{\mathcal G}_n  \bigcap \Omega_{\gamma^*, \tau^*}^{(I)}({\bf u}_n) :  \text{the \, operator}\, {\bf A}_\infty^+(\ell, j, j' ; \omega , { u}_n(\omega)) \nonumber\\
&  \text{is\,\, not\,\,invertible\,\,or} \,\, \text{\,\,it \,\,is \,\,invertible\,\,and}\, \nonumber\\
&  \| { \bf A}_\infty^+(\ell, j, j' ; \omega , { u}_n(\omega))^{- 1}\|_{{\rm Op}(j, j')} > \frac{\langle \ell \rangle^\tau}{2\gamma_n \langle j + j' \rangle}  \Big\}\,, \label{risonanti seconde di melnikov somma}
\end{align}
where we recall that by \eqref{bf A infinito - (ell,alpha,beta)}, \eqref{bf A infinito + (ell,alpha,beta)}, \eqref{matrice prime di Melnikov}
\begin{equation}\label{matrice prime Melnikov stime in misura}
{ \bf B}_\infty(\ell, j ; \omega , { u}_n(\omega)) := \omega \cdot \ell {\bf I}_j + {\bf D}_j^\infty(\omega, { u}_n(\omega))\,, \quad \ell \in \Z^\nu\,, \quad j \in \N\,,
\end{equation}
\begin{align}\label{definizione L infinito - seconde Melnikov stime misura}
{\bf A}_\infty^{-}(\ell, j, j'; \omega, { u}_n(\omega)) & :=  \omega \cdot \ell \,\,{\bf I}_{j,  j'} + M_L({\bf D}_j^\infty(\omega, { u}_n(\omega))) \\
& \quad  - M_R({\bf D}_{j'}^\infty(\omega, { u}_n(\omega)))\,, \nonumber
\end{align}
for all $\ell \in \Z^\nu$, $j, j' \in \N$, $(\ell, j, j' ) \neq (0, j, j)$ and 
\begin{align}\label{definizione L infinito + seconde Melnikov stime misura}
{\bf A}_\infty^{+}(\ell, j, j'; \omega, {u}_n(\omega)) & :=  \omega \cdot \ell \,\,{\bf I}_{j,  j'} + M_L({\bf D}_j^\infty(\omega, {u}_n(\omega)))  \\
& \qquad + M_R(\overline{{\bf D}_{j'}^\infty(\omega, { u}_n(\omega))})\, \nonumber
\end{align}
for $\ell \in \Z^\nu$, $j, j'  \in \N$.

\noindent
First we need to establish several auxiliary Lemmas. 
\begin{lemma}
For all $n \geq 1$, 
\begin{equation}\label{marco}
\sup_{j \in \N} \Big\|  \widehat{\bf D}_j^\infty({ u}_n) - \widehat{\bf D}_j^\infty({ u}_{n - 1}) \Big\|
\lessdot \e  N_{n - 1}^{-\mathtt a} \, , \quad \forall \omega \in {\mathcal G}_n \, , 
\end{equation}
where $\widehat{\bf D}_j^\infty({ u}_n) = \widehat{\bf D}_j^\infty(\omega,{ u}_n(\omega))$ is given in \eqref{definizione bf D infinito} and $ \mathtt a $ is defined in \eqref{definizione alpha}.
\end{lemma}
\begin{proof}
We first apply Theorem \ref{thm:abstract linear reducibility}-${\bf (S4)_{\nu}}$ 
with $ \nu = n  $, $ \g = \g_{n-1} $, $ \g - \rho = \g_n $, and $ {\bf u}_{1} $, $ {\bf u}_2 $, replaced, respectively,   
by $ {\bf u}_{n-1} $, $ {\bf u}_n $, 
 in order to conclude that
\be\label{primoste}
\Omega_{n}^{\g_{n-1}} ( {\bf u}_{n-1}) \subseteq \Omega_{n}^{\g_n} ( {\bf u}_n ) \, . 
\ee
The smallness condition in \eqref{legno} is satisfied because 
$ \overline \sigma + \mathtt b <  \mu_1$ (see \eqref{costante finale mu 1 linearizzato}) and so
\begin{align*}
&  K_1 N_{n - 1}^{\tau} \e \Vert { u}_n - { u}_{n - 1} \Vert_{s_0 + \overline \sigma + \mathtt b} \leq 
 K_1 N_{n - 1}^{\tau} \e \Vert { u}_n - { u}_{n - 1} \Vert_{ s_0 + \mu_1} \\
 & \leq  K_1 N_{n - 1}^{\tau} \e \Vert { \bf u}_n - { \bf u}_{n - 1} \Vert_{ s_0 + \mu_1} \stackrel{\eqref{hn}}{\leq} 
C_*^{'} K_1 \e^2 \gamma^{- 1} N_{n - 1}^{\tau} N_{n}^{- \mathtt a_1}   
  \\
  & \leq \gamma_{n-1} - \gamma_{n} =: \rho = \g 2^{-n } 
\end{align*}
for $\e \gamma^{-1}$ small enough, $N_0$ large enough and using that $ \mathtt a_1  > \t $ (see \eqref{costanti nash moser}, \eqref{costanti nash moser 1}, \eqref{costante finale mu 1 linearizzato}, \eqref{definizione alpha}).   
Then, by the definitions \eqref{G-n+1}, \eqref{prime e seconde di Melnikov}, \eqref{Omegainfty}, we have 
$$
{\mathcal G}_{n} \subseteq {\mathcal G}_{n-1} \cap \Omega_{\infty}^{2 \g_{n-1}} ({\bf u}_{n-1}) 
\stackrel{\eqref{cantorinclu}} \subseteq \bigcap_{\nu \geq 0} \Omega_{\nu}^{\gamma_{n - 1}}({\bf u}_{n - 1}) 
\subset \Omega_{n}^{\g_{n-1}}({\bf u}_{n-1})
\stackrel{\eqref{primoste}} \subseteq \Omega_{n}^{\g_n}({\bf u}_n).
$$
Next, for all $ \omega \in {\mathcal G}_n  \subset \Omega_{n}^{\g_{n-1}}({\bf u}_{n-1}) \cap 
 \Omega_{n}^{\g_n}({\bf u}_n)  $ both the operators $\widehat{\bf D}_j^n({ u}_{n - 1})$ and $\widehat{\bf D}_j^n({ u}_{n}) $ are well defined and applying the estimate \eqref{Delta12 rj} 
 with $ \nu = n $, we deduce that 
\begin{align}
\sup_{j \in \N} \Big\| \widehat{\bf D}_j^n({ u}_{n})  -  \widehat{\bf D}_j^n({ u}_{n - 1})\Big\|&  \stackrel{\eqref{Delta12 rj}} 
\lessdot  \e  \| { u}_{n} - { u}_{n - 1} \|_{s_0 + \overline \sigma + \mathtt b}  \nonumber\\
& \stackrel{\eqref{costante finale mu 1 linearizzato}}{\lessdot} \e \| {\bf u}_{n} - {\bf u}_{n - 1} \|_{s_0 + \mu_1} \stackrel{\eqref{hn}}{\lessdot} \e N_n^{- \mathtt a_1}\, . \label{vicin+1}
\end{align}
Moreover by \eqref{definizione bf D infinito}, \eqref{autovcon} (with $ \nu = n $), for all $j \in \N$, we get 
\begin{align}\label{diffrkn}
 \Big\|\widehat{\bf D}_j^\infty({ u}_{n - 1}) -\widehat{\bf D}_j^n({ u}_{n - 1})  \Big\| + \Big\| \widehat{\bf D}_j^\infty({u}_{n }) -\widehat{\bf D}_j^n({ u}_{n}) \Big\|  \lessdot 
\e  N_{n - 1}^{- \mathtt a} \,.
\end{align}
Therefore, for all $\omega \in {\mathcal G}_{n}$,  $ \forall j \in \N$,   
\begin{align}
\Big\| \widehat{\bf D}_j^\infty({ u}_n) - \widehat{\bf D}_j^\infty({ u}_{n - 1}) \Big\|
& \leq 
\Big\| \widehat{\bf D}_j^n({ u}_{n})  -  \widehat{\bf D}_j^n({ u}_{n - 1})\Big\|
+ \Big\|\widehat{\bf D}_j^\infty({ u}_{n - 1}) -\widehat{\bf D}_j^n({ u}_{n - 1})  \Big\| \nonumber\\
& \qquad + \Big\| \widehat{\bf D}_j^\infty({ u}_{n }) -\widehat{\bf D}_j^n({ u}_{n}) \Big\|  \nonumber\\
& \stackrel{\eqref{vicin+1}, \eqref{diffrkn}} \lessdot
 \e (N_{n - 1}^{- \mathtt a} + N_n^{- \mathtt a_1}) \lessdot \e N_{n - 1}^{- \mathtt a}\,,
  \label{variazione autovalori finali in u}
\end{align}
since by \eqref{costanti nash moser}, \eqref{costanti nash moser 1} one has $\mathtt a_1 > \mu_1$, and by \eqref{costante finale mu 1 linearizzato} $\mu_1 \geq \overline \sigma + \mathtt b \stackrel{\eqref{definizione alpha}}{\geq} \mathtt a$. Then the claimed estimate is proved. 
\end{proof}
\begin{lemma}\label{risonanti-1}
For  $ \e  \g^{- 1}$ small enough,  
for all $n \geq 1$, $|\ell|\leq N_{n - 1}$, 
\begin{equation}\label{inclusione-1}
Q_{\ell j}({\bf u}_n) \subseteq Q_{\ell j}({\bf u}_{n - 1})\,, \qquad R_{\ell j j'}^{\pm}({\bf u}_{n}) \subseteq R_{\ell j j'}^{\pm}({\bf u}_{n - 1}).
\end{equation}
\end{lemma}

\begin{proof}
We prove that $R_{\ell j j'}^{-}({\bf u}_{n}) \subseteq R_{\ell j j'}^{-}({\bf u}_{n - 1})$. The proof of the other inclusions is analogous. 
For all $ j, j' \in \N$, $|\ell| \leq N_{n - 1}$, $(\ell, j, j') \neq (0, j, j)$, $\omega \in \mG_n$, we write 
\begin{align*}
{\bf A}^-_\infty(\ell, j, j'; { u}_n) & = {\bf A}_\infty^-(\ell,j, j'; { u}_{n - 1}) \Big( {\bf I}_{j,  j'} + {\bf A}_\infty^-(\ell, j, j'; { u}_{n - 1})^{- 1} \Delta_\infty(j, j', n) \Big)\,,
\end{align*}
where 
\begin{align}\label{iron maiden 9}
\Delta_\infty(j, j', n) & := M_L \Big({\bf D}_j^\infty( { u}_n) - {\bf D}_j^\infty({ u}_{n - 1})\Big) \\
&  - M_R \Big( {\bf D}_{j'}^\infty( { u}_n) - {\bf D}_{j'}^\infty({ u}_{n - 1}) \Big)\,. \nonumber
\end{align}
Note that 
\begin{align}
& \Delta_\infty(j, j', n)  \stackrel{\eqref{mu-j-nu}, \eqref{definizione bf D infinito}}{=} \big(m({ u}_n) - m({ u}_{n - 1}) \big) (j - j') {\bf I}_{j,  j'} \nonumber\\
& \qquad + M_L \Big(\widehat{\bf D}_j^\infty( { u}_n) - \widehat{\bf D}_j^\infty({ u}_{n - 1})\Big) - M_R \Big( \widehat{\bf D}_{j'}^\infty( { u}_n) - \widehat{\bf D}_{j'}^\infty({ u}_{n - 1}) \Big)\,. \label{principe filiberto 0}
\end{align}
By \eqref{principe filiberto 0}, \eqref{stime m}, \eqref{norma operatoriale ML MR}, \eqref{marco} one gets 
\begin{align}
\| \Delta_\infty(j, j', n) \|_{{\rm Op}(j, j')} & \lessdot  \e \langle j - j' \rangle N_{n - 1}^{- \mathtt a}\,.\label{iron maiden 10}
\end{align}
\noindent
Hence for $|\ell| \leq N_{n - 1}$ we get  
\begin{align*}
\Big\| {\bf A}_\infty^-(\ell, j, j'; {u}_{n - 1})^{- 1} \Delta_\infty(j, j', n) \Big\|_{{\rm Op}(j, j')} & \leq \frac{\langle \ell \rangle^\tau}{2 \gamma_{n - 1} \langle j - j' \rangle} \| \Delta_\infty(j, j', n) \|_{{\rm Op}(j, j')} \\
& \stackrel{\eqref{iron maiden 10}}{\lessdot}  \e \gamma^{- 1}  N_{n - 1}^{\tau - \mathtt a} \leq \frac{1}{2}
\end{align*}
for $\e \gamma^{- 1}$ small enough and since $\mathtt a > \tau$ (see \eqref{definizione alpha}). Then for all $|\ell| \leq N_{n - 1}$, for all $j, j' \in \N$, the operator ${\bf A}_\infty^-(\ell,j, j'; {u}_n)$ is invertible by Neumann series and 
\begin{align*}
\| {\bf A}_\infty^-(\ell, j, j'; { u}_n)^{- 1} \|_{{\rm Op}(j, j')} & \leq  \| {\bf A}_\infty^-(\ell, j, j'; { u}_{n - 1})^{- 1} \|_{{\rm Op}(j, j')} \Big(1 + C \e \gamma^{- 1} N_{n - 1}^{\tau - \mathtt a} \Big) \nonumber\\
& \leq \frac{\langle \ell \rangle^\tau}{2 \gamma_{n - 1}\langle j - j' \rangle} \Big( 1 + C \e \gamma^{- 1} N_{n - 1}^{\tau - \mathtt a} \Big)\,.
\end{align*}
Since by the definition of $\gamma_n$, 
$$
\frac{\gamma_{n - 1} - \gamma_n}{\gamma_n} = \frac{1}{1 + 2^n}\,,
$$
it follows that for $\e \gamma^{- 1}$ sufficiently small  
$$
C \e \gamma^{- 1} N_{n - 1}^{\tau - \mathtt a} \leq \frac{\gamma_{n - 1} - \gamma_n}{\gamma_n}\,.
$$
Hence
$$
\| {\bf A}_\infty^-(\ell, j, j'; {u}_n)^{- 1} \|_{{\rm Op}(j, j')} \leq \frac{\langle \ell \rangle^\tau}{2 \gamma_{n}\langle j - j' \rangle}
$$
which implies the claimed inclusion  $R_{\ell j j'}^{-}({\bf u}_{n}) \subseteq R_{\ell j j'}^{-}({\bf u}_{n - 1})$ in \eqref{inclusione-1}. 
\end{proof}

\begin{corollary}\label{excisione parametri io}
For any $n \geq 1$
\noindent
$(i)$ $Q_{\ell j}({\bf u}_n) = \emptyset$, for all $|\ell| \leq N_{n - 1}$, 

\noindent
$(ii)$ $R_{\ell j j' }^-({\bf u}_n) = \emptyset$, for all $|\ell| \leq N_{n - 1}$, $(\ell,j, j') \neq (0, j, j)$,

\noindent
$(iii)$ $R_{\ell j j'}^+({\bf u}_n) = \emptyset$, for all $|\ell| \leq N_{n - 1}$.

\noindent
Hence for any $n \geq 1$,
\be\label{parametri cattivi}
{\mathcal A}_n^{(1)}  \stackrel{\eqref{espansione risonanti}}  
= \bigcup_{\begin{subarray}{c}
|\ell| > N_{n - 1} \\
j \in \N
\end{subarray}} Q_{\ell j}({\bf u}_n) \bigcup \bigcup_{\begin{subarray}{c}
|\ell| > N_{n - 1}  \\
j, j' \in\N \\
 (\ell, j, j') \neq (0, j, j)
\end{subarray}} R_{\ell j j'}^{-}({\bf u}_n) \bigcup \bigcup_{\begin{subarray}{c}
|\ell| > N_{n - 1}  \\
j, j' \in \N
\end{subarray}}R_{\ell j j'}^{+}({\bf u}_n)\,.   
\ee
\end{corollary}
\begin{proof}
By definition, $ R_{\ell j j'}^{\pm}({\bf u}_n), Q_{\ell j}({\bf u}_n) \subset {\mathcal G}_n $ and, by \eqref{inclusione-1},
for all $ |\ell| \leq N_{n - 1} $, we have 
$ R_{\ell j j'}^{\pm} ({\bf u}_n) \subseteq R_{\ell j j'}^{\pm} ({\bf u}_{n-1}) $ and $Q_{\ell j}({\bf u}_n) \subseteq Q_{\ell j}({\bf u}_{n - 1})$.
On the other hand again by definition $ R_{\ell j j'}^\pm({\bf u}_{n-1}) \cap {\mathcal G}_n, Q_{\ell j}({\bf u}_{n - 1}) \cap {\mathcal G}_n = \emptyset $. 
As a consequence, $ \forall |\ell| \leq N_{n - 1} $, 
$ R_{\ell j j'}^{\pm} ({\bf u}_n)\,,\, Q_{\ell j}({\bf u}_n) = \emptyset $. 
\end{proof}
\begin{lemma}\label{limitazioni indici risonanti}
For all $n \geq 0$ the following statements hold: 

\noindent
$(i)$ If $Q_{\ell j }({\bf u}_n) \neq \emptyset$, then $\ell \neq 0$ and $j \lessdot |\ell |$.

\noindent
$(ii)$ If $R_{\ell j j'}^{-}({\bf u}_n) \neq \emptyset$, then $\ell \neq 0$ and  $|j - j'| \lessdot |\ell|$, $j, j' \lessdot | \ell |^{\tau^*}$. 

\noindent
$(iii)$ If $R_{\ell j j'}^{+}({\bf u}_n) \neq \emptyset$, then $\ell \neq 0$ and $j, j' \lessdot  |\ell |$.

\end{lemma}
\begin{proof}
We prove item $(ii)$. The proofs of items $(i)$ and $(iii)$ are similar. The statement follows by the following claim: 
\begin{itemize}
\item{\sc Claim} If $\e \gamma^{- 1}$ is small enough and  
\begin{equation}\label{piccolezza pippo misura}
\langle \ell \rangle^{\tau^*} \leq \langle j - j' \rangle  \min\{ j, j' \} 
\end{equation}
then for all $\omega \in {\mathcal G}_n \cap \Omega_{\gamma^*, \tau^*}^{(I)}({\bf u}_n)$ (recall \eqref{prime di Melnikov}), the matrix ${\bf A}_\infty^-(\ell, j, j') = {\bf A}_\infty^-(\ell, j, j'; \omega, {u}_n(\omega))$ is invertible and 
$$
\| {\bf A}_\infty^-(\ell, j, j')^{- 1} \|_{{\rm Op}(j, j')} \leq \frac{\langle \ell \rangle^{\tau}}{2 \gamma_n \langle j - j' \rangle}\,.
$$
\end{itemize}
\noindent{\sc Proof of the claim.} By \eqref{bf A infinito - (ell,alpha,beta)}, \eqref{definizione bf D infinito}, \eqref{mu-j-nu}, \eqref{rappresentazione a blocchi cal D0 (1)}, we can write 
\begin{equation}\label{barbaria}
{\bf A}_\infty^-(\ell, j, j') = {\bf I}_\infty(\ell, j, j') + \Delta_\infty(j, j')\,,
\end{equation}
where 
$$
{\bf I}_\infty(\ell, j, j') :=\Big( \omega \cdot \ell  + m (j - j') \Big) {\bf I}_{j,  j'}\,, \quad \Delta_\infty(j, j') :=  M_L( \widehat{\bf D}_j^\infty) - M_R(\widehat{\bf D}_{j'}^\infty)\,.
$$
Since $\omega \in  \Omega_{\gamma^*, \tau^*}^{(I)}({\bf u}_n)$, for any $(\ell, j, j') \neq (0, j, j)$, the operator ${\bf I}_\infty(\ell, j, j')$ is invertible and its inverse satisfies the bound 
\begin{equation}\label{bound bf I ell alpha beta - 1}
\| {\bf I}_\infty(\ell, j, j')^{- 1} \|_{{\rm Op}(j, j')} \leq \frac{\langle \ell \rangle^{\tau^*}}{ \gamma^*_n \langle j - j' \rangle}\,.
\end{equation}
Moreover the operatorial norm of the operator $\Delta_\infty(j, j')$ satisfies
\begin{align}
\| \Delta_\infty(j, j')\|_{{\rm Op}(j, j')} & \stackrel{\eqref{norma operatoriale ML MR}, \eqref{autovcon}}{\lessdot}  \e \Big( \frac{1}{j} + \frac{1}{j'} \Big) \lessdot \frac{\e}{\min\{ j, j' \}}\,.  \label{stima Delta infty alpha beta}
\end{align}
The estimates \eqref{bound bf I ell alpha beta - 1}, \eqref{stima Delta infty alpha beta} imply that 
$$
\| {\bf I}_\infty(\ell, j, j')^{- 1} \Delta_\infty(j, j')  \|_{{\rm Op}(j, j')} \lessdot \frac{\e \langle \ell \rangle^{\tau^* }}{\gamma^*_n \langle j - j' \rangle \min\{ j, j' \}} \stackrel{}{\leq} \frac12\,,
$$
by \eqref{piccolezza pippo misura} and for $\e \gamma^{- 1}$ small enough.
Hence by \eqref{barbaria}, the matrix ${\bf A}_\infty^-(\ell, j, j')$ is invertible by Neumann series and 
$$
\| {\bf A}_\infty^-(\ell, j, j')^{- 1} \|_{{\rm Op}(j, j')} \leq 2 \| {\bf I}_\infty(\ell,j, j')^{- 1} \|_{{\rm Op}(j, j')} \stackrel{\eqref{bound bf I ell alpha beta - 1}}{\leq} \frac{2 \langle \ell \rangle^{\tau^*}}{ \gamma^*_n \langle j - j' \rangle} {\leq} \frac{\langle \ell \rangle^{\tau}}{ 2 \gamma_n \langle j - j' \rangle}\,,
$$
since by \eqref{definizione gamma* tau*}, \eqref{gamma n stime misura} we have $\tau < \tau^*$ and $\gamma^*_n > 4 \gamma_n$. 

\medskip

\noindent
By the definition \eqref{risonanti seconde di melnikov differenza}, the claim implies that if $\omega \in {\mathcal G}_n \cap \Omega_{\gamma^*, \tau^*}^{(I)}({\bf u}_n)$ and if the condition \eqref{piccolezza pippo misura} holds, then $\omega \notin R_{\ell j j'}^{-}({\bf u}_n)$, hence if $R_{\ell j j'}^{-}({\bf u}_n) \neq \emptyset$, then 
\begin{equation}\label{pippa mentale}
\langle j - j' \rangle  \min\{ j, j' \}  < \langle \ell \rangle^{\tau^*}\,.
\end{equation}
For $\ell = 0$, since $\langle \ell \rangle = {\rm max}\{ |\ell|, 1\} = 1$, the above condition becomes $\langle j - j' \rangle  \min\{ j, j' \}  < 1$ which is violated since $\langle j - j' \rangle  \min\{ j, j' \} = {\rm max}\{ |j - j'|, 1\}\min\{ j, j' \} \geq 1$, therefore $R_{0 j j'}({\bf u}_n) = \emptyset$ for any $j \neq j'$. Finally, by \eqref{pippa mentale}, we may easily deduce that 
$$
j, j' \lessdot \langle \ell \rangle^{\tau^*}\,.
$$
 By similar arguments, it can be proved that if $R_{\ell j j'}^-({\bf u}_n) \neq \emptyset$, then $|j - j'| \lessdot \langle \ell \rangle$ and the proof is concluded.  
 \end{proof}
Combining Corollary \ref{excisione parametri io} and Lemma \ref{limitazioni indici risonanti}, recalling the formulae \eqref{espansione risonanti}, \eqref{parametri cattivi}, we get 
\be\label{parametri cattivi forma finale 0}
{\mathcal A}_0^{(1)}  
= \bigcup_{\begin{subarray}{c}
 \ell \neq 0\\
 j \in \N \\
j \lessdot  |\ell|
\end{subarray}} Q_{\ell j}({\bf u}_0) \bigcup \bigcup_{\begin{subarray}{c}
\ell \neq 0  \\
j, j'  \in \N \\
 (\ell, j, j') \neq (0, j, j) \\
 |j - j'| \lessdot |\ell| \\
 j, j' \lessdot  \langle \ell \rangle^{\tau^*}
\end{subarray}} R_{\ell j j'}^{-}({\bf u}_0) \bigcup \bigcup_{\begin{subarray}{c}
\ell \neq 0  \\
j, j' \in \N\\
j, j' \lessdot |\ell |
\end{subarray}}R_{\ell j j'}^{+}({\bf u}_0)\, , 
\quad    
\ee
 \be\label{parametri cattivi forma finale}
{\mathcal A}_n^{(1)} 
= \bigcup_{\begin{subarray}{c}
|\ell| > N_{n - 1} \\
 j \in \N \\
j \lessdot  \langle \ell \rangle
\end{subarray}} Q_{\ell j}({\bf u}_n) \bigcup \bigcup_{\begin{subarray}{c}
|\ell| > N_{n - 1} \\
\ell \neq 0  \\
j, j' \in \N \\
 (\ell, j, j') \neq (0, j, j) \\
j, j' \lessdot  \langle \ell \rangle^{\tau^*}
\end{subarray}} R_{\ell j j'}^{-}({\bf u}_n) \bigcup \bigcup_{\begin{subarray}{c}
|\ell| > N_{n - 1} \\
\ell \neq 0  \\
j, j' \in \N\\
j, j' \lessdot \langle \ell \rangle
\end{subarray}}R_{\ell j j'}^{+}({\bf u}_n)
\quad    
\ee
$\forall n \geq 1$.
The measure of the resonant sets on the right hand side of the latter identities now are estimated separately:  
\begin{lemma}\label{stima in misura insiemi risonanti}
For $\e \gamma^{-1}$ small enough, if $Q_{\ell j}({\bf u}_n)\,,\,R_{\ell j j'}^\pm({\bf u}_n) \neq \emptyset$, then 
$$
|Q_{\ell j}({\bf u}_n)|\,,\,|R_{\ell j j'}^\pm({\bf u}_n)| \lessdot \gamma  \langle \ell \rangle^{- \tau}\,.
$$
\end{lemma}
 \begin{proof}
We prove the estimate of the set $R_{\ell j j'}^{-}({\bf u}_n)$. The other estimates can be proven arguing similarly. Recall that for all $j \in \N$, the $2 \times 2$ blocks ${\bf D}_j^\infty = m \, j\,{\bf I}_j + \widehat{\bf D}_j^\infty \in {\mathcal S}({\bf E}_j)$, are self-adjoint and Lipschitz continuous with respect to the parameter $\omega$. We set 
 \begin{equation}\label{spettro correzione cappuccio}
 {\rm spec}(\widehat{\bf D}_j^\infty(\omega)) := \big\{ r_1^{(j)}(\omega), r_2^{(j)}(\omega) \big\} \quad \text{with} \quad r_1^{(j)}(\omega)\leq r_2^{(j)}(\omega)\,,
 \end{equation}
  By Lemma \ref{risultato astratto operatori autoaggiunti}-$(i)$ the functions $\omega \mapsto r_k^{(j)}(\omega)$ are Lipschitz with respect to $\omega$, since 
 \begin{align}
 |r_k^{(j)}(\omega_1) - r_k^{(j)}(\omega_2)| &\leq \| \widehat{\bf D}_j^\infty(\omega_1) -  \widehat{\bf D}_j^\infty(\omega_2)  \|  \leq \| \widehat{\bf D}_j^\infty \|^{\rm lip} |\omega_1- \omega_2|  \nonumber\\
 & \stackrel{\eqref{autovcon}}{\lessdot} \e \gamma^{- 1} j^{- 1} |\omega_1- \omega_2|\,. \label{dino zoff}
 \end{align}
 Setting also
 $$
{\rm spec}({\bf D}_j^\infty(\omega)) := \big\{ \lambda_1^{(j)}(\omega), \lambda_2^{(j)}(\omega) \big\} \quad \text{with} \quad \lambda_1^{(j)}(\omega)\leq \lambda_2^{(j)}(\omega)\,,
$$
by Lemma \ref{risultato astratto operatori autoaggiunti}-$(ii)$ we have that 
\begin{equation}\label{espansione asintotica autovalori}
\lambda_k^{(j)}(\omega) =m(\omega) \, j+ r_k^{(j)}(\omega)\,, \qquad  k = 1, 2\,.
\end{equation}
By the definition \eqref{bf A infinito - (ell,alpha,beta)} and by Lemmata \ref{properties operators matrices}, \ref{risultato astratto operatori autoaggiunti}-$(ii)$ the operator ${\bf A}_\infty^{-}(\ell, j, j') : {\mathcal L}({\bf E}_{j'}, {\bf E}_j) \to {\mathcal L}({\bf E}_{j'}, {\bf E}_j) $ is self-adjoint with respect to the scalar product \eqref{prodotto scalare traccia matrici} and 
$$
{\rm spec}\Big({\bf A}_\infty^{-}(\ell, j, j'; \omega)\Big) = \Big\{ \omega \cdot \ell + \lambda_k^{(j)}(\omega) - \lambda_{k'}^{(j')}(\omega)\,, \quad k, k' =1, 2  \Big\}\,.
$$
Therefore, recalling the definition \eqref{risonanti seconde di melnikov differenza} and by Lemma \ref{risultato astratto operatori autoaggiunti}-$(iii)$ we get 
\begin{equation}\label{decomposizione risonante 2 autovalori}
R_{\ell j j'}^{-}({\bf u}_n) \subseteq \bigcup_{k, k' = 1}^{2}  \widetilde R_{\ell j j'}(k, k')\,,
\end{equation}
where 
$$
\widetilde R_{\ell j, j'}(k, k') := \Big\{ \omega  : |\omega \cdot \ell + \lambda_k^{(j)}(\omega) - \lambda_{k'}^{(j')}(\omega)| < \frac{2 \gamma_n \langle j - j'\rangle}{\langle \ell \rangle^\tau}\Big\}\,.
$$
We estimate the measure of the set $\widetilde R_{\ell j j'}(k, k')$ defined above for all $k, k' = 1, 2$.  
Since, by Lemma \ref{limitazioni indici risonanti}-$(ii)$, $\ell \neq 0$, we can write 
$$
\omega = \frac{\ell}{|\ell|} s + v\,, \quad \text{with} \quad v \cdot \ell = 0\,.
$$
and we define 
\begin{equation}\label{phi (s) importante}
\phi(s ) := |\ell| s + \lambda_k^{(j)}(s) - \lambda_{k'}^{(j')}(s)\,, 
\end{equation}
where for any $j \in \N$, for all $k = 1,2$
$$
\lambda_k^{(j)}(s) := \lambda^{(j)}_k\Big( \frac{\ell}{|\ell|} s + v \Big)\,.
$$
According to \eqref{espansione asintotica autovalori}, \eqref{dino zoff}
\begin{equation}\label{genesis 0}
\lambda_k^{(j)}(s) =  m (s)\, j + r_k^{(j)}(s)\,,\qquad |r_k^{(j)}|^{\rm lip} \lessdot \e \gamma^{- 1} j^{- 1}\,.
\end{equation}
One gets 
\begin{align}
|\phi(s_1) - \phi(s_2)| & \geq |\ell||s_1 - s_2| - |m(s_1) - m(s_2) | |j - j'|   \nonumber\\
& \quad - \big(|r_k^{(j)}|^{\rm lip} + |r_{k'}^{(j')}|^{\rm lip} \big) |s_1 - s_2| \nonumber\\
& \stackrel{\eqref{stime m}, \eqref{genesis 0}}{\geq} \Big( |\ell| -  \e \gamma^{- 1}\langle j - j' \rangle \Big) |s_1 - s_2|  \nonumber\\
& \quad \stackrel{Lemma\,\ref{limitazioni indici risonanti}-(ii)}{\geq} \frac{|\ell|}{2} |s_1 - s_2|\, \nonumber
\end{align}
for $\e \gamma^{- 1}$ small enough.
The above estimate implies that  
$$
\Big| \Big\{ s : \frac{\ell}{|\ell|}s + v \in \widetilde R_{\ell j j'}(k, k') \Big\} \Big| \lessdot \frac{4 \gamma \langle j - j'\rangle}{\langle \ell \rangle^{\tau + 1}}
$$
and by Fubini Theorem we get 
$$
|\widetilde R_{\ell j j'}(k, k')| \lessdot \frac{4 \gamma \langle j - j'\rangle}{\langle \ell \rangle^{\tau + 1}}\,.
$$
The claimed estimate follows by recalling \eqref{decomposizione risonante 2 autovalori} and using that by Lemma \ref{limitazioni indici risonanti}-$(ii)$, $|j - j'| \lessdot |\ell|$. 
 \end{proof}

\noindent
\begin{lemma}\label{stima cal A n (1)}
For all $n \geq 0$, we get 
$$
|{\mathcal A}_n^{(1)}| \lessdot \gamma N_{n - 1}^{- 1}\,.
$$
\end{lemma}
\begin{proof}
We prove the estimate for ${\mathcal A}_n^{(1)}$ in \eqref{parametri cattivi forma finale} with $n \geq 1$. The estimate for ${\mathcal A}_0^{(1)}$ in \eqref{parametri cattivi forma finale 0} follows similarly. One has 
\begin{align}
|{\mathcal A}_n^{(1)}| & \lessdot \sum_{\begin{subarray}{c}
|\ell| > N_{n - 1} \\
 j \in \N \\
 j \lessdot  \langle \ell \rangle
\end{subarray}} |Q_{\ell j}({\bf u}_n)| + \sum_{\begin{subarray}{c}
|\ell| > N_{n - 1} \\
 j, j' \in \N \\
 (\ell, j, j') \neq (0, j, j) \\
 |j - j'| \lessdot \langle \ell \rangle \\
 j, j' \lessdot  \langle \ell \rangle^{\tau^*}
\end{subarray}} |R_{\ell j j'}^{-}({\bf u}_n)| + \sum_{\begin{subarray}{c}
|\ell| > N_{n - 1} \\
j, j' \in \N\\
j, j' \lessdot \langle \ell \rangle
\end{subarray}} |R_{\ell j j'}^{+}({\bf u}_n) | \nonumber\\
& \stackrel{Lemma\,\,\ref{stima in misura insiemi risonanti}}{\lessdot}  \gamma \Big( \sum_{\begin{subarray}{c}
|\ell| > N_{n - 1} \\
j \in \N \\
 j \lessdot  \langle \ell \rangle
\end{subarray}} \frac{1}{\langle \ell \rangle^{\tau}}  + \sum_{\begin{subarray}{c}
|\ell| > N_{n - 1} \\
j, j' \in \N \\
 j, j' \lessdot  \langle \ell \rangle^{\tau^*}
\end{subarray}} \frac{1}{ \langle \ell \rangle^{\tau}}+ \sum_{\begin{subarray}{c}
|\ell| > N_{n - 1}  \\
j, j' \in \N\\
j, j'\lessdot \langle \ell \rangle
\end{subarray}} \frac{1}{ \langle\ell \rangle^{ \tau }}  \Big) \nonumber\\ 
& \lessdot \gamma \sum_{\begin{subarray}{c}
|\ell| > N_{n - 1} 
\end{subarray}} \Big(\frac{1}{\langle \ell \rangle^{\tau  - 1 }}  +  \frac{1}{ \langle \ell \rangle^{\tau - 2 \tau^* }}+  \frac{1}{ \langle\ell \rangle^{ \tau - 2 }}  \Big) \stackrel{\eqref{valori finali tau tau*}}{\lessdot} \gamma N_{n - 1}^{- 1}
\end{align}
which is the claimed estimate. 
\end{proof}

\noindent
{\bf Estimate of ${\mathcal A}_n^{(2)}$.} By the same arguments used to prove Lemma \ref{stima cal A n (1)}, one may deduce that the sets ${\mathcal A}_n^{(2)}$ defined in \eqref{cal An (2)} satisfy the lemma below. 
\begin{lemma}\label{lemma stima cal An (2)}
$$
 |{\mathcal A}_n^{(2)}| \lessdot \gamma N_{n- 1}^{- 1}\,, \qquad \forall n \geq 0\,.
$$
\end{lemma}

\noindent
{\bf Proof of Theorem \ref{stima finale parametri cattivi}}. The theorem follows by \eqref{splitting Omega - G infty 1}, \eqref{pupella 0}, \eqref{pappa pappa}, by Lemmata \ref{lemma stima cal An (2)}, \ref{stima cal A n (1)}, using that the series $\sum_{n \geq 0} N_{n - 1}^{- 1} < + \infty$, since $N_n = N_0^{\chi^n}$, $N_{- 1} := 1$ with $N_0 > 1$. 

\section{ Proof of the main Theorems concluded}\label{sezione conclusiva dim}
{\sc Proof of Theorem \ref{main theorem kirchoff 1}.}
 Set $\gamma = \e^a$, with $0 < a < 1$. Then $ \e \gamma^{- 1} =  \e^{1 - a} $ and hence the smallness condition \eqref{nash moser smallness condition} is fullfilled by taking $\e$ small enough. By the estimate \eqref{hn} we deduce that the sequence $({\bf u}_n)_{n \geq 1}$ is a Cauchy sequence with respect to the norm $\| \cdot \|_{s_0 + \mu_1}^\Lipg$ on the set ${\mathcal G}_\infty$ defined in \eqref{insieme di Cantor cal G infinito},
 then  it converges to a limit ${\bf u}_\infty$, which satisfies the estimate 
\begin{equation}\label{stime bf u infinito}
\| {\bf u}_\infty \|_{s_0 + \mu_1}^\Lipg  \lessdot \e \gamma^{- 1} \lessdot \e^{1 - a} \stackrel{\e \to 0}{\to}0\,.
\end{equation}
Moreover, by Theorem \ref{iterazione-non-lineare}-$({\mathcal P}2)_n$, we deduce that for all $\omega \in {\mathcal G}_\infty$, ${\mathcal F}(
 {\bf u}_\infty) = 0$ and by Theorem \ref{stima finale parametri cattivi}, since $\gamma = \e^a$, we get 
$$
|\Omega \setminus {\mathcal G}_\infty| \to 0 \qquad \text{as} \qquad \e \to 0\,.
$$

\noindent
{\sc Proof of Theorem \ref{main theorem kirchoff}.} Recalling the splitting \eqref{sistema su media nulla}, \eqref{sistema sul modo 0}, if $(v_0(\vphi), p_0(\vphi)) $ are the solution of the equation \eqref{sistema sul modo 0} found in Lemma \ref{lemma equazione sulle medie nulle}, and by applying Theorem \ref{main theorem kirchoff 1}, we get that for any $\omega \in {\mathcal G}_\infty$ the function $(v_\infty, p_\infty) = (v_0 , p_0) + {\bf u}_\infty$
satisfies $F( v_\infty, p_\infty ) = 0$. Furthermore, choosing $\gamma = \e^a$, with $0 < a < 1/2$, by \eqref{caccoletta media}
$$
\| v_0\|_s \lessdot \e^{1 - 2 a} \stackrel{\e \to 0}{\to} 0\,, \quad \| p_0\|_s \lessdot \e^{1 -  a} \stackrel{\e \to 0}{\to} 0\,,
$$
hence, \eqref{pappoletta} follows by recalling \eqref{stime bf u infinito}. Finally \eqref{condizione media nulla soluzione vphi x} follows since 
$$
\int_{\T^{\nu + 1}} v_\infty(\vphi, x)\, d \vphi\, d x = \int_{\T^\nu} v_0(\vphi)\, d \vphi\,, \quad \int_{\T^{\nu + 1}} p_\infty(\vphi, x)\, d \vphi\, d x = \int_{\T^\nu} p_0(\vphi)\, d \vphi
$$ 
and by Lemma \ref{lemma equazione sulle medie nulle}, $v_0$ and $p_0$ have zero average in $\vphi$, this concludes the proof of Theorem \ref{main theorem kirchoff}.
\subsection{Linear stability}\label{proof linear stability}
In this section we prove Theorem \ref{theorem linear stability}. The linearized equation on a quasi-periodic Sobolev function ${\bf v}(\omega t, x) = (v(\omega t, x), p(\omega t, x))$, with $v, p \in H^S(\T^{\nu + 1}, \R)$ has the form \eqref{equazione linearizzata}. Since the linearized vector field 
$$
L(\omega t) := \begin{pmatrix}
0 & 1 \\
a(\omega t) \partial_{xx} + {\mathcal R}(\omega t) &  0 
\end{pmatrix}
$$
(recall \eqref{coefficienti equazione linearizzata1}, \eqref{coefficienti equazione linearizzata2}) preserves the space of the functions with zero average in $x$, the equation \eqref{equazione linearizzata} can be splitted into the two systems
\begin{equation}\label{papapapa}
\quad \begin{cases}
\partial_t \widehat v_0 = \widehat p_0 \\
\partial_t \widehat p_0 = 0
\end{cases}
\end{equation}
\begin{equation}\label{blablabla}
\begin{cases}
\partial_t \widehat u = \widehat \psi \\
\partial_t \widehat \psi = a(\omega t) \partial_{xx}\widehat u + {\mathcal R}(\omega t)[\widehat u]\,,
\end{cases} 
\end{equation}
where, recalling \eqref{proiettore media}, $\widehat v_0 = \pi_0 \widehat v$, $\widehat p_0 = \pi_0 \widehat p$, $\widehat u = \pi_0^\bot \widehat v$, $\widehat \psi = \pi_0^\bot \widehat p$. By assumption, the initial datum $\widehat p^{(0)}$ has zero average in $x$, hence the solution of the system \eqref{papapapa}  is given by $\widehat v_0(t) = const$, $\widehat p_0(t) = 0$, for all $t \in \R$, implying that the system projected on the zero Fourier mode in $x$ is stable.  It remains to establish the stability for the system projected on the $x$-zero average functions \eqref{blablabla}. 
\noindent
By Lemma \ref{stime Hs x 1}, by \eqref{proprieta coordinate complesse x}, by Lemmata \ref{stime Hs x cal v}, \ref{stime Hs x Phi infty}, using also Lemma \ref{lemma decadimento Kirchoff in x}, there exists $\overline \mu  > 0$ such that for $S > s_0 + \overline \mu$, for any $s_0 \leq s \leq S - \overline \mu$, for any $\omega \in \Omega_\infty^{2 \gamma}$ (see \eqref{Omegainfty}), the linear and continuous maps ${\mathcal T}_1(\omega t) := {\mathcal S}(\omega t)  \circ {\mathcal B}$ and ${\mathcal T}_2(\omega t) := {\mathcal V}(\omega t) \circ \Phi_\infty(\omega t)$ satisfy
$$
{\mathcal T}_1(\omega t) : {\bf H}^{s - \frac12}_0(\T_x) \to H^s_0(\T_x, \R) \times H^{s - 1}_0(\T_x, \R)\,, 
$$
$$
{\mathcal T}_1(\omega t)^{- 1}  : H^s_0(\T_x, \R) \times H^{s - 1}_0(\T_x, \R) \to   {\bf H}^{s - \frac12}_0(\T_x) \,,
$$
$$
{\mathcal T}_2(\omega t)^{\pm 1} : {\bf H}^{s - \frac12}_0(\T_x) \to {\bf H}^{s - \frac12}_0(\T_x)\,.
$$
Setting $A {\bf h} (t, x) = {\bf h}(t + \alpha(\omega t), x)$, $A^{- 1} {\bf h}(\tau + \tilde \alpha(\omega \tau), x) $,
where $\alpha$ and $\tilde \alpha$ are given in Section \ref{sezione diffeo del toro}, by the results of Sections \ref{step 1 riduzione}-\ref{descent method}, by Theorem \ref{teoremadiriducibilita} and using the arguments of Section \ref{formalismo Hamiltoniano}, we get that a curve $\widehat{\bf u}(t) = (\widehat u(t), \widehat \psi(t)) \in H^s_0(\T_x, \R) \times H^{s - 1}_0(\T_x, \R)$ is a solution of the PDE \eqref{blablabla} if and only if 
$$
{\bf h}(t) = \begin{pmatrix}
h(t) \\
\overline h(t)
\end{pmatrix} = {\mathcal T}_2(\omega t)^{- 1} \circ A^{- 1} \circ {\mathcal T}_1(\omega t)^{- 1} \widehat{\bf  u}(t) \in {\bf H}^{s - \frac12}_0(\T_x) 
$$
is a solution of the PDE 
\begin{equation}\label{equazione lineare ridotta}
\begin{cases}
\partial_t h = - \ii{\mathcal D}_\infty^{(1)} h \\
\partial_t \overline h = \ii \overline{\mathcal D}_\infty^{(1)} \overline h
\end{cases}
\end{equation}

where ${\mathcal D}_\infty^{(1)}$ is defined by \eqref{cal D infinito}, \eqref{definizione bf D infinito}\,. Using that ${\mathcal D}_\infty^{(1)}$ is a $2 \times 2$-block diagonal operator, it is straightforward to verify that the commutator $[{\mathcal D}_\infty^{(1)}, |D|^s] = 0$. Furthermore, using the self-adjointness of ${\mathcal D}_\infty^{(1)}$ one sees by a standard energy estimate that  $\partial_t \| h (t, \cdot)\|_{H^s_x}^2 = 0$, implying that 
$$
\| h(t, \cdot)\|_{H^s_x} = const\,, \qquad \forall t \in \R\,,
$$
Arguing as in the proof of Theorem 1.5 in \cite{BBM-Airy} one concludes that $\|\widehat{\bf u}(t, \cdot) \|_{H^s_x \times H^{s - 1}_x} \leq_s \| \widehat{\bf u}(0, \cdot) \|_{H^s_x \times H^{s - 1}_x}$ for all $t \in \R$, which proves the linear stability of \eqref{blablabla} and the proof of Theorem  \ref{theorem linear stability} is concluded. 
\subsection*{Acknowledgment}
The author thanks Michela Procesi for many useful discussions and comments.

\vspace{1.0cm}

\noindent
Riccardo Montalto, 
Institut f\"ur Mathematik, 
Universit\"at Z\"urich, Winterthurerstrasse 190, CH-8057 Z\"urich;\\
${}\qquad$ email: riccardo.montalto@math.uzh.ch

\end{document}